%% file: main.tex
\documentclass[11pt,reqno]{amsart}
\pdfoutput=1

\usepackage[T1]{fontenc}

\usepackage{amsmath}								
\usepackage{amssymb}
\usepackage{amsthm}
\usepackage{amscd}
\usepackage{amsfonts}
\usepackage{mathtools}
\usepackage{mathptmx}
\usepackage[scaled]{helvet}
\usepackage{microtype,comment}

\usepackage[all]{xy}

\usepackage{euler}
\usepackage{extarrows}

\usepackage[colorlinks, citecolor = blue, linkcolor = blue]{hyperref}
\usepackage{color}

\usepackage{tikz}									
\usetikzlibrary{matrix}
\usetikzlibrary{patterns}
\usetikzlibrary{matrix}
\usetikzlibrary{positioning}
\usetikzlibrary{decorations.pathmorphing}
\usetikzlibrary{decorations.pathreplacing}
\usetikzlibrary{cd}
\usetikzlibrary{intersections,calc}
\tikzset{dot/.style={circle, fill=black, inner sep=.05cm}}

\usepackage{dsfont}

\usepackage[left=3.5cm,top=3.5cm,right=3.5cm]{geometry}

\usepackage[shortlabels]{enumitem}

\newtheorem{nntheorem}{Theorem}			
\newtheorem{nndef}[nntheorem]{Definition}

\newtheorem{theorem}{Theorem}[section]
\newtheorem{lemma}[theorem]{Lemma}
\newtheorem{proposition}[theorem]{Proposition}
\newtheorem{corollary}[theorem]{Corollary}

\theoremstyle{definition}
\newtheorem{definition}[theorem]{Definition}
\newtheorem{example}[theorem]{Example}

\newtheorem{remark}[theorem]{Remark}

\newcommand{\R}{\mathds{R}}

\newcommand{\Z}{\mathds{Z}}
\newcommand{\bZ}{\mathds{Z}}
\newcommand{\ZZ}{\mathds{Z}}
\newcommand{\Q}{\mathds{Q}}

\newcommand{\CC}{\mathds{C}}

\newcommand{\PP}{\mathds{P}}

\newcommand{\T}{\mathds{T}}

\newcommand{\mc}{\mathcal}
\newcommand{\mtt}{\mathtt}

\newcommand{\Mbar}{\overline{\calM}}
\newcommand{\mMbar}{\overline{\mathcal M}}

\newcommand{\calB}{\mathcal{B}}
\newcommand{\calC}{\mathcal{C}}

\newcommand{\calM}{\mathcal{M}}
\newcommand{\calO}{\mathcal{O}}

 \newcommand{\CPL}{{\mathrm{CP}}}

\DeclareMathOperator{\Pic}{Pic}

\DeclareMathOperator{\Spec}{Spec}
\DeclareMathOperator{\Hom}{Hom}

\DeclareMathOperator{\Aut}{Aut}

\newcommand\trop{{\mathtt{trop}}}

\DeclareMathOperator{\Trop}{\mathtt{Trop}}

\DeclareMathOperator{\pr}{pr}
\DeclareMathOperator{\id}{id}

\newcommand{\bL}{{\mathds L}}

\newcommand{\TPL}{TPL}

\DeclareMathOperator{\Aff}{Aff}
\newcommand{\PL}{{\mathrm{PL}}}
\newcommand\sPL{{\mathrm{sPL}}}
\newcommand{\PLfin}{\PL^{\mathrm{fin}}}

\newcommand\Mgn[1]{\mathcal M_{#1}}
\newcommand\Mgnbar[1]{\mMbar_{#1}}
\newcommand\Mgnbart[1]{\mMbar^{\mathrm{trop}}_{#1}}

\newcommand\Atlass[1]{{\overline {\mc V}_{#1}}}

\newcommand\GoodAtlas[1]{\mc V^{\mathrm{good}}_{#1}}

\newcommand\br{\mathrm{br}}
\newcommand\src{\mathrm{src}}

\newcommand\tot{{\mathtt{tot}}}
\newcommand\Adm{\overline{Adm}^{\trop}_{1\to 0}((3), (2,1)^4)}
\newcommand\trof{F}
\renewcommand\injlim\varinjlim

\newcommand\TB{{ \mtt B}}
\newcommand\TC{{ \mtt C}}
\newcommand\TU{{ \mtt U}}
\newcommand\TF{{ \mtt F}}
\newcommand\TV{{ \mtt V}}
\newcommand\TX{{ \mtt X}}
\newcommand\TL{{ \mtt L}}

\newcommand\ts{{\mtt s}}
\newcommand\Tpi{{\Pi }}

\newcommand{\mysetminusD}{\hbox{\tikz{\draw[line width=0.6pt,line cap=round] (3pt,0) -- (0,6pt);}}}
\newcommand{\mysetminusT}{\mysetminusD}
\newcommand{\mysetminusS}{\hbox{\tikz{\draw[line width=0.45pt,line cap=round] (2pt,0) -- (0,4pt);}}}
\newcommand{\mysetminusSS}{\hbox{\tikz{\draw[line width=0.4pt,line cap=round] (1.5pt,0) -- (0,3pt);}}}
\newcommand{\mysetminus}{\mathbin{\mathchoice{\mysetminusD}{\mysetminusT}{\mysetminusS}{\mysetminusSS}}}
\renewcommand\setminus\mysetminus
\renewcommand\smallsetminus\mysetminus

\title[Tropicalization of $\psi$ classes]{Tropicalization of $\psi$ classes}

\author{Renzo Cavalieri}
\address{Department of Mathematics, Colorado State University, Fort Collins, Colorado 80523-1874}
\email{\href{mailto:renzo@colostate.edu}{renzo@colostate.edu}}
\author{Andreas Gross}
\address{Institut f\"ur Mathematik, Goethe--Universit\"at Frankfurt,
60325 Frankfurt am Main, Germany}
\email{\href{mailto:gross@math.uni-frankfurt.de}{gross@math.uni-frankfurt.de}}
\thanks{}


\begin{document}

\begin{abstract}
Under suitable conditions on a family of logarithmic curves, we endow the tropicalization of the family with an affine structure in a neighborhood of the sections in such a way that the tropical $\psi$ classes from \cite{psi-classes} arise as tropicalizations of algebraic $\psi$ classes.
\end{abstract}

\maketitle

\section{Introduction}

The main goal of this paper is to address the following question, posed by the authors and Hannah Markwig in \cite{psi-classes}:
\begin{quote}
    \emph{Are tropical $\psi$ classes the tropicalization of algebraic $\psi$ classes?}
\end{quote}
In  Section \ref{sec:Results}  we summarize the results and give a streamlined account of the \textit{story} this work is telling. In Section \ref{sec:context}  we provide context, motivation, and a discussion of the ideas informing our constructions.

\subsection{Results}
\label{sec:Results}
The first step in the journey is to define a notion of tropicalization that takes value in the category of tropical spaces, that is spaces that are locally modeled on abstract rational polyhedral complexes, and in addition are endowed with a sheaf of affine linear functions. 
\begin{nndef}[Tropicalization]
Let $X$ be a toroidal embedding with no self-intersections. The \textbf{tropicalization} of $X$ in the category of tropical spaces is obtained by endowing the boundary (extended) cone complex $\overline{\Sigma}_X$ with the sheaf of affine functions from Definition \ref{def:affine function}.
\end{nndef}
Informally, if $\sigma \in \Sigma_X$ corresponds to a (closed) stratum $V(\sigma)\subseteq X$, a piecewise linear function $\phi$  defined on a neighborhood $\Sigma_X^\sigma$ of $\sigma$ is declared affine when the corresponding line bundle (defined on the open set $X_\sigma$ obtained by removing all boundary divisors that do not meet $V(\sigma)$) trivializes on $V(\sigma)$:
\begin{equation}
{\mc O}_{X_\sigma}(\phi)\vert_{ V(\sigma)}\cong {\mc O}_{V(\sigma)}.    
\end{equation}
When $\sigma/\tau$ is a cone at infinity, one makes the additional requirement that $\phi$ is constant on $\tau$, meaning that the support of the divisor associated to $\phi$ does not contain any boundary divisor containing $V(\tau)$. This notion of tropicalization is functorial, and  it is invariant under log modifications of the toroidal variety $X$.

\begin{nntheorem}[Proposition \ref{prop:functoriality}, Proposition \ref{prop:invariance of affine structure under refinement}]
A morphism of toroidal varieties $f: X\to Y$ induces a morphism of tropical spaces:
    \begin{equation}
    \trof:= \Trop(f) \colon \vert \Sigma_X\vert\to \vert \Sigma_Y\vert \ .
    \end{equation}
If $f$ is a log modification, then $F$ is an isomorphism.
\end{nntheorem}

Since the main application of this technology is to families of curves, an important check of the soundness of the definitions is that in genus zero this notion of tropicalization recovers the previous ones (via torus embedding \cite{MacStu} or cross ratios \cite{psi-classes}).

\begin{nntheorem}[Theorem \ref{thm:cross-ratios on M0n}]
  The tropicalization of the toroidal variety $\overline{\mc M}_{0,n}$ is the tropical space ${\mc M}_{0,n}^\trop$.
\end{nntheorem}

The situation is  more subtle for curves of positive genus. A family $\pi: \mc C\to \mc B$ of logarithmic curves naturally produces a map of extended cone complexes
$\Pi: {\overline\TC}\to {\overline \TB}$; we call the family \textbf{tropicalizable} if its tropicalization gives rise to the expected (set theoretic) moduli map. As soon as a family of curve is tropicalizable, it is  \emph{almost} a family of tropical curves in the sense of \cite[Definition 3.10]{psi-classes}.

\begin{nntheorem}[Proposition \ref{prop:cross-ratios are linear}, Proposition \ref{prop:goodlocus}]
   Given $\pi: \mc C\to \mc B$ a tropicalizable family of stable, $n$-marked logarithmic curves, we obtain a map of tropical spaces $\Pi: {\overline\TC}\to {\overline\TB}$ such that:
   \begin{itemize}
       \item if $x$ is a genus-$0$  point  of  ${\overline \TC}$, we have the exact sequence 
       \begin{equation}
          0 \to \Aff_{\overline\TB,\Tpi(x)}\to \Aff_{{\overline\TC},x} \to \Omega^1_{{{\overline\TC}}_{\Tpi(x)},x}\to 0 \ . 
          \end{equation}
        \item if $x$ is a rational vertex of   ${\overline\TC}$ or a point on an edge adjacent to a rational vertex, then, near $x$, the affine functions on its fiber consist of all harmonic functions.  
   \end{itemize}
\end{nntheorem}

A necessary step towards tropicalizing $\psi$ classes is to define the notion of tropicalization of line bundles. In the next result we discuss when such an object is a tropical line bundle.

\begin{nntheorem}[Proposition \ref{prop:tropicalization of line bundle is tropical line bundle}, Proposition \ref{prop:if O(phi) is line bundle, then phi is CP}] Denote by $\mc L$ the invertible sheaf of a line bundle $L\to X$; let $\Trop(L)$ be the tropicalization of the total space of $L$ and  $\mathtt{U} \subseteq \overline{\Sigma}_X$ an open subcomplex of the tropicalization of the base space.
\begin{enumerate}
    \item $\Trop(L)$ is a tropical line bundle on $\mathtt U$ if and only if for every $\sigma/\tau\in U$ there exists a strictly piecewise linear function $\phi_{\sigma/\tau}$ on a neighborhood of $\tau$ that is constant on $\tau$ such that
    \begin{equation}\label{eq:tlbonU}
      \left(\mc O_{X_{\sigma}}(\phi_{\sigma/\tau}) \otimes \mc L\right)\vert_{V(\sigma)} \cong \mc O_{V(\sigma)} \ .   
    \end{equation}
    \item if $\mc L =  \mc O_X(\phi)$, for $\phi$ a strict piecewise linear function on $\Sigma_X$, then condition \eqref{eq:tlbonU} is equivalent to the existence of an affine function $\chi_{\sigma/\tau}$ on a neighborhood of $\sigma$ such that
    \begin{equation}\label{eq:CPmeansthuis}
     \chi\vert_{\tau}   = \phi\vert_{\tau}.  
    \end{equation}
\end{enumerate}
\end{nntheorem}

With this technology in place, we turn our attention to the $i$-th cotangent line bundle of a family of curves $\mc C\to \mc B$. The main result is that if the tropicalization $\overline{\TC}\to \overline\TB$ has enough local affine functions near the $i$-th section at infinity, then the tropicalization of the cotangent line bundle agrees with the definition of the tropical cotangent line bundle from \cite[Definition 6.16]{psi-classes}. 

\begin{nntheorem}[Theorem \ref{thm:if psi is tropicalizable, then it tropicalizes to psi}, Corollary \ref{cor:psi}]
Let $\mc C \to \mc B$ be a family of $n$-marked stable curves with tropicalization $\overline{\TC}\to \overline\TB$, and let $\phi_i$ be the strict piecewise linear function on $\TC$ having slope one along the ray dual to the $i$-th section of the family, and $0$ on all other rays. If $\Aff_{\overline{\TC}}(\phi_i)$ is a tropical line bundle, then
\begin{equation}
\Trop(\bL_i)\cong \bL_i^\trop \ .
\end{equation} 
By taking first Chern classes, we obtain
\begin{equation}
\Trop(\psi_i)= \psi_i^\trop \ .
\end{equation}
\end{nntheorem}

We conclude the manuscript with an extended example. We compute the tropical $\psi$ classes for two two-dimensional families of genus-one, tropical curves, obtained by stabilization from a space of tropical admissible covers. We show that the tropical $\psi$ classes of \cite{psi-classes} are the tropicalization of the algebraic $\psi$ classes, and that they agree with the operational perspective of tropicalization from \cite{Katz}.

\subsection{Context and commentary}
\label{sec:context}
Tropical geometry and moduli spaces of curves have enjoyed  fruitful interactions ever since Mikhalkin's seminal paper \cite{MikhalkinModuli}
interpreted the space of phylogenetic trees of \cite{billeraetal} as the tropicalization of $\mc M_{0,n}$. With Hannah  Markwig, we gave our perspective and account on the evolution of the subject in the introduction to \cite{psi-classes}, to which we refer the interested reader. The one line (cheeky) summary is, when it comes to tropical geometry and moduli spaces of curves, anything you may wish comes true for rational curves, while it appears to crash and burn for positive genus. This article is part of a collection of works (including, but not limited to \cite{ACP,CCUW,ChanGalatiusPayne, psi-classes}) proving that tropical geometry can still provide a meaningful and powerful approach to the study of the geometry and intersection theory of moduli spaces of positive genus curves.

The notion of tropical $\psi$ classes for families of tropical curves of arbitrary genus was introduced in \cite{psi-classes}. The key step there was to define the notion of family of tropical curves by requiring an \emph{affine structure}, that is, distinguished subsheaves of the sheaves of piecewise linear functions on both the base and the total space of the family satisfying some natural requirements: for example that the restriction of affine functions to fibers are harmonic (see Section \ref{sec:logfam} for a more complete review). Borrowing intuition from the algebraic identification of the $\psi$ class with the negative self-intersection of the corresponding section, the tropical cotangent line bundle $\bL_i$ of a family of tropical curves $\overline{\TC}\to \overline\TB$ is defined to be the pullback via the $i$-th section of the tropical counterpart of the co-normal bundle to the section: the $\Aff_{\TC}$ torsor whose local sections are affine functions on the finite part of $\overline{\TC}$, going to infinity  towards the section $\ts_i$ with slope $1$ along the fibers of the family.

The theory thus obtained is a combinatorial theory where the affine structures are canonically determined in certain cases: for example, for families of explicit tropical curves (that is when all vertices are rational); in this case we have natural correspondence statements with the corresponding algebraic families. In general, the theory of tropical $\psi$ classes is strictly broader than the algebraic one: in \cite[Example 6.17]{psi-classes}, the authors exhibited one-dimensional families of tropical curves where the degree of the tropical $\psi$ class cannot agree with the degree of $\psi$ on any corresponding algebraic family.

In order to make a meaningful connection with algebraic geometry, we  define a notion of tropicalization which is  less restrictive than torus embedded tropicalization, yet still contains  information about affine structures. While ultimately we wish to apply our constructions to logarithmic families of log smooth curves, we choose to develop the main definitions in the simpler set-up and language of toroidal embeddings with no self-intersections (which we  call toroidal varieties for simplicity): varieties with a distinguished divisor called boundary, which are locally isomorphic to toric varieties. A toroidal variety is assigned, in a functorial way, a cone complex agreeing with the (cone over the  topological) boundary complex, where each cone is additionally endowed with an integral lattice, see \cite{toremb}. 

In order to define a natural notion of affine structure on the cone complex of a toroidal variety, we use the following intuition from toric varieties and tropicalization of subvarieties of tori:  piecewise linear functions on the fan of a toric variety correspond to Cartier divisors, and therefore line bundles; the main role of the torus in these theories is that its characters are a class of functions that are invertible in the interior of the space, which can be thought of as sections of bundles that trivialize in the interior. We thus choose to define a piecewise linear function on (an open set of) the cone complex of a toroidal variety to be affine when the corresponding line bundle trivializes (on some collection of orbits depending on the chosen open set)\footnote{The idea of connecting functions on the tropicalization with Picard theory of the variety  also appears in \cite{TropicalComplexes, GJR}.}. The resulting affine structure makes the cone complex into a tropical space, which we call the \textbf{tropicalization} of the toroidal variety. 

In the broader generality of logarithmic schemes with a divisorial log structure, a functorial assignment of a generalized cone complex is given in \cite{Uli:Fun}, where it is called logarithmic tropicalization; the fact that multiple faces of a single cone can be identified makes it a bit delicate to immediately define an affine structure on the logarithmic tropicalization. We  sidestep this problem by observing that the sheaf of affine functions  defined is invariant under logarithmic modifications (roughly speaking sequences of blow-ups and blow-downs of strata and their transforms); a generalized cone complex can  be refined to an honest cone complex, which is the cone complex of an appropriate log modification of the original space. Hence one may define the affine structure on the logarithmic tropicalization via any suitable log modification making the space toroidal.

There are two classes of objects whose tropicalization we are especially interested in: families of curves and line bundles. We begin discussing the latter.

 A line bundle $L$ over a toroidal variety $X$ may be given a natural toroidal boundary, obtained as the union of the the pull-back of the boundary on the base, plus the zero section. We define the tropicalization of $L$ essentially to be the tropicalization of its total space, see Remark \ref{rem:tlbviapb}; this is a tropical space, but it may not have enough affine functions to be a tropical line bundle as in \cite[Definition 4.4]{MZJacobians}(see Section \ref{sec:tlb}). In an open neighborhood of a cone $\sigma\in \Sigma_X$, this amounts to the invertible sheaf $\mc L$ of $L$  restricted to the orbit $V(\sigma)$ being the restriction of a sheaf of \textbf{boundary type}, i.e.\ 
 \begin{equation}
    \mc L\vert_{ V(\sigma)} \cong   \mc O_{X_\sigma}\left(\sum a_\rho D_\rho\right)\big\vert_{ V(\sigma)}, 
 \end{equation}
 where the $D_\rho$'s are the boundary divisors in $X_\sigma$. For a cone at infinity $\sigma/\tau$ the situation is analogous: we need $\mc L\vert_{ V(\sigma)}$ to be of the form $\mc O_{ V(\tau)_\sigma}(\sum a_\rho (D_\rho \cap V(\tau)))\vert_{V(\sigma)}$, where the $D_\rho$'s are  boundary divisors which do not contain the orbit $V(\tau)$: this means that on $V(\sigma)$,  the line   bundle is the restriction of a line bundle of boundary type on the orbit $V(\tau)$. It follows that a line bundle of boundary type, which we  associate to a piecewise linear function $\phi$ on $\Sigma_X$, tropicalizes to a tropical line bundle on the cone complex $\Sigma_X$, but not necessarily on the whole extended cone complex $\overline{\Sigma}_X$. We identify a combinatorial condition, which we call being \textbf{combinatorially principal}, that makes it a tropical line bundle on a cone at infinity of the form $\sigma/\tau$: there must exist an affine function $\phi_\tau$  which agrees with $\phi$ on the cone $\tau$. 

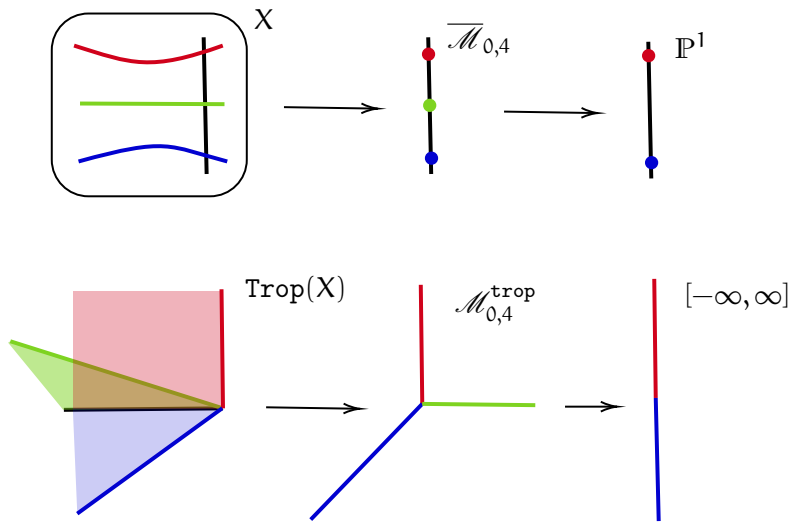
\begin{figure}
     \centering

\tikzset{every picture/.style={line width=0.75pt}} 

\begin{tikzpicture}[x=0.75pt,y=0.75pt,yscale=-.75,xscale=.75]

\draw   (60,49.03) .. controls (60,35.44) and (71.01,24.43) .. (84.6,24.43) -- (166.54,24.43) .. controls (180.13,24.43) and (191.14,35.44) .. (191.14,49.03) -- (191.14,122.83) .. controls (191.14,136.41) and (180.13,147.43) .. (166.54,147.43) -- (84.6,147.43) .. controls (71.01,147.43) and (60,136.41) .. (60,122.83) -- cycle ;
\draw [line width=1.5]    (162.14,40.43) -- (164.14,132.43) ;
\draw [color={rgb, 255:red, 208; green, 2; blue, 27 }  ,draw opacity=1 ][line width=1.5]    (75,46) .. controls (114.14,58.43) and (126.14,63.43) .. (175,46) ;
\draw [color={rgb, 255:red, 2; green, 2; blue, 208 }  ,draw opacity=1 ][line width=1.5]    (78,124) .. controls (138.14,107.43) and (138.14,113.43) .. (178,124) ;
\draw [color={rgb, 255:red, 126; green, 211; blue, 33 }  ,draw opacity=1 ][line width=1.5]    (79,85) -- (176.14,85.43) ;
\draw    (216.14,87.43) -- (277.14,88.4) ;
\draw [shift={(279.14,88.43)}, rotate = 180.91] [color={rgb, 255:red, 0; green, 0; blue, 0 }  ][line width=0.75]    (10.93,-3.29) .. controls (6.95,-1.4) and (3.31,-0.3) .. (0,0) .. controls (3.31,0.3) and (6.95,1.4) .. (10.93,3.29)   ;
\draw [line width=1.5]    (313.14,40.43) -- (315.14,132.43) ;
\draw    (364.14,90.43) -- (425.14,91.4) ;
\draw [shift={(427.14,91.43)}, rotate = 180.91] [color={rgb, 255:red, 0; green, 0; blue, 0 }  ][line width=0.75]    (10.93,-3.29) .. controls (6.95,-1.4) and (3.31,-0.3) .. (0,0) .. controls (3.31,0.3) and (6.95,1.4) .. (10.93,3.29)   ;
\draw [line width=1.5]    (461.14,43.43) -- (463.14,135.43) ;
\draw  [color={rgb, 255:red, 2; green, 2; blue, 208 }  ,draw opacity=1 ][fill={rgb, 255:red, 2; green, 2; blue, 208 }  ,fill opacity=1 ] (311.57,121.79) .. controls (311.57,119.69) and (313.27,118) .. (315.36,118) .. controls (317.45,118) and (319.14,119.69) .. (319.14,121.79) .. controls (319.14,123.88) and (317.45,125.57) .. (315.36,125.57) .. controls (313.27,125.57) and (311.57,123.88) .. (311.57,121.79) -- cycle ;
\draw  [color={rgb, 255:red, 2; green, 2; blue, 208 }  ,draw opacity=1 ][fill={rgb, 255:red, 2; green, 2; blue, 208 }  ,fill opacity=1 ] (459.57,124.79) .. controls (459.57,122.69) and (461.27,121) .. (463.36,121) .. controls (465.45,121) and (467.14,122.69) .. (467.14,124.79) .. controls (467.14,126.88) and (465.45,128.57) .. (463.36,128.57) .. controls (461.27,128.57) and (459.57,126.88) .. (459.57,124.79) -- cycle ;
\draw  [color={rgb, 255:red, 208; green, 2; blue, 27 }  ,draw opacity=1 ][fill={rgb, 255:red, 208; green, 2; blue, 27 }  ,fill opacity=1 ] (309.57,51.79) .. controls (309.57,49.69) and (311.27,48) .. (313.36,48) .. controls (315.45,48) and (317.14,49.69) .. (317.14,51.79) .. controls (317.14,53.88) and (315.45,55.57) .. (313.36,55.57) .. controls (311.27,55.57) and (309.57,53.88) .. (309.57,51.79) -- cycle ;
\draw  [color={rgb, 255:red, 208; green, 2; blue, 27 }  ,draw opacity=1 ][fill={rgb, 255:red, 208; green, 2; blue, 27 }  ,fill opacity=1 ] (457.57,52.79) .. controls (457.57,50.69) and (459.27,49) .. (461.36,49) .. controls (463.45,49) and (465.14,50.69) .. (465.14,52.79) .. controls (465.14,54.88) and (463.45,56.57) .. (461.36,56.57) .. controls (459.27,56.57) and (457.57,54.88) .. (457.57,52.79) -- cycle ;
\draw  [color={rgb, 255:red, 126; green, 211; blue, 33 }  ,draw opacity=1 ][fill={rgb, 255:red, 126; green, 211; blue, 33 }  ,fill opacity=1 ] (310.36,86.21) .. controls (310.36,84.12) and (312.05,82.43) .. (314.14,82.43) .. controls (316.23,82.43) and (317.93,84.12) .. (317.93,86.21) .. controls (317.93,88.31) and (316.23,90) .. (314.14,90) .. controls (312.05,90) and (310.36,88.31) .. (310.36,86.21) -- cycle ;
\draw [line width=1.5]    (68.14,291.14) -- (175.14,290.14) ;
\draw [color={rgb, 255:red, 126; green, 211; blue, 33 }  ,draw opacity=1 ][line width=1.5]    (32.14,245) -- (175.14,290.14) ;
\draw  [draw opacity=0][fill={rgb, 255:red, 126; green, 211; blue, 33 }  ,fill opacity=0.5 ] (174.64,290.07) -- (67.64,291.07) -- (30.64,245.93) -- (32.14,245) -- cycle ;
\draw [color={rgb, 255:red, 208; green, 2; blue, 27 }  ,draw opacity=1 ][line width=1.5]    (174.14,210) -- (175.14,290.14) ;
\draw  [draw opacity=0][fill={rgb, 255:red, 208; green, 2; blue, 27 }  ,fill opacity=0.3 ] (74.14,211) -- (174.64,211) -- (174.64,290.07) -- (74.14,290.07) -- cycle ;
\draw [color={rgb, 255:red, 2; green, 2; blue, 208 }  ,draw opacity=1 ][line width=1.5]    (77.14,361) -- (174.64,290.07) ;
\draw  [draw opacity=0][fill={rgb, 255:red, 2; green, 2; blue, 208 }  ,fill opacity=0.2 ] (175.14,290.14) -- (77.14,361) -- (74.14,290.07) -- cycle ;
\draw [color={rgb, 255:red, 208; green, 2; blue, 27 }  ,draw opacity=1 ][line width=1.5]    (308.14,207) -- (309.14,287.14) ;
\draw [color={rgb, 255:red, 2; green, 2; blue, 208 }  ,draw opacity=1 ][line width=1.5]    (234.14,365) -- (309.14,287.14) ;
\draw [color={rgb, 255:red, 126; green, 211; blue, 33 }  ,draw opacity=1 ][line width=1.5]    (309.14,287.14) -- (385.14,288) ;
\draw [color={rgb, 255:red, 208; green, 2; blue, 27 }  ,draw opacity=1 ][line width=1.5]    (465.14,203) -- (466.14,283.14) ;
\draw [color={rgb, 255:red, 2; green, 2; blue, 208 }  ,draw opacity=1 ][line width=1.5]    (468.14,366) -- (466.14,283.14) ;
\draw    (204.14,289.43) -- (265.14,290.4) ;
\draw [shift={(267.14,290.43)}, rotate = 180.91] [color={rgb, 255:red, 0; green, 0; blue, 0 }  ][line width=0.75]    (10.93,-3.29) .. controls (6.95,-1.4) and (3.31,-0.3) .. (0,0) .. controls (3.31,0.3) and (6.95,1.4) .. (10.93,3.29)   ;
\draw    (405.14,289) -- (438.14,289) ;
\draw [shift={(440.14,289)}, rotate = 180] [color={rgb, 255:red, 0; green, 0; blue, 0 }  ][line width=0.75]    (10.93,-3.29) .. controls (6.95,-1.4) and (3.31,-0.3) .. (0,0) .. controls (3.31,0.3) and (6.95,1.4) .. (10.93,3.29)   ;

\draw (194,19.4) node [anchor=north west][inner sep=0.75pt]    {$X$};
\draw (324,28.4) node [anchor=north west][inner sep=0.75pt]    {$\overline{\mc M}_{0,4}$};
\draw (477,35.4) node [anchor=north west][inner sep=0.75pt]    {$\PP^1$};
\draw (188,198.4) node [anchor=north west][inner sep=0.75pt]    {$\Trop(X)$};
\draw (329,205.4) node [anchor=north west][inner sep=0.75pt]    {$\mc M_{0,4}^\trop$};
\draw (483,203.4) node [anchor=north west][inner sep=0.75pt]    {$[-\infty, \infty]$};

\end{tikzpicture}

\caption{An illustration that tropical cross ratios are the tropicalization of algebraic cross ratios. The top of the figure shows a cross ratio as factoring through a morphism to $\overline{\mc M}_{0,4}$, which is then identified with $\PP^1$ (the difference between the two spaces is the boundary structure). Tropical cross ratios are similarly obtained in the row below.
Functoriality of tropicalization shows the bottom part of the figure to be the tropicalization of the top.}
\label{fig:tropofcrossratio}
\end{figure}
 
Turning our attention to families of curves, given a family $\pi\colon\mc C\to \mc B$, the tropicalization gives a map of tropical spaces $\Pi\colon\overline{\TC}\to \overline{\TB}$. A paraphrasis of $\mc C\to \mc B$ being tropicalizable is that, once appropriately endowing the points of $\overline{\TC}$ with a genus function, the fibers of $\Pi$ are the dual graphs of the corresponding fibers of $\pi$. This is some sort of minimal requirement, that due to automorphisms and monodromy issues may not always be guaranteed: in Example \ref{ex:monodromy} we see an instance of what may go wrong and how it can be remedied.

By functoriality, tropical cross ratios are tropicalization of algebraic cross ratios, see Figure \ref{fig:tropofcrossratio}. This fact guarantees that this notion of tropicalization agrees with the \emph{classical} notion (coming from embedding into a torus) for families of rational pointed curves and gives yet another equivalent way to get to $\mc M_{0,n}^\trop$.
The tropicalization $\overline{\TC}\to \overline{\TB}$ is in general not a family of tropical curves in the sense of \cite[Definition 3.10]{psi-classes}: one may not have enough (germs of) affine functions on the fibers, but this failure is restricted to points of positive genus, or points on edges that are not adjacent to any rational vertex. In particular, we recover the result from \cite[Proposition 4.24]{psi-classes} that $\overline{\TC}\to \overline{\TB}$ is a family of tropical curves when the image of ${\TB}$ via the tropical moduli map lies in the \emph{good} locus ${\GoodAtlas{g,n}}$ (\cite[Definition 4.22]{psi-classes}).

For a family of marked tropical  curves $\overline{\TC}\to \overline{\TB}$, the tropical cotangent line bundle to the $i$-th section is defined 
as
\begin{equation}
    \bL_i^\trop := \ts_i^\ast(\Aff_{\overline{\TC}}(-\ts_i)),
\end{equation}
see \cite[Definition 6.16]{psi-classes}. In this paper, we define $\Aff_{\overline{\TC}}(-\ts_i)$ when $\overline{\TC}$ is a tropical space, and observe that in general it is a pseudo $\Aff_{\overline{\TC}}$ torsor: its local sections  either give a torsor or are empty. An important observation from \cite[Propositions 3.24, 3.25]{psi-classes} is that  $\overline \TC\to \overline \TB$ being a family of tropical curves ensures that $\Aff_{\overline \TC}(-\ts_i)$ really is an $\Aff_{\overline \TC}$-torsor. However, this condition is stronger than strictly necessary: for $\Aff_{\overline \TC}(-\ts_i)$ to be a tropical line bundle, the affine structure of $\overline{\TC}$ only matters locally near $\ts_i$, so one only needs $\overline{\TC}\to \overline{\TB}$  to satisfy the conditions on the affine structure of families of tropical curves in a neighborhood of $\ts_i(\overline\TB)$. When this happens, it is then possible to check that the tropicalization of the cotangent line bundle is a tropical line bundle, and it agrees with $\ts_i^\ast(\Aff_{\overline \TC}(-\ts_i))$.

Pulling all the loose strings together, one finally can answer the motivating question we began this section with.
\begin{quote}
    \emph{Given a tropicalizable family of stable marked curves $\mc C \to \mc B $ with tropicalization $\overline{\TC}\to \overline{\TB}$, we have
    \begin{equation}
        \Trop(\psi_i)= \psi_i^\trop
    \end{equation}
    whenever $\Aff_{\overline \TC}(-\ts_i)$ is a tropical line bundle; further, this notion admits the following combinatorial characterization: the piecewise linear function $\phi_{\rho_i}$ on $\overline \TC$ having slope one on the ray $\rho_i$ dual to $s_i$ and zero on all other rays is combinatorially principal.
    }
\end{quote}
One might interpret the above conclusion as follows: tropical geometry does not see the algebraic information contained in  $\psi$ classes on all families of curves. In order  to capture this information tropically, the family must degenerate \emph{sufficiently} near the section. However, when tropical geometry does see $\psi$ classes, it indeed sees the tropical $\psi$ classes of \cite{psi-classes}.

\subsection{Computations}
We conclude this paper with an extended example of a computation of tropical $\psi$ classes for two two-dimensional families of  tropical curves of genus $1$ with two marked points. The families are obtained from tropical admissible covers by forgetting some of the marked ends. In both cases we show that the tropical $\psi$ class is the tropicalization of the algebraic $\psi$ class, and exhibit it as a one-dimensional tropical cycle, that is a Minkowski weight on the base of the family. The tropical cycle obtained  agrees with the operational tropicalization of the $\psi$ class, that is with the Minkowski weight decorating each cone with the intersection number of $\psi$ with the corresponding stratum. Besides illustrating a concrete instance of the theoretical story told, there are some interesting lessons one can observe from this extended example.
In order to prove that the tropicalization of the families  curves have enough affine functions to define the tropical $\psi$ classes it is more efficient to rely on functoriality,   rather than on the direct definition of tropicalization. Then one is essentially able to use the affine functions from the genus zero theory to obtain affine functions near the section. The second observation is that the combinatorics of tropical intersection theory is  manageable, which is encouraging that this technology might prove a useful tool for the study of tautological rings of moduli spaces of curves. There has been recent interest in approaching the study of the logarithmic tautological ring of moduli spaces of curves \cite{hs:logint,mps:Hodge, mr:case} as the image of a ring homomorphism from the ring of piecewise polynomial functions on $\mc M_{g,n}^\trop$. The affine linear functions defined here generate an ideal that lies in the kernel of that ring homomorphism.

\subsection{Notations and Conventions}
 We assume throughout  working on an algebraically closed field  of characteristic $0$; we expect all   constructions to go through for perfect fields, but care should be taken to what objects the constructions are applied to. In this paper, the tropical semiring is $\mathds{T} = \{(-\infty, \infty], \min, +\}$; the choice of this convention is to make the affine structure near the section at infinity parallel the affine structure of the tropicalization of the normal bundle to the section near its  ``zero'' section.
We mostly use calligraphic fonts for algebraic objects, and typewriter ones for tropical objects. The superscript $\trop$ is  used for certain tropical objects where the use of an identifying font was considered not viable, whereas the tropicalization functor is denoted by $\Trop$.

\subsection{Acknowledgements}
We wish to thank Hannah Markwig, Dhruv Ranganathan and Martin Ulirsch for interesting conversations related to this project.
Renzo Cavalieri is grateful for support from NSF grant DMS - 2100962 and Simons Collaboration Grant MPS-TSM-00007937. 
Andreas Gross has received funding from the Deutsche Forschungsgemeinschaft (DFG, German Research Foundation) TRR 326 \emph{Geometry and Arithmetic of Uniformized Structures}, project number 444845124, from the Deutsche Forschungsgemeinschaft (DFG, German Research Foundation) Sachbeihilfe \emph{From Riemann surfaces to tropical curves (and back again)}, project number 456557832, and from the Marie-Sk\l{}odowska-Curie-Stipendium Hessen (as part of the HESSEN HORIZON initiative).

\section{Preliminaries}

We recall the relevant definitions and results about tropical moduli spaces of curves, tropical cycles, and tropical line bundles in the generality introduced in \cite{psi-classes}. We refer to that manuscript for a comprehensive list of references to the relevant preexisting literature.


\subsection{Tropical curves and their moduli}
\label{sec:logfam}

In \cite{psi-classes}, moduli spaces of tropical curves are introduced as stacks over the category of tropical spaces. We review the relevant definitions. 

A \emph{\TPL{}-space} is a pair $({ X}, \PL_{ X})$ consisting of a topological space ${ X}$ and a sheaf $\PL_{X}$ of continuous $\T$-valued functions on ${ X}$  that is locally isomorphic to $(U, \PL_U)$, the sheaf of continuous, integral, piecewise linear functions of some open subset $U$ of a rational polyhedral set $P$ in $\T^n$. Here, two pairs $(Y, \mc F_Y)$ and $(Z, \mc F_Z)$  of topological spaces and sheaves of functions on them are isomorphic if there exists a homeomorphism $f\colon Y\to Z$ so that for a function $\phi$ on an open subset of $Z$ we have that $\phi\circ f$ is a section of $\mc F_Y$ if and only if $\phi$ is a section of $\mc F_Z$.  For a \TPL-space ${ X}$ we denote by $\PLfin_{ X}$ the subsheaf of $\R$-valued piecewise linear functions. A morphism of \TPL-spaces is a continuous map that pulls piecewise linear functions (section of $\PL$) back to piecewise linear functions.

A \emph{tropical space} is a pair $(\mathtt{X}, \Aff_{\mathtt X})$ consisting of a \TPL-space ${\mathtt X}$ and a subsheaf $\Aff_{\mathtt X}\subseteq \PLfin_{\mathtt X}$, the sheaf of \emph{affine functions}, that contains the constant sheaf $\R_{\mathtt X}$. A morphism of tropical spaces is morphism of \TPL-spaces that pulls affine functions back to affine functions. In that case, we also say that the underlying morphism of \TPL-spaces is \emph{linear}. The \emph{tropical cotangent bundle} is the quotient sheaf $\Omega_{\mathtt X}^1\coloneqq \Aff_{\mathtt X}/\R_{\mathtt X}$.

A \emph{tropical curve} is a one-dimensional tropical space $\TC$, together with a finitely supported genus function $\gamma:\TC\to \bZ_{\ge 0}$. The genus of a tropical curve is  $\sum_{x\in \TC} \gamma(x)+ \mathrm{b_1}(\TC)$.  A \emph{cycle rigidification} of a genus-$g$ tropical curve $\TC$ is a $g$-tuple of elements spanning $H_1(\TC,\Z)$ and such that each element is either $0$ or coming from a circuit in $\TC$. We denote by $\Atlass{g,n}$ the set of isomorphism classes of cycle-rigidified stable\footnote{ see \cite[Definitions 3.2, 3.4]{psi-classes} for the notion of stability of a tropical curve. Such notion parallels the algebraic one, that allows only nodal singularities and finitely many automorphisms; it takes a little bit of care in describing the tropical version of a nodal singularity.} $n$-marked tropical curves of genus $g$. Forgetting the metric structure (that is,  only retaining the topological graph and the cycle rigidification, the so-called combinatorial type) stratifies $\Atlass{g,n}$ into locally closed sets that can be naturally identified with open cones. For a combinatorial type $\Gamma$  we denote by $\sigma_\Gamma^\diamond$ the corresponding open cone and by $\sigma_\Gamma$ its closure. The natural face morphisms among the closed cones gives $\Atlass{g,n}$ the structure of an extended cone complex. Given an $n\sqcup\{\star\}$-marked curve, one can forget the $\star$-mark and stabilize. As stabilization is a contraction, a cycle rigidification on the original graph induces a cycle rigidification on the stabilized graph. We thus obtain a map
\begin{equation}
\mu_\star \colon \Atlass{g,n\sqcup\{\star\}}\to \Atlass{g,n} \ ,    
\end{equation}
the forgetful map, which is a morphism of extended cone complexes. The fiber of $\mu_\star$ over a point $\Atlass{g,n}$ corresponding to a cycle-rigidified tropical curve $\TC$ is isomorphic to $\TC$. In particular, there are $n$ natural sections of $\mu_\star$ and a global genus function $\Atlass{g,n\sqcup\{\star\}}\to \Z_{\geq 0}$.

A \emph{family of genus-$g$, $n$-marked stable \TPL-curves} over a \TPL-space ${ X}$ is a morphism $\pi\colon C\to X$ of \TPL-spaces, together with $n$ sections $s_i\colon { X}\to  C$, $1\leq i\leq n$, and a genus function $ C\to \Z_{\geq 0}$, such that every $x\in { X}$ has an open neighborhood $U$ for which there exists a map $f\colon U\to \Atlass {g,n}$ and an isomorphism $\pi^{-1}U\cong \Atlass{g,n\sqcup\{\star\}} \times_{\mu_\star,\Atlass {g,n},f} U$ over $U$ that respects the sections and the genus function. 

A \emph{family of genus-$g$, $n$-marked stable tropical curves} over a tropical space ${\mathtt X}$ is a family of genus-$g$, $n$-marked \TPL-curves  $\Pi\colon \mathtt C\to {\mathtt X}$ satisfying the following  conditions:
\begin{enumerate}
    \item for each $x\in {\mathtt X}$, the fiber ${\mathtt C}_x\coloneqq \Pi^{-1}\{x\}$, equipped with the induced affine structure, is a stable tropical curve,
    \item for every $y\in \mathtt C$, the sequence
    \begin{equation}
        0 \to 
        \Omega^1_{{\mathtt X}, \Pi(y)} \to
        \Omega^1_{{\mathtt C}, y} \to 
        \Omega^1_{{\mathtt C}_{\Pi(y)},y} \to
        0
    \end{equation}
induced by pull-back from ${\mathtt X}$ and restriction to ${\mathtt C}_{\Pi(y)}$, is exact. This can be phrased equivalently as saying that affine functions that are constant on fibers are pull-backs of affine functions on the base. 
\end{enumerate}

There is a natural affine structure on the space $\Atlass{g,n}$ induced by  \emph{cross ratio functions}. We  describe these functions, starting with $\Atlass{0,4}=\Mgnbart{0,4}$. Choosing an ordering, e.g.\ $((p_1,p_2), (p_3,p_4))$, of the four markings defines two unique minimal paths in a curve $[\TC]\in \Mgnbart{0,4}$, one from $p_1$ to $p_2$ and one from $p_3$ to $p_4$. The value of the cross ratio function $\xi_{((p_1,p_2),(p_3,p_4))}$ at $[\TC]$ is the signed length of the intersection of these two paths. Let $\Gamma$ be a combinatorial type of cycle-rigidified $n$-marked genus-$g$ curves, let $\widetilde \Gamma$ denote the universal covering space of $\Gamma$, and let $T\subset \widetilde \Gamma$ be a connected open subset all of whose vertices have genus $0$. Note that $T$ is is automatically a tree, that is  $h^1(T)=0$, and its leaves  are open edge segments. For every combinatorial type $\Gamma'$ specializing to $\Gamma$, the preimage $T'$ of $T$ under the  induced retraction $\widetilde \Gamma'\to \widetilde \Gamma$ of universal covers is a connected open subset of $\widetilde \Gamma'$ not containing higher-genus vertices, and the retraction induces a bijection between the leaves of $T'$ and the leaves of $T$. Forgetting the rest of the curve, this produces for every $\Gamma'$ specializing to $\Gamma$ a map $\sigma_{\Gamma'}^\diamond\to \Mgnbart{0, m}$, where $m$ is the number of leaves of $T$. Given a quadruple $((p_1,p_2),(p_3,p_4))$ of distinct leaves of $T$, we can postcompose with the forgetful morphism $\Mgnbart{0,m}\to \Mgnbart{0,4}$ that forgets all but the marks $p_1,\ldots,p_4$. From $\Mgnbart{0,4}$ we can then pull-back the cross ratio function $\xi_{((p_1,p_2),(p_3,p_4))}$ to obtain a function on $\sigma_{\Gamma'}$.  For different $\Gamma'\to \Gamma$, these functions glue together  and define a cross ratio function $\xi_{(T,(p_1,p_2),(p_3,p_4))}$ on a neighborhood of $\sigma_\Gamma^\diamond$. This is the cross ratio function $\xi_c$ associated to the cross ratio datum  $c=(T,(p_1,p_2),(p_3,p_4))$  (the tuple $c$ is not literally a cross ratio datum in the sense of  \cite[Definition 4.8]{psi-classes}, but there is a natural induced cross ratio datum). The affine structure of $\Atlass{g,n}$ is the sheaf generated by  cross ratio functions and constants. 
It is in fact generated by constants and the cross ratio functions associated to  \emph{primitive} cross ratio data:  these correspond to the cases where $T$ is a neighborhood of a single vertex $v$ or of a single edge $e$. The four markings $p_1,\ldots,p_4$ are induced by two pairs of flags $((f_1,f_2),(f_3,f_4))$ at $v$ (resp.\  at the end points of $e$) and we write $c_{((f_1,f_2),(f_3,f_4))}$ (resp.\ $c_{(e,(f_1,f_2),(f_3,f_4))}$) for the cross ratio datum induced by $(T,(p_1,p_2),(p_3,p_4))$ in the two respective cases and $\xi_{((f_1,f_2),(f_3,f_4))}$ (resp.\ $\xi_{(e,(f_1,f_2),(f_3,f_4))})$ for the associated cross ratio functions.

With these affine structures, the forgetful map $\mu_\star\colon \Atlass{g,n\sqcup\{\star\}}\to \Atlass{g,n}$ is a morphism of tropical spaces, but not a family of tropical curves. What fails is the condition on the fibers: it is true that all functions on $\Atlass{g,n\sqcup\{\star\}}$ are harmonic on the fibers, but not every harmonic function on a fiber is the restriction of an affine function.

\subsection{Cone complexes and toroidal embeddings}
\label{sec:ccte}

We  recall how to associate a cone complex to a toroidal embedding without self intersections. The original reference is \cite{toremb}, where what we call cone complex was called conical polyhedral complex with integral structures, and we refer  there and to \cite{ACP,Uli:Fun,KajiwaraTropicalToric,CCUW} for a more detailed treatment.

A \emph{(abstract) cone} is a pair $(\sigma, M^\sigma)$ consisting of a topological space $\sigma$ and a lattice $M^\sigma$ of real-valued continuous functions on $\sigma$ with the property that the natural map $\sigma\to \Hom(M^\sigma,\R)$ maps $\sigma$ homeomorphically onto its image, a strictly convex rational polyhedral cone in  $\Hom(M^\sigma,\R)$. If we denote $M^\sigma_+=\{m\in M^\sigma: m\geq 0\}$, then the image of $\sigma$ in $\Hom(M^\sigma,\R)$ is precisely $\Hom(M^\sigma_+,\R_{\geq 0})$. The faces of $(\sigma,M^\sigma)$ are the pairs $(\tau, \{m\vert_\tau:m\in M^\sigma\})$, where $\tau$ is a face of $\sigma$. The \emph{relative interior} of a cone $\sigma$ is the complement of all its proper faces; we denote it by $\sigma^\diamond$. A \emph{cone complex} is pair $(\vert\Sigma\vert,\Sigma)$ consisting of a topological space $\vert\Sigma\vert$ and a collection $\Sigma$ of cones whose underlying sets are closed subsets of $\vert\Sigma\vert$ and such that $\vert\Sigma\vert=\bigsqcup_{\sigma\in \Sigma}\sigma^\diamond$. Moreover, the set $\Sigma$ is closed under taking faces and intersections. A continuous function $\phi$ on a subset $U\subseteq \Sigma$ is \emph{strict piecewise linear} if for all $\sigma\in \Sigma$ there exists $m\in M^\sigma$ and $c\in \R$ with $\phi\vert_{\sigma\cap U}=m\vert_{\sigma\cap U}+c$. We denote the group of strict piecewise linear functions on $U$ by $\sPL_\Sigma(U)$.

For a cone $\sigma$, the \emph{extended cone} $\overline\sigma$ is given by $\overline\sigma=\Hom(M^\sigma_+, \R_{\geq 0}\cup\{\infty\})$. The extended cones of the cones in a cone complex $\Sigma$ can be glued to an extended cone complex $\overline \Sigma$. The extended cone complex $\Sigma$ has a natural stratification $\overline\Sigma=\bigsqcup_{\tau\prec \sigma\in \Sigma}\sigma/\tau$. Any choice of generators of $M^\sigma_+$ induces an embedding of an extended cone $\overline\sigma$ into $\T^n$ for some $n\in \Z_{\geq 0}$ and thus $\overline\sigma$ obtains the structure of a \TPL{}-space. Declaring a function on an extended cone complex to be piecewise linear if and only if its restrictions to all extended cones of the complex are piecewise linear defines a \TPL{}-structure on any extended cone complex. 

A toroidal embedding without self-intersection is a pair $(X_0,X)$ consisting of a variety $X$ and an open subset $X_0\subseteq X$ that locally looks like the inclusion of the big torus into a toric variety. More precisely, for every point $x\in X$ there exists a toric chart, i.e.\ an open neighborhood $U$ of $x$, a toric variety $Y$ with big open torus $T$, and an étale morphism $f\colon U\to Y$ with $f^{-1}T=U\cap X_0$.\footnote{In the language of logarithmic geometry, toroidal embeddings without self-intersections are log smooth log schemes over a point with trivial log structure, where the log structure is defined on the Zariski site.} The toroidal embedding $X$ has an associated cone complex $\Sigma_X$, with each $\sigma\in \Sigma_X$ corresponding to a \emph{stratum} $O(\sigma)$ in the stratification of $X$ induced by the toric charts. If we denote by $V(\sigma)$ the closure of $O(\sigma)$, the strata closures $V(\rho)$ for rays $\rho\in \Sigma_X(1)$ are  the  irreducible components of the \emph{boundary} $X\setminus X_0$, and all other strata are connected components of $\bigcap_{\rho\in I}V(\rho)\setminus \bigcup_{\rho\notin I}V(\rho)$ for subsets $I\subseteq  \Sigma_X(1)$. For a cone $\sigma\in \Sigma_X$, the monoid $M^\sigma$ is naturally identified with the set of effective boundary Cartier divisors near $O(\sigma)$. In particular, the functions in $\sPL(\Sigma)$ with trivial constant part are in natural bijection with boundary Cartier divisors on $X$.

A dominant toroidal  morphism of toroidal embeddings $X$ and $Y$ is a dominant morphism $X\to Y$ of schemes that can be expressed as a morphism of toric varieties in toric charts. A toroidal morphism from $X$ to $Y$ is any morphism of schemes $X\to Y$ that factors through a dominant toroidal morphism $X\to V(\sigma)$ for some $\sigma\in \Sigma_Y$.\footnote{Dominant toroidal morphisms are precisely the log smooth morphisms between toroidal varieties. Non-dominant toroidal morphisms do not define morphisms of log schemes.} A dominant toroidal morphism $f\colon X\to Y$ induces a morphism $\Trop(f)\colon \Sigma_X\to \Sigma_Y$, that is a continuous map $\vert\Sigma_X\vert\to \vert\Sigma_Y\vert$ mapping the cones of $\Sigma_X$ linearly into cones of $\Sigma_Y$. The map $\Trop(f)$  extends to a morphism on the extended cone complexes. For $\tau\in \Sigma_Y$, we have a natural identification of $\Sigma_{V(\tau)}$ with the subcomplex $\bigcup_{\tau\subseteq \sigma\in \Sigma_Y} \sigma/\tau$ of $\overline\Sigma_Y$. In particular, for a not necessarily dominant toroidal morphism $f\colon X\to Y$ there is an induced morphism $\Trop(f)\colon \overline\Sigma_X\to \overline\Sigma_Y$.

\subsection{Tropical cycles}
We present the notion of tropical cycles and their intersection pairing with tropical divisors in the generality developed in \cite{psi-classes}. 

A $k$-\emph{weight} on a cone complex $\Sigma$ is an equivalence class of pairs $(\Delta,c)$, where $\Delta$ is a proper subdivision of $\Sigma$ and $c\colon \Delta(k)\to \Z$ is a map. Here, $\Delta(k)$ denote the set of $k$-dimensional cones of $\Delta$, and the equivalence relation is induced by compatibility under refinements of $\Delta$. Given an affine structure $\Aff_\Sigma$, which we assume to satisfy the condition that $\Omega^1_\Sigma$ is constant on all cones, we define the balancing condition for weights as follows: a $1$-weight represented by $(\Delta, c)$ is balanced with respect to an affine function $\phi\in \Aff_\Sigma(\Sigma)$ with $\phi(0_\Sigma)=0$ (here, $0_\Sigma$ refers to the cone point of $\Sigma$), if we have $\sum_{\rho\in \Delta(1)} \phi(u_\rho)\cdot c(\rho)=0$, where $u_\rho$ denotes the primitive lattice generator of a ray $\rho$. We say that $c$ is \emph{balanced} if it is balanced with respect to every $\phi\in \Aff_\Sigma(\Sigma)$ with $\phi(0_\Sigma)=0$. More generally, a $k$-weight represented by $(\Delta, c)$ is balanced if for every $(k-1)$-dimensional cone $\tau\in \Delta$ and every affine function $\phi\in \Aff_\Sigma(\Delta^\tau)$ with $\phi\vert_\tau=0$, the induced $1$-weight $\overline c$ on $\Sigma^\tau/\tau$ is balanced with respect to the induced function $\overline\phi$. Note that the balancing condition is independent of the choice of representative $(\Delta,c)$. A balanced $k$-weight is called a \emph{tropical $k$-cycle}.

We say a function $\phi\in \sPL_\Sigma(\Sigma)$ is \emph{combinatorially principal at the cone} $\sigma$, and write $\phi \in \CPL(\sigma)$ if there exists $\chi\in \Aff_\Sigma(\Sigma^\sigma)$ with $\chi\vert_\sigma=\phi\vert_\sigma$. We say $\phi$ is \emph{combinatorially principal} if it is combinatorially principal at every cone $\sigma\in \Sigma$, and denote the group of combinatorially principal functions on $\Sigma$ by $\CPL(\Sigma)$. For a function $\phi\in \sPL_\Sigma(\Sigma)$ and a tropical $1$-cycle $A$ represented by $(\Delta,c)$, the intersection product $\phi\cdot A$ is the unique $0$-cycle whose weight at the origin is given by 
\begin{equation}
\label{eq:tropical intersection product}
    -\sum_{\rho\in \Delta(1)} \mathrm{slope}_\rho(\phi)c(\rho) \ ,
\end{equation}
where $-\mathrm{slope}_\rho(\phi)=-\phi(u_\rho)$ is the incoming slope of $\phi$ at the origin, as required by the $\min$ convention.
This only depends on the class of $\phi\in \sPL_\Sigma(\Sigma)/\Aff_\Sigma(\Sigma)$. 
Let $\phi\in \CPL(\Sigma)$ and let $A=[(\Delta,c)]$ be a tropical $k$-cycle on $\Sigma$. The intersection product $\phi\cdot A$ is the tropical cycle represented by the $(k-1)$-weight whose weight on $\tau\in \Delta(k-1)$ is given by the weight at the origin of $\overline \phi \cdot \overline c$, where $\overline c$ is the tropical $1$-cycle  on $\Delta^\tau/\tau$ induced by $c$ and $\overline \phi$ is the function on $\Delta^\tau/\tau$ induced by $\phi-\chi$ for some $\chi\in \Aff_\Sigma(\Sigma)$ with $\chi\vert_\tau=\phi\vert_\tau$. By construction, $\phi\cdot A$ depends only on the class of $\phi$ in $\CPL(\Sigma)/\Aff_\Sigma(\Sigma)$. 

We just presented the local theory of tropical cycles, which is all we need in this paper. These concepts can  be globalized: there exists a sheaf $Z_k$ of tropical $k$-cycles on tropical spaces, a sheaf $\CPL$ (denoted by $\mathrm{Rat}$ in \cite[Definition 6.3]{psi-classes}) of combinatorially principal functions, and an intersection-pairing 
\begin{equation}
\CPL/\Aff \otimes_{\Z}  Z_{k}  \to  Z_{k-1}  
\end{equation}
 that locally looks like the pairing in the case of cone complexes with affine structure. The sheaf $\CPL/\Aff$ is the sheaf of combinatorially principal tropical Cartier divisors, sitting inside the sheaf $\PLfin/\Aff$ of \emph{tropical Cartier divisors}.
 The \emph{$k$-th tropical Chow group $A_k(\TX)$ of a tropical space $\TX$} is given by 
 \begin{equation}
     A_k(\TX)\coloneqq Z_k(\TX)/(\CPL(\TX)\cdot Z_{k+1}(\TX)) \ .
 \end{equation}

\subsection{Tropical line bundles and psi classes}
\label{sec:tlb}

A tropical line bundle on a tropical space $\TX$ is an $\Aff_\TX$-torsor. Equivalently, it is a morphism $\TL\to \TX$ whose fibers are $\T$-torsors and which locally trivializes: it is locally isomorphic to $\TX\times\T$ as tropical spaces, and the isomorphism is an isomorphism of $\T$-torsors on each fiber. One can translate the two notions into each other by associating to a line bundle $\TL\to \TX$ the sheaf $\underline{\mathrm{Isom}}(\T\times \TX, \TL)$ of isomorphisms with the trivial line bundle, which is an $\Aff_\TX$-torsor because $\underline{\mathrm{Aut}}(\T\times \TX)\cong \Aff_\TX$.

By \cite[Lemma 4.5]{Lefschetz}, every tropical line bundle $\TL$ has a piecewise linear section. Any such section defines a tropical Cartier divisor $\mtt D$ whose associated tropical line bundle is isomorphic to $\TL$. In the language of torsors, the existence of a piecewise linear sections means that given any $\Aff_\TX$-torsor, the associated $\PL$-torsor is trivial. However, the associated $\CPL$-torsor is not necessarily trivial, that is the line bundle does not necessarily have a combinatorially principal section (i.e.\ the tropical Cartier divisor $\mtt D$ from above is not necessarily combinatorially principal). If a tropical line bundle $\TL$ on a tropical space $\TX$ has a  combinatorially principal section, any choice of such section defines the same element in $(\CPL/\Aff)(\TX)/\CPL(\TX)$ (a combinatorially principal tropical Cartier divisor modulo linear equivalence), and this element defines a map
\begin{equation}
    c_1(\TL)\colon A_*(\TX)\to A_{*-1}(\TX)
\end{equation}
via the intersection pairing. The map $c_1(\TL)$ is called the \emph{first Chern class of $\TL$}. Given $\alpha\in A_\ast(\TX)$, we denote $c_1(\TL)(\alpha)$ by $c_1(\TL) 
\frown \alpha$ to avoid proliferation of nested parentheses. 

We recall the construction of tropical $\psi$ classes from \cite{psi-classes}. Given a  family of $n$-marked, stable tropical curves $\overline \TC\to \overline \TB$, for every $1\leq i\leq n$ and $k\in \Z$ the  $i$-th section $\ts_i\colon \overline \TB\to \overline \TC$ defines a tropical line bundle $\Aff_{\overline\TC}(k \ts_i)$ in the form of an $\Aff_{\overline \TC}$-torsor  on the total space $\overline\TC$. Namely, every point in the image of $\ts_i$ is the infinite point of a leg in its fiber. Therefore, we can define $\Aff_{\overline\TC}(k\ts_i)$ as the subsheaf of $\iota_*\Aff_{\overline\TC\setminus \ts_i(\overline\TB)}$, where $\iota\colon \overline\TC\setminus \ts_i(\TB)\to \mathtt C$ is the inclusion, of all functions whose slope on the leg approaching $\ts_i$ is $-k$.  The \emph{$i$-th tropical cotangent} line bundle $\bL_i^\trop$ is the line bundle on $\overline\TB$  defined as $\ts_i^*\Aff_{\overline\TC}(-\ts_i)$. 

 If $\TL=\bL_i^\trop$ is the $i$-the cotangent bundle of a marked family of tropical curves, we call the first Chern class of $\TL$ the \emph{$i$-th $\psi$ class} and denote it by 
\begin{equation}
    \psi_i\coloneqq c_1(\bL_i^\trop) \ .
\end{equation}

\section{Affine Structures for Tropicalizations}
\label{sec:affstrtrop}

This section contains the core technical definitions and constructions. Given a toroidal variety $X$, we endow its cone complex with an affine structure, calling the result the tropicalization of $X$. Tropicalization becomes a functor valued in the category of tropical spaces. To simplify the exposition, we initially make the following seemingly strong assumption: 
\begin{quote}
   \hspace{-1cm}$(\ast)$ \hspace{0.32cm} there are no self-intersections among the toroidal strata or multiple intersections among pairs of strata.
\end{quote} 
We prove that the affine structure induced by tropicalization is invariant under log modifications and explain in Remark \ref{rem:allowselfintersections} how this allows us to remove such hypothesis without loss of generality.

\subsection{Affine structures on cone complexes}
\label{sec:affstrconcomplex}

Let $\Sigma$ be a cone complex. A \emph{subcomplex} $S$ of $\Sigma$ is a union
\begin{equation}S = \bigcup_{\sigma\in I}\sigma^\diamond,\end{equation}
where $I$ is a subset of cones of $\Sigma$. We say that $S$ is an open (resp.\ closed) subcomplex of $\Sigma$   if it defines an open (resp.\ closed) set. 

For a cone  $\sigma\in \Sigma$, we define the \emph{open star}  $\Sigma^\sigma$ of $\sigma$ to be the open subcomplex
\begin{equation}\label{openstar}
    \Sigma^\sigma= \bigcup_{\sigma\subseteq \tau} \tau^\diamond \ ,
\end{equation}
given by the union of the relative interiors of all cones containing $\sigma$.
The usual \emph{star} of $\sigma$, consisting of the union of all cones containing $\sigma$, is the closure of the open star, and  we denote it by $\overline{\Sigma^\sigma}.$\footnote{We could just as well have defined the open star as the interior of the star, and adapted the notation accordingly. We chose this order because the open star, which provides a natural open neighborhood of the generic point of a cone, plays a more prominent role in the definition of sheaves of affine functions. }

Let $X$ be a toroidal variety satisfying $(\ast)$, and $\Sigma_X$ its cone complex as described in Section \ref{sec:ccte}. Every $\sigma\in\Sigma_X$ corresponds to a stratum $O(\sigma)$. 
   
For a cone $\sigma\in \Sigma_X$  we denote by $X_\sigma$ the open neighborhood of $V(\sigma)$ given by
\begin{equation}
    X_\sigma=\bigcup_{\tau\in \overline{\Sigma_X^\sigma}} O(\tau) \ .
\end{equation}

A strict piecewise linear function $\phi$ on $\Sigma_X^\sigma$ determines a Cartier divisor, and consequently an invertible sheaf on $X_\sigma$:  extend $\phi$ to $\overline{\Sigma_X^\sigma}$ by linearity, and for any ray $\rho\in \overline{\Sigma_X^\sigma}$ denote the slope of $\phi|_{\rho}$ by $\mathrm{slope}_{\rho}(\phi)$. 
We define:
\begin{align}
     D_{X_\sigma}(\phi) = \sum_{\rho\in \overline{\Sigma_X^\sigma}} \mathrm{slope}_{\rho}(\phi) V(\rho)\ ,
    & &
    \mc O_{X_{\sigma}}(\phi)=
    \mc O_{X_{\sigma}}\left(D_{X_\sigma}(\phi)\right).
\end{align}

\begin{definition}
\label{def:affine function}
Let $\Sigma_X$ be the cone complex of a toroidal variety $X$, and $\sigma\in\Sigma_X$.

\noindent An \textbf{affine function} at $\sigma$  is a strict piecewise linear function $\phi$ on $\Sigma^\sigma$ such that $\mc O_{X_{\sigma}}(\phi)\vert _{V(\sigma)}$ is trivial.

\noindent The \textbf{affine structure} $\Aff_{\Sigma_X}$ on $\Sigma_X$ induced by $X$ is defined to be the subsheaf of $\PLfin_{\Sigma_X}$ generated by 
 affine functions at cones $\sigma\in{\Sigma_X}$.
\end{definition}

As shown in the following lemma, affine functions in the sense of Definition \ref{def:affine function}  are closed under restriction, hence we are not generating any new sections of $\Aff_\Sigma$ by restriction.

\begin{lemma}
\label{lem: compatibility of affine functions}
Let $\sigma\in\Sigma$, let $\tau$ be a face of $\sigma$, and let $\phi$ be an affine function at $\tau$. Then $\phi\vert_{\Sigma^\sigma}$ is an affine function at $\sigma$.
\end{lemma}

\begin{proof}
Since  $V(\sigma)$ is contained in $V(\tau)$, $\mc O_{X_\tau}(\phi)\vert _{V(\tau)}$ being trivial implies that the bundle $\mc O_{X_\tau}(\phi)\vert _{V(\sigma)}$ is trivial as well. The assertion now follows from the fact that 
\begin{equation}
\calO_{X_\tau}(\phi)\vert_{X_\sigma}=\calO_{X_\sigma}(\phi\vert_{\Sigma_\sigma}) \ .
\end{equation}
\end{proof}

\begin{corollary}
\label{cor:affine functions near sigma are affine functions at sigma}
Let $\sigma\in \Sigma$. Then $ \Aff_\Sigma(\Sigma^\sigma)$ consists precisely of affine functions at $\sigma$. 
\end{corollary}

\begin{proof}
This is an immediate consequence of Lemma \ref{lem: compatibility of affine functions}.
\end{proof}

\begin{remark}
This notion of affine functions is related to the linear functions on Cartwright's tropical complexes \cite{TropicalComplexes}. To compare the two notions, consider the situation where we are given a proper morphism $\pi\colon X\to \mathds A^1$ such that $X$ is smooth and connected and $\pi^{-1}\{0\}$ is a reduced simple normal crossings divisor. In particular, the pair $(\pi^{-1}( \mathds A^1\setminus \{0\}),X)$ defines a toroidal embedding and $\pi$ is a dominant toroidal morphism. We obtain an induced morphism $\Pi\colon \Sigma_X\to \R_{\geq 0}$ of cone complexes and the linear functions of \cite{TropicalComplexes} are defined on the fiber $\Delta\coloneqq \Pi^{-1}\{1\}$. To be linear in the sense of Cartwright, a function $\chi$ on $\Delta$ has to satisfy the condition that for each $\sigma\in \Sigma$ with $V(\sigma)$ a curve, there exists a function $\phi\in \sPL(\Sigma^\sigma)$ with
\begin{align}
    \phi\vert_{\Delta\cap \Sigma^\sigma}
        &=\chi\vert_{\Delta\cap\Sigma^\sigma}    \qquad\text{, and}\\
    \deg(\mc O_{X_\sigma}(\phi)\vert_{V(\sigma)}) 
        &= 0 \ .
\end{align}
 So Cartwright's linearity is only checked locally around $\sigma\in \Sigma$ with $V(\sigma)$ a curve, and our condition that $\mc O_{X_\sigma}(\phi)\vert_{V(\sigma)}$ be trivial is weakened to it having degree $0$. In other words, our notion of affine functions agrees with Cartwright's locally around $\sigma\in \Sigma$ with $V(\sigma)\cong \PP^1$, and differs from it otherwise.
\end{remark}

We now show that the assignment of tropical spaces $\Sigma_X$  to toroidal varieties $X$ is functorial. 

\begin{proposition}
\label{prop:functoriality}
Let $f: X\to Y$ be a dominant morphism of toroidal varieties. Then the natural piecewise linear map
\begin{equation}
    \trof\coloneqq \Trop(f) \colon \vert \Sigma_X\vert\to \vert \Sigma_Y\vert
\end{equation}
is a morphism of tropical spaces.
\end{proposition}

\begin{proof}
Let $\tau\in \Sigma_Y$ and let $\phi\in \Aff_{\Sigma_Y}(\Sigma^\tau_Y)$ be an affine function at $\tau$.  Let $\sigma$ be a cone of $\trof^{-1}\Sigma_Y^\tau$. We want to show that $\trof^\ast(\phi)\vert_{\Sigma_X^\sigma}$ is an affine function at $\sigma$, i.e.\ $\trof^\ast(\phi)\vert_{\Sigma_X^\sigma}\in \Aff_{\Sigma_X}(\Sigma^\sigma_X)$.
After potentially restricting $\phi$ to the star of a cone containing $\tau$, we may assume by Lemma \ref{lem: compatibility of affine functions} that $\trof(\sigma^\diamond)\subseteq \tau^\diamond$.
Then the morphism $f$ maps $V(\sigma)$ into $V(\tau)$ and thus induces a morphism $f\vert_{V(\sigma)}\colon V(\sigma)\to V(\tau)$. We have
\begin{equation}
\mc O_{X_\sigma}(\trof^*(\phi))\vert_{V(\sigma)} \cong 
    \left(f\vert_{X_\sigma}^* \mc O_{Y_\tau}(\phi)\right)\vert_{V(\sigma)} \cong 
        f\vert_{V(\sigma)}^* \left(  \mc O_{Y_\tau}(\phi)\vert_{V(\tau)} \right)
            \cong f\vert_{V(\sigma)}^* \left(\mc O_{V(\tau)}\right)\cong \mc O_{V(\sigma)} \ ,
\end{equation}
where the second to last isomorphism follows from $\phi$ being affine. Hence $\trof^\ast(\phi)$ is affine in a neighborhood of $\sigma$. 
\end{proof}

\begin{remark}
In Proposition \ref{prop:functoriality}, the assumption that $f$ is a dominant toroidal morphism is too strong. All that is needed is that $f$ defines a morphism of log schemes, that is that $f(X_0)\subseteq Y_0$.
\end{remark}

In the proof of the invariance of the affine structure under modifications, we  use the following lemma.

\begin{lemma}
\label{lem:pull-back under modification is affine}
Let $f\colon X\to Y$ be a log modification, let $\tau\in  \Sigma_X$, and let $\sigma\in  \Sigma_Y$ be the unique cone with $\tau^\diamond \subseteq \sigma^\diamond$.  If  $\phi\in \sPL_{ \Sigma_Y}(\Sigma_Y^\sigma)$ is a piecewise linear function which restricts to an affine function $\phi\vert_{\Sigma_X^\tau}\in \Aff_{\Sigma_X}(\Sigma_X^\tau)$,  then we have $\phi\in \Aff_{\Sigma_Y}(\Sigma_Y^\sigma)$.
\end{lemma}

\begin{proof}
Let $\theta\in\Sigma_X$ be an inclusion-maximal cone with $\tau\subseteq \theta\subseteq \sigma$, note that $\theta$ and $\sigma$ have the same dimension.  The induced morphism $\Sigma_X^\theta/\theta\to \Sigma_Y^\sigma/\sigma$ is a subdivision, and the corresponding logarithmic modification agrees with the restriction $f_\theta:= f\vert_{V(\theta)} \colon V(\theta)\to V(\sigma)$. By Corollary \ref{cor:affine functions near sigma are affine functions at sigma}, the function $\phi\vert_{\Sigma_X^\theta}$ is affine at $\theta$ and hence
\begin{equation}
    \mc O_{V(\theta)}=\mc O_{X_\theta}(\phi\vert_{\Sigma_X^\theta})\vert_{V(\theta)} \cong f_\theta^*\left(\mc O_{Y_\sigma}(\phi)\vert_{V(\sigma)} \right)
\end{equation}
Taking first Chern classes and using the projection formula, we conclude that 
\begin{equation}
    \mc O_{Y_\sigma}(\phi)\vert_{V(\sigma)} \cong \mc O_{V(\sigma)} \ , 
\end{equation}
which is equivalent to $\phi\in \Aff_{\Sigma_Y}(\Sigma_Y^\sigma)$.
\end{proof}

\begin{proposition}
\label{prop:invariance of affine structure under refinement}
If $f: X\to Y$ is a log modification then the induced morphism 
\begin{equation}
    \trof \colon \vert \Sigma_X\vert\to \vert \Sigma_Y\vert
\end{equation}
is an isomorphism of tropical spaces.
\end{proposition}

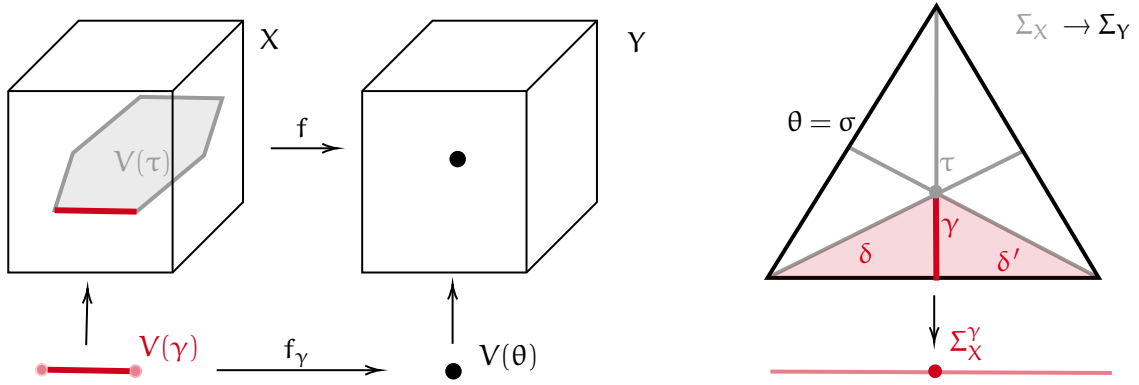
\begin{figure}
    \centering
  
\tikzset{every picture/.style={line width=0.75pt}} 
\begin{tikzpicture}[x=0.75pt,y=0.75pt,yscale=-.7,xscale=.7]

\draw [color={rgb, 255:red, 155; green, 155; blue, 155 }  ,draw opacity=1 ][line width=1.5]    (691.79,11.57) -- (691.71,209.43) ;
\draw [color={rgb, 255:red, 155; green, 155; blue, 155 }  ,draw opacity=1 ][line width=1.5]    (628.71,113.43) -- (808.71,207.57) ;
\draw [color={rgb, 255:red, 155; green, 155; blue, 155 }  ,draw opacity=1 ][line width=1.5]    (570,207.57) -- (753.71,116.43) ;
\draw  [line width=1.5]  (691.79,11.57) -- (808.71,207.57) -- (570,207.57) -- cycle ;
\draw  [draw opacity=0][fill={rgb, 255:red, 208; green, 2; blue, 27 }  ,fill opacity=0.15 ] (808.71,207.57) -- (570,207.57) -- (693.14,146.57) -- cycle ;
\draw [color={rgb, 255:red, 208; green, 2; blue, 27 }  ,draw opacity=1 ][line width=2.25]    (691.31,145.4) -- (691.71,209.43) ;
\draw    (690.14,220.14) -- (690.14,252.14) ;
\draw [shift={(690.14,254.14)}, rotate = 270] [color={rgb, 255:red, 0; green, 0; blue, 0 }  ][line width=0.75]    (10.93,-3.29) .. controls (6.95,-1.4) and (3.31,-0.3) .. (0,0) .. controls (3.31,0.3) and (6.95,1.4) .. (10.93,3.29)   ;
\draw  [color={rgb, 255:red, 208; green, 2; blue, 27 }  ,draw opacity=1 ][fill={rgb, 255:red, 208; green, 2; blue, 27 }  ,fill opacity=1 ] (686.57,274.93) .. controls (686.57,272.6) and (688.46,270.71) .. (690.79,270.71) .. controls (693.11,270.71) and (695,272.6) .. (695,274.93) .. controls (695,277.26) and (693.11,279.14) .. (690.79,279.14) .. controls (688.46,279.14) and (686.57,277.26) .. (686.57,274.93) -- cycle ;
\draw [color={rgb, 255:red, 208; green, 2; blue, 27 }  ,draw opacity=0.5 ][fill={rgb, 255:red, 208; green, 2; blue, 27 }  ,fill opacity=1 ][line width=1.5]    (572.14,275.14) -- (818.14,276.14) ;
\draw  [color={rgb, 255:red, 155; green, 155; blue, 155 }  ,draw opacity=1 ][fill={rgb, 255:red, 155; green, 155; blue, 155 }  ,fill opacity=1 ] (687.1,145.4) .. controls (687.1,143.07) and (688.99,141.19) .. (691.31,141.19) .. controls (693.64,141.19) and (695.53,143.07) .. (695.53,145.4) .. controls (695.53,147.73) and (693.64,149.61) .. (691.31,149.61) .. controls (688.99,149.61) and (687.1,147.73) .. (687.1,145.4) -- cycle ;
\draw    (212.14,114.14) -- (258.14,114.14) ;
\draw [shift={(260.14,114.14)}, rotate = 180] [color={rgb, 255:red, 0; green, 0; blue, 0 }  ][line width=0.75]    (10.93,-3.29) .. controls (6.95,-1.4) and (3.31,-0.3) .. (0,0) .. controls (3.31,0.3) and (6.95,1.4) .. (10.93,3.29)   ;
\draw    (79.14,256.14) -- (80.09,217.14) ;
\draw [shift={(80.14,215.14)}, rotate = 91.4] [color={rgb, 255:red, 0; green, 0; blue, 0 }  ][line width=0.75]    (10.93,-3.29) .. controls (6.95,-1.4) and (3.31,-0.3) .. (0,0) .. controls (3.31,0.3) and (6.95,1.4) .. (10.93,3.29)   ;
\draw   (278,72.89) -- (328.74,22.14) -- (447.14,22.14) -- (447.14,152.97) -- (396.4,203.71) -- (278,203.71) -- cycle ; \draw   (447.14,22.14) -- (396.4,72.89) -- (278,72.89) ; \draw   (396.4,72.89) -- (396.4,203.71) ;
\draw  [color={rgb, 255:red, 155; green, 155; blue, 155 }  ,draw opacity=1 ][fill={rgb, 255:red, 155; green, 155; blue, 155 }  ,fill opacity=0.2 ][line width=1.5]  (69.57,117.63) -- (118.18,77.07) -- (177.22,78.1) -- (164.04,119.28) -- (115.43,159.83) -- (56.39,158.8) -- cycle ;
\draw   (23,72.89) -- (73.74,22.14) -- (192.14,22.14) -- (192.14,152.97) -- (141.4,203.71) -- (23,203.71) -- cycle ; \draw   (192.14,22.14) -- (141.4,72.89) -- (23,72.89) ; \draw   (141.4,72.89) -- (141.4,203.71) ;
\draw  [fill={rgb, 255:red, 0; green, 0; blue, 0 }  ,fill opacity=1 ] (341.57,121.93) .. controls (341.57,119.05) and (343.91,116.71) .. (346.79,116.71) .. controls (349.67,116.71) and (352,119.05) .. (352,121.93) .. controls (352,124.81) and (349.67,127.14) .. (346.79,127.14) .. controls (343.91,127.14) and (341.57,124.81) .. (341.57,121.93) -- cycle ;
\draw [color={rgb, 255:red, 208; green, 2; blue, 27 }  ,draw opacity=1 ][line width=2.25]    (56.39,158.8) -- (115.43,159.83) ;
\draw [color={rgb, 255:red, 208; green, 2; blue, 27 }  ,draw opacity=1 ][line width=2.25]    (51.39,273.8) -- (110.43,274.83) ;
\draw  [color={rgb, 255:red, 208; green, 2; blue, 27 }  ,draw opacity=0.3 ][fill={rgb, 255:red, 208; green, 2; blue, 27 }  ,fill opacity=0.5 ] (110.43,274.83) .. controls (110.43,272.51) and (112.32,270.62) .. (114.65,270.62) .. controls (116.97,270.62) and (118.86,272.51) .. (118.86,274.83) .. controls (118.86,277.16) and (116.97,279.05) .. (114.65,279.05) .. controls (112.32,279.05) and (110.43,277.16) .. (110.43,274.83) -- cycle ;
\draw  [color={rgb, 255:red, 208; green, 2; blue, 27 }  ,draw opacity=0.2 ][fill={rgb, 255:red, 208; green, 2; blue, 27 }  ,fill opacity=0.5 ] (42.96,273.8) .. controls (42.96,271.48) and (44.85,269.59) .. (47.18,269.59) .. controls (49.5,269.59) and (51.39,271.48) .. (51.39,273.8) .. controls (51.39,276.13) and (49.5,278.02) .. (47.18,278.02) .. controls (44.85,278.02) and (42.96,276.13) .. (42.96,273.8) -- cycle ;
\draw  [fill={rgb, 255:red, 0; green, 0; blue, 0 }  ,fill opacity=1 ] (338.57,274.93) .. controls (338.57,272.05) and (340.91,269.71) .. (343.79,269.71) .. controls (346.67,269.71) and (349,272.05) .. (349,274.93) .. controls (349,277.81) and (346.67,280.14) .. (343.79,280.14) .. controls (340.91,280.14) and (338.57,277.81) .. (338.57,274.93) -- cycle ;
\draw    (343.14,255.14) -- (343.14,214.14) ;
\draw [shift={(343.14,212.14)}, rotate = 90] [color={rgb, 255:red, 0; green, 0; blue, 0 }  ][line width=0.75]    (10.93,-3.29) .. controls (6.95,-1.4) and (3.31,-0.3) .. (0,0) .. controls (3.31,0.3) and (6.95,1.4) .. (10.93,3.29)   ;
\draw    (173.14,274.14) -- (289.14,274.14) ;
\draw [shift={(291.14,274.14)}, rotate = 180] [color={rgb, 255:red, 0; green, 0; blue, 0 }  ][line width=0.75]    (10.93,-3.29) .. controls (6.95,-1.4) and (3.31,-0.3) .. (0,0) .. controls (3.31,0.3) and (6.95,1.4) .. (10.93,3.29)   ;

\draw (746,15.4) node [anchor=north west][inner sep=0.75pt]    {$\textcolor[rgb]{0.61,0.61,0.61}{\Sigma_X\ } \to \Sigma_Y$};
\draw (360,252.4) node [anchor=north west][inner sep=0.75pt]    {$V(\theta)$};
\draw (691.75,118.9) node [anchor=north west][inner sep=0.75pt]  [color={rgb, 255:red, 155; green, 155; blue, 155 }  ,opacity=1 ]  {$\textcolor[rgb]{0.61,0.61,0.61}{\tau}$};
\draw (633.57,180.47) node [anchor=north west][inner sep=0.75pt]    {$\textcolor[rgb]{0.82,0.01,0.11}{\delta}$};
\draw (733,184.4) node [anchor=north west][inner sep=0.75pt]    {$\textcolor[rgb]{0.82,0.01,0.11}{\delta'}$};
\draw (693.75,162.9) node [anchor=north west][inner sep=0.75pt]    {$\textcolor[rgb]{0.82,0.01,0.11}{\gamma}$};
\draw (698.75,240.9) node [anchor=north west][inner sep=0.75pt]    {$\textcolor[rgb]{0.82,0.01,0.11}{\Sigma_X^\gamma}$};
\draw (114.75,246.4) node [anchor=north west][inner sep=0.75pt]    {$\textcolor[rgb]{0.82,0.01,0.11}{V(\gamma)}$};
\draw (96.75,112.9) node [anchor=north west][inner sep=0.75pt]  [color={rgb, 255:red, 155; green, 155; blue, 155 }  ,opacity=1 ]  {$\textcolor[rgb]{0.61,0.61,0.61}{V(\tau)}$};
\draw (582,85.4) node [anchor=north west][inner sep=0.75pt]    {$\theta = \sigma$};
\draw (202,26.11) node [anchor=north west][inner sep=0.75pt]    {$X$};
\draw (467,28.11) node [anchor=north west][inner sep=0.75pt]    {$Y$};
\draw (219,249.11) node [anchor=north west][inner sep=0.75pt]    {$f_\gamma$};
\draw (228,89.4) node [anchor=north west][inner sep=0.75pt]    {$f$};

\end{tikzpicture}

\caption{The setup used in showing that a function $\phi\in  \Aff_{ \Sigma_X}( \Sigma_X^\tau)$ is the restriction of a function in $\Aff_{\Sigma_Y}( \Sigma_Y^\sigma)$. The key point is that the map $f_\gamma$ is a $\mathbb{P}^1$ bundle. In the illustration we draw $\theta = \sigma$ for lack of dimensions: $\sigma$ is the minimal cone of $\Sigma_Y$ containing $\tau$, and $\theta$ a maximal cone with such property.  }
    \label{fig:enter-label}
\end{figure}

\begin{proof}
Since $f$ is a log modification, we know $\trof$ is a bijective function, and by Proposition \ref{prop:functoriality} affine functions pull-back to affine functions. It remains to show that any affine function on $\Sigma_X$ can be obtained by pull-back from an affine function of $\Sigma_Y$.
Let $\tau\in  \Sigma_X$ and let $\sigma\in \Sigma_Y$ be the unique cone with $\tau^\diamond \subseteq \sigma^\diamond$. We must prove that every function $\phi\in  \Aff_{ \Sigma_X}( \Sigma_X^\tau)$ is the restriction of a function in $\Aff_{\Sigma_Y}( \Sigma_Y^\sigma)$.

First we show that $\phi$ is the restriction of a function in $\sPL_{ \Sigma_Y}(\Sigma_Y^\sigma)$, see Figure \ref{fig:enter-label} for the setup. Let $\theta\in  \Sigma_Y$ be a maximal cone containing $\sigma$ and let $\delta\in  \Sigma_X$ be a maximal cone of $ \Sigma_X$ with $\tau\subseteq \delta\subseteq \theta$. Since the cones $\delta$ and $\theta$ have the same dimension, the restriction $\phi\vert_\delta$ extends uniquely to a linear function $m_\theta^\delta$  on the cone $\theta$. We now show that this function does not depend on the choice of $\delta$. Let $\delta'\in  \Sigma_X$ be a second maximal cone with $\tau\subseteq \delta'\subseteq \theta$. Because $\delta$ can be reached from $\delta'$ by passing through codimension-$1$ cones of $ \Sigma_X$ that contain $\tau$ and are contained in $\theta$, we may assume that $\delta$ and $\delta'$ are adjacent cones, and $\gamma\coloneqq \delta\cap \delta'$ has codimension $1$.  Let $\varphi=\phi\vert_{\Sigma_X^\gamma}-m_\theta^\delta\vert_{\Sigma_X^\gamma}$. Then 
\begin{equation}\label{eq:agreeoncod1wall}
    \mc O_{X_\gamma}(\phi\vert_{\Sigma_X^\gamma}) \cong \mc O_{X_\gamma}(\varphi) \otimes f\vert_{X_\gamma}^* \mc O_{Y_\theta}(m_\theta^\delta) \ .
\end{equation}
Denote by $f_\gamma \colon V(\gamma)\to V(\theta)$ the restriction of $f$. Restricting \eqref{eq:agreeoncod1wall} to $V(\gamma)$ and using the hypothesis of $\phi$ being affine we obtain
\begin{equation}\label{eq:eqongamma}
    \mc O_{V(\gamma)} \cong 
    \mc O_{X_\gamma}(\varphi)\vert_{V(\gamma)} \otimes f_\gamma^* \mc O_{Y_\theta}(m_\theta^\delta)\vert_{V(\theta)} \ .
\end{equation}
The function $\varphi$ vanishes on $\gamma$ and therefore induces a function $\overline\varphi$ on the star $ \Sigma_X^\gamma/\gamma=\Sigma_{V(\gamma)}$.  
Since $\gamma$ has codimension $1$ in $\theta$, the star $\Sigma_X^\gamma/\gamma$ is a union of two rays,  corresponding to the divisors $V(\delta)$ and $V(\delta')$ in $V(\gamma)$, which both map isomorphically onto $V(\theta)$. That makes $f_\gamma$  a $\PP^1$-bundle on $V(\theta)$ with two distinct sections. 
The function $\overline \varphi$ has slope $0$ on the ray corresponding to $V(\delta)$, we show that it must have slope $0$ along the other ray as well. If $k$ is the slope of $\overline \varphi$ on the  ray corresponding to $V(\delta')$, then  
\begin{equation}
    \mc O_{X_\gamma}(\varphi)\vert_{V(\gamma)} 
    \cong \mc O_{V(\gamma)}(\overline \varphi)
    \cong \mc O_{V(\gamma)}(V(\delta'))^{\otimes k}
\end{equation}
and hence \eqref{eq:eqongamma} gives
\begin{equation}
    \mc O_{V(\gamma)} \cong 
    \mc O_{V(\gamma)}(V(\delta'))^{\otimes k} \otimes f_\gamma^* \mc O_{Y_\theta}(m_\theta^\delta)\vert_{V(\theta)} \ .
\end{equation}
Since $\Pic(V(\gamma))= \Z\cdot \mc O_{V(\gamma)}(V(\delta'))  \oplus \Pic(V(\theta))$
by virtue of the projective bundle formula, it follows that $k=0$, i.e.\ $\varphi\vert_{\delta'}=0$. From this it follows immediately  that $m^\delta_\theta=m^{\delta'}_\theta$.

For $\theta\in  \Sigma_Y$ a maximal cone containing $\sigma$ we can now define $m_\theta=m^\delta_\theta$ for some maximal cone $\delta\in \Sigma_X$ with $\tau\subseteq\delta\subseteq\theta$;  this is independent of the choice of $\delta$. Given a second maximal cone $\theta'\in  \Sigma_Y$ containing $\sigma$ intersecting $\theta$ in $\gamma\coloneqq\theta\cap \theta'$, we claim that 
\begin{equation}
    m_\theta\vert_\gamma=m_\theta'\vert_\gamma \ .
\end{equation}
Referring to Figure \ref{fig:extendtwo}, if $\widetilde\gamma\in  \Sigma_X$ is a maximal cone in $\gamma$ containing $\tau$, then there exist maximal cones $\delta,\delta'\in  \Sigma_X$ containing $\widetilde\gamma$ with $\delta\subseteq \theta$ and $\delta'\subseteq \theta'$. Then we have
\begin{equation}
    m_\theta\vert_{\widetilde\gamma}=m_\theta^\delta\vert_{\widetilde\gamma}=\phi\vert_{\widetilde\gamma}=m_{\theta'}^{\delta'}\vert_{\widetilde\gamma}=m_{\theta'}\vert_{\widetilde\gamma} \ .
\end{equation}
As $\widetilde\gamma$ is full-dimensional in $\gamma$, this shows the desired equality. We conclude that the functions $m_\theta$ glue to give a function $\widetilde \phi\in \sPL_{ \Sigma_Y}(\Sigma_Y^\sigma)$, and by construction we have $\widetilde\phi\vert_{\Sigma_X^\tau}=\phi$. By Lemma \ref{lem:pull-back under modification is affine}, we have $\widetilde\phi\in \Aff_{ \Sigma_Y}(\Sigma_Y^\sigma)$, completing the proof.
\end{proof}

\begin{figure}
    \centering
  
\tikzset{every picture/.style={line width=0.75pt}} 

\begin{tikzpicture}[x=0.75pt,y=0.75pt,yscale=-.9,xscale=.9]

\draw [color={rgb, 255:red, 208; green, 2; blue, 27 }  ,draw opacity=1 ][line width=3.5]    (239.14,60.57) -- (240.07,169.07) ;
\draw [line width=1.5]    (240.07,60.07) -- (242.14,256.57) ;
\draw [line width=1.5]    (239.14,60.57) -- (390.14,171.43) ;
\draw [line width=1.5]    (242.14,256.57) -- (390.14,171.43) ;
\draw [line width=1.5]    (93.14,166.57) -- (242.14,256.57) ;
\draw [line width=1.5]    (93.14,166.57) -- (239.14,60.57) ;
\draw [color={rgb, 255:red, 155; green, 155; blue, 155 }  ,draw opacity=1 ][line width=1.5]    (93.14,166.57) -- (240.07,169.07) ;
\draw  [draw opacity=0][fill={rgb, 255:red, 208; green, 2; blue, 27 }  ,fill opacity=0.2 ] (390.14,171.43) -- (93.14,166.57) -- (239.14,60.57) -- cycle ;
\draw [color={rgb, 255:red, 155; green, 155; blue, 155 }  ,draw opacity=1 ][line width=1.5]    (241.64,169) -- (388.57,171.5) ;
\draw  [color={rgb, 255:red, 155; green, 155; blue, 155 }  ,draw opacity=1 ][fill={rgb, 255:red, 155; green, 155; blue, 155 }  ,fill opacity=1 ][line width=1.5]  (235,169.07) .. controls (235,166.27) and (237.27,164) .. (240.07,164) .. controls (242.87,164) and (245.14,166.27) .. (245.14,169.07) .. controls (245.14,171.87) and (242.87,174.14) .. (240.07,174.14) .. controls (237.27,174.14) and (235,171.87) .. (235,169.07) -- cycle ;

\draw (288,52.4) node [anchor=north west][inner sep=0.75pt]    {$\ \textcolor[rgb]{0.61,0.61,0.61}{\Sigma_X}\to \Sigma_Y$};
\draw (247.14,172.47) node [anchor=north west][inner sep=0.75pt]    {$\textcolor[rgb]{0.61,0.61,0.61}{\tau}$};
\draw (244,200.4) node [anchor=north west][inner sep=0.75pt]    {$\sigma =\gamma$};
\draw (243,94.4) node [anchor=north west][inner sep=0.75pt]    {$\textcolor[rgb]{0.82,0.01,0.11}{\tilde{\gamma}}$};
\draw (176,141.4) node [anchor=north west][inner sep=0.75pt]    {$\textcolor[rgb]{0.82,0.01,0.11}{\delta}$};
\draw (276,139.4) node [anchor=north west][inner sep=0.75pt]    {$\textcolor[rgb]{0.82,0.01,0.11}{\delta'}$};
\draw (52,159.4) node [anchor=north west][inner sep=0.75pt]    {$\theta$};
\draw (402,163.4) node [anchor=north west][inner sep=0.75pt]    {$\theta'$};

\end{tikzpicture}

\caption{Illustration for the second part of the argument in Proposition \ref{prop:invariance of affine structure under refinement}. Here by lack of dimensions we have drawn $\sigma = \gamma$. Observe that $\tilde{\gamma} \subset \gamma$, and the cones $\theta$ and $\theta'$ are the two-dimensional triangles bounded by the black edges.}
    \label{fig:extendtwo}
\end{figure}
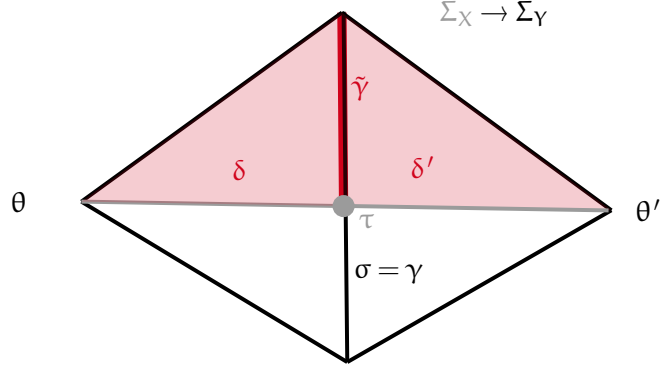

\begin{remark}\label{rem:allowselfintersections}
Proposition \ref{prop:invariance of affine structure under refinement} allows us to remove the hypothesis that the boundary of $X$ does not have strata with self intersections or pairs of strata with multiple intersections. Hypothesis $(\ast)$ may always be achieved via some log modification, e.g.\ by subdividing barycentrically the generalized cone complex $\Sigma_X$ of the original space. We then define the affine structure on the tropicalization as the affine structure on any refinement where $(\ast)$ holds; by Proposition \ref{prop:invariance of affine structure under refinement}, it is a well defined  affine structure on $\Sigma_X$.
\end{remark}

\subsection{The affine structure at infinity}

In this section, we extend the previous constructions to the context of extended cone complexes. Given a cone complex $\Sigma$, its extended cone complex is set-theoretically described as a disjoint union of quotients of cone complexes indexed by the cones of $\Sigma$:
\begin{equation}
\overline\Sigma = \bigsqcup_{\sigma\in \Sigma}\Sigma^\sigma/\sigma \ .
\end{equation}
Observe that if $0_\Sigma$ is the cone point of $\Sigma$, then $\Sigma^{0_\Sigma}/0_\Sigma = \Sigma$. 
We call sets of the form 
$\sigma/\tau\subseteq \overline\Sigma$ for some $\sigma\supseteq\tau\in \Sigma$ the \emph{cells} of the extended cone complex.
The \emph{open star} of a cell of an extended cone complex is 
\begin{equation}\label{openstaratinfty}
    \overline\Sigma^{\sigma/\tau}= \bigcup_{
    \substack{\scriptscriptstyle\tilde\sigma \supseteq \sigma\\
    \scriptscriptstyle\tilde\tau \subseteq \tau}
    } (\tilde\sigma/\tilde\tau)^\diamond \ .
\end{equation}
Taking $\tau = 0_\Sigma$, \eqref{openstaratinfty} agrees with the earlier notion of open star of a cone in a cone complex \eqref{openstar}. We will denote by $\sPL(\overline\Sigma^{\sigma/\tau})$\label{page:sPL in extended cone complex} the subgroup of $\sPL(\Sigma^\sigma)$ consisting of those functions whose continuous extensions to $\overline{\Sigma^\sigma}$ are constant on $\tau$. Equivalently, the functions in $\sPL(\overline\Sigma^{\sigma/\tau})$ are precisely those functions in $\sPL(\Sigma^\sigma)$ that extend continuously to $\overline\Sigma^{\sigma/\tau}$.

\begin{definition}
\label{def:affine structure at infinity}
Let $X$ be a toroidal variety, and let $\overline{\Sigma}_X$ denote the extended cone complex associated to $X$. 
We extend the affine structure on $\Sigma_X$ to $\overline{\Sigma}_X$ by defining
\begin{equation}
    \Aff_{\overline\Sigma_X}= \iota_*\Aff_{\Sigma_X}\cap \PLfin_{\overline\Sigma_X} \ ,
\end{equation}
where $\iota\colon \Sigma_X\to \overline\Sigma_X$ denotes the inclusion, and the intersection is taken in
$\iota_*\PLfin_{\Sigma_X}$. {In other words, the affine functions on $\overline\Sigma_X$ are precisely the continuous extension of the affine functions on $\Sigma_X$.}
\end{definition}

An important feature of extended tropicalization is that \emph{at infinity} one can witness a stratification by the ${\Sigma}^\sigma/\sigma$'s, which is parallel (rather than opposite) to the original toroidal stratification. 
The idea is that the geometry at infinity in the direction of $\tau$ should reflect the geometry of $V(\tau)$. For a cell at infinity of the the form $\sigma/\tau$,  we define $V(\sigma/\tau)=V(\sigma)$, which  we  think of as a closed stratum inside $V(\tau)$. We  show that affine structures are compatible with this perspective.

\begin{lemma}
\label{lem: relation between normal bundle functions}
Let $\sigma\in\Sigma_X$, let $\tau$ be a face of $\sigma$, and let $\phi\in \sPL(\Sigma_X^\sigma)$ such that $\phi\vert_\tau$ is constant. Denote by 
$\pi_\tau\colon \Sigma_X^\tau\to \Sigma_X^\tau/\tau$ the projection and by $\chi$ the unique piecewise linear function  on $(\Sigma_X^\tau/\tau)^{\sigma/\tau}$  such that $\phi=\pi^*_\tau \chi$. Then $\phi$ is affine at $\sigma$ if and only if $\chi$ is affine at $\sigma/\tau$ in the affine structure on $\Sigma_X^\tau/\tau$ defined using the natural identification $\Sigma_X^\tau/\tau\cong \Sigma_{V(\tau)}$. 
\end{lemma}

\begin{proof}
As $\phi$ is constant on $\tau$, the support of the Cartier divisor $D_{X_\sigma}(\phi)$ is does not contain $V(\tau)$, hence its restriction to $V(\tau)$ is transversal: 
\begin{equation}
    D_{X_\sigma}(\phi)\vert_{V(\tau)_{\sigma/\tau}}=D_{V(\tau)_{\sigma/\tau}}(\chi) \ .
\end{equation}
It follows that $\mc O_{X_\sigma}(\phi)\vert_{V(\tau)_{\sigma/\tau}}$  agrees with $\mc O_{V(\tau)_{\sigma/\tau}}(\chi)$, and, in particular, their restriction to $V(\sigma)$ agree, from which the assertion follows immediately. 
\end{proof}

\begin{proposition}
\label{prop:affine functions at infinity}
Let $X$ be a toroidal variety with tropicalization $\Sigma_X$, and let $\tau\in  \Sigma_X$. With the natural identification $ \Sigma_X^\tau/\tau= \Sigma_{V(\tau)}$, we have an isomorphism
\begin{equation}
    \Aff_{\overline\Sigma_X}\vert_{ \Sigma_X^\tau/\tau} \cong \Aff_{\Sigma_{V(\tau)}}
\end{equation}
\end{proposition}

\begin{proof}
By definition, the affine structure $\Aff_{\Sigma_X}$ is generated by strict piecewise linear functions. Such a function extends to a finite piecewise linear function in a neighborhood of a cell $\sigma/\tau\subseteq \overline\Sigma_X$ if and only if it is constant on $\tau$. The assertion now follows from Lemma \ref{lem: relation between normal bundle functions}.
\end{proof}

\begin{corollary}
\label{cor:functoriality on extended level}
    Let $f: X\to Y$ be a (not necessarily dominant) toroidal morphism of toroidal varieties. Then the natural piecewise linear map
    \begin{equation}
   \Trop(f) \colon \vert \overline \Sigma_X\vert\to \vert \overline \Sigma_Y\vert
    \end{equation}
    is a morphism of tropical spaces.
\end{corollary}

\begin{proof}
Combine Proposition \ref{prop:functoriality} with Proposition \ref{prop:affine functions at infinity}.
\end{proof}

\subsection{Affine structures on families of curves}

We  apply the constructions  in Section \ref{sec:affstrconcomplex} to families of curves and their tropicalizations. We first observe, in what seems like a \emph{groundhog day} moment but is a necessary verification, that in genus zero the affine structures we obtain from  tropicalization agree with those from \cite{psi-classes} and with those obtained from any prior incarnation of $\Mgn{0,n}^\trop$. Next we study the tropicalization of families of curves of arbitrary genus, and observe that while tropicalization gives a map of \TPL{}-spaces with affine structures, this is not always a family of tropical curves with affine structures in the sense of \cite{psi-classes}. However we identify combinatorial conditions describing a locus where the affine structures agree. We begin with some technical definitions that are useful in developing these ideas.

\begin{definition}
Let $\Sigma$ be a pure-dimensional cone complex, let $\sigma\in \Sigma$, and let $H^\sigma$ be a subgroup of $\sPL(\Sigma^\sigma)$ such that $[\Sigma^\sigma]$ (the weight on $\Sigma^\sigma$ which is $1$ on all top-dimensional cones) is balanced in the affine structure induced by $H^\sigma$. Let $\CPL_{H^\sigma}(\Sigma^\sigma)$ denote the subgroup of $\sPL(\Sigma^\sigma)$ consisting of those piecewise linear functions $\phi$ for which on every cone $\delta$ of $\Sigma^\sigma$ there exists a function $h\in H^\sigma$ with $\phi\vert_\delta=h\vert_\delta$.
We define the closure $\overline{H^\sigma}$ of $H^\sigma$ as  
\begin{equation}
    \overline{H^\sigma}=\{\phi\in \CPL_{H^\sigma}(\Sigma^\sigma) \mid \phi\cdot [\Sigma^\sigma]=0 \}
\end{equation}
and say that $H^\sigma$ is \textbf{normal} if $H^\sigma=\overline{H^\sigma}$. When no superscript is added, we assume that we are choosing $\sigma$ to be the cone point $0_\Sigma$.
\end{definition}

\begin{example}
Consider a cone complex $\Sigma$ with one vertex and four rays denoted $\rho_{i}$, for $i = 1,\ldots, 4$. Denote by $\varphi_{i}$ the piecewise linear function with slope one along $\rho_i$ and zero along all other rays.
For $\sigma$ equal to the vertex $0_\Sigma$, if we choose \begin{equation}H = \langle \varphi_1-\varphi_3, \varphi_2-\varphi_4 \rangle_\Z,\end{equation}
then $\Sigma^\sigma = \Sigma$ is balanced at $\sigma$ in the affine structure induced by $H$. In this case $\CPL_H(\Sigma) = \sPL(\Sigma)$, as one can find on each ray a function in $H$ with any given integral slope.

The piecewise linear function $\varphi_3-\varphi_4$  belongs to $\overline{H}\smallsetminus H$, thus showing that $H$ is not normal. One can verify that adding this function one  spans all of $\overline{H}$, thus 
\begin{equation}
    \overline{H} = \langle \varphi_1-\varphi_4,  \varphi_2-\varphi_4,  \varphi_3-\varphi_4\rangle_\Z.
\end{equation}
See Figure \ref{fig:normal} for a geometric illustration of this example.
\end{example}

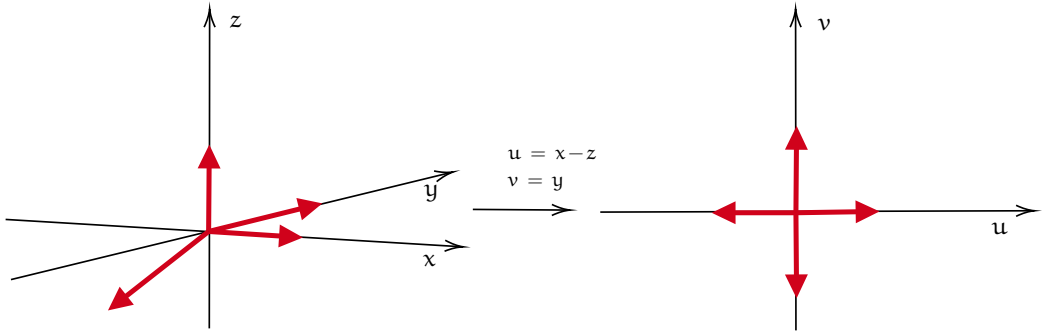
\begin{figure}[tb]
\centering
     \resizebox{.95 \textwidth}{!}{\input{normal.tikz}}

\caption{The primitive ray generators for the cone complex $\Sigma$, depicted in red, are embedded into $\R^2$ via the chosen generators of $H$, and into $\R^3$ via the generators of $\overline{H}$. The two balanced fans are related by a linear function. The failure of normality of $H$ detects the fact that the embedding of $\Sigma$ as a plane balanced fan is not as linearly independent as possible.}
    \label{fig:normal}
\end{figure}

In the next proposition, we show that $\overline{H^\sigma}$ is the largest subgroup of $\CPL_{H^\sigma}(\Sigma^\sigma)$ containing $H^\sigma$ in which $\Sigma^\sigma$ is balanced at $\sigma$.

\begin{proposition}
\label{prop:normalization}
Let $X$ be a toroidal variety with associated cone complex $\Sigma_X$, let $\sigma\in \Sigma_X$, and let $H^\sigma\subseteq \Aff_{\Sigma_X}(\Sigma_X^\sigma)$ be a subgroup such that $[\Sigma_X^\sigma]$ is balanced. Then we have
\begin{equation}
    \Aff_{\Sigma_X}(\Sigma_X^\sigma)\cap \CPL_{H^\sigma}(\Sigma_X^\sigma)\subseteq \overline{H^\sigma} \ .
\end{equation}
In particular, if $\CPL_{H^\sigma}(\Sigma_X^\sigma)=\sPL(\Sigma_X^ \sigma)$, we have
$\Aff_{\Sigma_X}(\Sigma_X^\sigma) \subseteq \overline{H^\sigma} $.
\end{proposition}

\begin{proof}
Let $\phi\in \Aff_{\Sigma_X}(\Sigma_X^\sigma)\cap \CPL_{H^\sigma}(\Sigma_X^\sigma)$. We will show in   Corollary \ref{cor:fundamental cycle exists on tropicalizations} that with the affine structure $\Aff_{\Sigma_X}$, the intersection product $\phi\cdot [\Sigma_X^\sigma]$ vanishes. 
Since $H^\sigma\subseteq\Aff_{\Sigma_X}(\Sigma_X^\sigma)$, we have $\phi\cdot [\Sigma_X^\sigma]=0$ even with respect to the affine structure defined by $H^\sigma$, that is $\phi\in \overline{H^\sigma}$.
\end{proof}

\begin{theorem}
\label{thm:cross-ratios on M0n}
Via the canonical identification of the underlying cone complexes, the affine structure on $\overline\Sigma_{\Mgnbar{0,n}}$ agrees with the affine structure on $\Mgnbar{0,n}^\trop$  (as defined in \cite[Section 4]{psi-classes}).
\end{theorem}

\begin{proof}
By Definition \ref{def:affine structure at infinity}, we have 
\begin{equation}
    \Aff_{\overline\Sigma_{\Mgnbar{0,n}}}=\iota_*\Aff_{\Sigma_{\Mgnbar{0,n}}} \cap \PLfin_{\overline\Sigma_{\Mgnbar{0,n}}} \ ,
\end{equation}
where $\iota\colon \Sigma_{\Mgnbar{0,n}}\to \overline\Sigma_{\Mgnbar{0,n}}$ is the inclusion. By \cite[Lemma 4.16]{psi-classes} we also have
\begin{equation}
    \Aff_{\Mgnbar{0,n}^\trop}= \kappa_*\Aff_{\Mgn{0,n}^\trop}\cap \PLfin_{\Mgnbar{0,n}^\trop} \ ,
\end{equation}
where $\kappa\colon \Mgn{0,n}^\trop\to \Mgnbar{0,n}^\trop$ denotes the inclusion. Since $ \PLfin_{\overline\Sigma_{\Mgnbar{0,n}}} \hspace{-.4cm}= \ \PLfin_{\Mgnbar{0,n}^\trop}$, it suffices to show that the affine structures on $\Sigma_{\Mgnbar{0,n}}$ and $\Mgn{0,n}^\trop$ coincide.

By \cite[Theorem 4.32]{psi-classes},  affine functions on $\Mgn{0,n}^\trop$ are given by cross ratios \cite[Definition 4.8]{psi-classes}. We first show that the tropicalization of a cross ratio on $\Mgn{0,n}$ is a tropical cross ratio. All cross ratios are pull-backs of cross ratios on $\Mgn{0,4}$ via the forgetful map, both algebraically and tropically, and tropicalization respects the $S_4$-action on $\Mgn{0,4}$.
Thus, by the functoriality of taking cone complexes of Proposition \ref{prop:functoriality},  it suffices to show that the cross ratio $(p_1,p_2;p_3,p_4)$ on $\Mgn{0,4}$ tropicalizes to a cross ratio. We have
\begin{equation}
    (\infty,0;1, z) = z \ ,
\end{equation}
whose associated divisor is $(0)-(\infty)$. The point $0$ corresponds to the nodal curve $(p_2p_4 \vert p_1p_3)$, whereas the point $\infty$ corresponds to the curve $(p_1p_4\vert p_2p_3)$. Therefore, the tropicalization of the cross ratio $(p_1,p_2;p_3,p_4)$ is the function with slope $1$ on the ray $\rho_{(13\vert 24)}$, slope $-1$ on the ray  $\rho_{(14\vert 23)}$, and slope $0$ on the ray  $\rho_{(12\vert 34)}$. This is precisely the tropical cross ratio function $\xi_{((1,2),(3,4))}$ corresponding to the primitive vertex cross ratio datum $c_{(((1,2),(3,4))}$ at the unique vertex of the  tropical curve in the minimal cone of $\Mgn{0,4}^\trop$. 

We have  shown that the natural identification 
\begin{equation}
    \Sigma_{\Mgnbar{0,n}}\to \Mgn{0,n}^\trop 
\end{equation}
is linear. It remains to show that there are not more functions on $\Sigma_{\Mgnbar{0,n}}$. The subgroup $H = \Aff_{\Mgn{0,n}^\trop}(\Sigma_{\Mgnbar{0,n}})$ is naturally identified with the dual to the integral lattice in $\R^{\binom {n-1} 2-1}$. 
Since $\Mgn{0,n}^\trop$ is a tropical linear space \cite[Theorem 6.4.12]{MacStu}, we have that $H = \overline{H}$ by \cite[Proposition 4.5]{FrancoisCocycle}.
In addition $\CPL_H(\Mgn{0,n}^\trop)=\sPL(\Mgn{0,n}^\trop)$, because $\Mgn{0,n}^\trop$ is embedded in $\R^{\binom{n-1} 2-1}$ by its global linear functions. We can thus apply Proposition \ref{prop:normalization} to conclude that $\Aff_{\Sigma_{\Mgnbar{0,n}}}(\Sigma_{\Mgnbar{0,n}})\subseteq \overline{H} = \Aff_{\Mgn{0,n}^\trop}(\Sigma_{\Mgnbar{0,n}})$.

Given any cone $\sigma\in \Mgn{0,n}^\trop$, the same argument applies to $\Sigma_{\Mgnbar{0,n}}^\sigma$, because $\Sigma_{\Mgnbar{0,n}}^\sigma-\sigma$  is again a tropical linear space \cite[Proposition 2.5]{FeichtnerSturmfels}. This shows that 
$\Aff_{\Sigma_{\Mgnbar{0,n}}}(\Sigma^\sigma_{\Mgnbar{0,n}}) = \Aff_{{\Mgn{0,n}^\trop}}(\Sigma^\sigma_{\Mgnbar{0,n}})$, thus concluding the proof.
\end{proof}

We turn our attention to more general families of curves and their tropicalization. 

\begin{definition}
\label{def:tropicalizable family}
A \textbf{tropicalizable} family of $n$-marked stable curves of genus $g$ is a toroidal morphism $\pi\colon\mc C\to \mc B$ of toroidal varieties, together with $n$ toroidal sections $(\sigma_i)_{1\leq i \leq n}$, satisfying the following three properties:
\begin{enumerate}
    \item The family $\pi$ together with the sections $(\sigma_i)_i$ is a family of $n$-marked genus-$g$ stable curves.
    \item The family $\Trop(\pi)\colon \overline {\TC}\to \overline{\TB}$, together with the sections $\Trop(\sigma_i)$ and the genus function that assigns on $\sigma^\diamond$ for a cone $\sigma\in {\TC}$ the arithmetic genus of the generic fiber of $V(\sigma)\to \pi(V(\sigma))$, is an $n$-marked family of \TPL-curves i.e.\ there exists a cartesian diagram
    \begin{equation}\label{diag:carttoV}
    \begin{tikzcd}
        \overline{\TC}  \arrow[r, "f_\star"]\arrow[d,"\Trop(\pi)"]  &
            \Atlass{g,n\sqcup \{\star\}} \arrow[d,"\mu_\star"]\\
        \overline{\TB} \arrow[r,"f"] &
            \Atlass{g,n}\ .
    \end{tikzcd}
\end{equation}
\item For every $\sigma\in {\TC}$ the generic fiber of the induced map $V(\sigma)\to \pi(V(\sigma))$ is irreducible.
\end{enumerate}
\end{definition}

\begin{remark}
To allow for greater flexibility, we allow the total space $\mc C$ (but not $\mc B$, for simplicity of the exposition) to have self-intersections, that is to be a toroidal embedding without self-intersections only étale locally. There still is an associated cone complex ${\TC}$, albeit one in which some cones have some of their faces glued together. One can remove the self-intersections by blowing up $\mc C$ according to a suitable subdivision of ${\TC}$ (essentially, the only thing that can go wrong is the existence of loops in the fibers; these need to be subdivided), but this comes at the cost of trading a family of stable curves for one of semistable curves.
\end{remark}

\begin{example}
\label{ex:monodromy}
While the strata $O(\sigma)$ and their closures $V(\sigma)$ are always irreducible, condition (3) in Definition \ref{def:tropicalizable family} is not superfluous. It can happen that the fibers of $V(\sigma)\to \pi(V(\sigma))$ are reducible, but because of the monodromy action on the irreducible components, the domain $V(\sigma)$ is still irreducible. As an example consider the family of plane curves over $\Spec \CC[t^{\pm 1}]$ given by a union of three lines as
\begin{multline}
    \{(t-1)x+(t-\zeta)y+ (t-\zeta^2)z=0\} \quad\cup \\
        \cup \quad\{(t-\zeta)x+(t-\zeta^2)y+ (t-1)z=0\} \quad \cup \\
            \cup \quad \{(t-\zeta^2)x+(t-1)y+ (t-\zeta)z=0\} \ ,
\end{multline}
where $\zeta\in \CC^*$ denotes a primitive third root of unity. Traveling from $t=1$ to $t=\zeta$ permutes the three lines cyclically, and in particular the family descends to an irreducible family over $\Spec \CC[t^{\pm 3}]$ whose fibers are unions of three lines in a plane that do not meet in a point. The construction of this example already shows how to fix families in which condition (3) is violated: one has to pass to an étale cover that trivializes the monodromy action on the components of the fibers.
\end{example}

\begin{proposition}
\label{prop:cross-ratios are linear}
Let $\pi:\calC \to \calB$ be a tropicalizable family of of stable, $n$-marked logarithmic curves, and let $\Tpi:=\Trop(\pi)\colon \overline {\TC}\to \overline{\TB}$ be  its tropicalization. Then 
the maps  $f$ and $f_\star$ from diagram \eqref{diag:carttoV} are linear. 
\end{proposition}

\begin{proof}
We first show the linearity of $f$ on the interior $\TB$ of $\overline\TB$. Let $b_0\in \TB$ and let $\sigma$ be the unique cone of $\TB$ with $b_0\in \sigma^\diamond$. Let $G$ be the cycle-rigidified, $n$-marked, genus-$g$ stable graph such that $f(b_0) \in  \sigma_G^\diamond$. As the combinatorial type of the fibers of $\Tpi$ remains constant on $\sigma^\diamond$, we have $f(\sigma^\diamond)\subseteq \sigma_G^\diamond$ and $f$ and $f_\star$ induce an identification of the combinatorial type of $\overline\TC_b$ with $G$ for every $b\in \sigma^\diamond$. 
Since linear functions on $\Atlass{g,n}$ are generated by vertex and edge primitive cross ratios (\cite[Proposition 4.13]{psi-classes}), we show linearity of $f$ by showing that the pullback of a primitive cross ratio is linear on $\TB$. 

Let $c=c_{((f_s,f_e),(f_{-1},f_1))}$ be a primitive cross ratio of vertex type at a rational vertex $v$ of $G$ with associated function $\xi_c$ on $(\Atlass{g,n})^{\sigma_G}$. We  show that $f^*\xi_c$ is linear on $\mathtt B^\sigma$. 

We may assume, possibly after subdividing some cones of $\TB$, that $\TB^\sigma$ is simply connected. We may also assume, possibly after subdividing some cones of $\overline\TC$, that the four flags $f_s, f_e, f_{-1}, f_1$ belong to distinct edges in each fiber. Then the vertex $v$ defines a component $V(v)$ of $\mc C_{V(\sigma)}$. The four tangent directions $f_s$, $f_e$, $f_{-1}$, and $f_1$ determine four sections of $V(v)$ and thus a morphism
\begin{equation}
    V(\sigma)\to \Mgnbar{0,4};
\end{equation}
precomposing its tropicalization with the map quotienting $\sigma$, we obtain $\TB^\sigma \to \Mgn{0,4}^\trop$. Since  $f^*\xi_c$ agrees with the pullback of the appropriate cross ratio from $\Mgn{0,4}^\trop$, linearity follows from Theorem~\ref{thm:cross-ratios on M0n}.

Now let $c=c_{(e, (f_s,f_e), (f_{-1},f_1))}$ be a primitive cross ratio of edge type on $G$ at an edge $e$. First assume that there exists a face $\tau$ of $\sigma$ such  that $f(\tau^\diamond)\subseteq \sigma_{\overline G}$ for some edge contraction $\overline G$ of $G$ in which $e$ got contracted to a genus-$0$ vertex $v$. Then $f_s$, $f_e$, $f_{-1}$, and $f_1$ induce flags at $v$ and in particular a primitive cross ratio $\overline c=c_{((f_s,f_e),(f_1,f_1))}$ of vertex type on $\overline G$. By the vertex type case above, the function $f^*\xi_{\overline c}$ is linear on $\mathtt B^\tau$, and by Lemma \ref{lem: compatibility of affine functions} the restriction $f^*\xi_{\overline c}\vert_{\mathtt{B}^\sigma}$ is linear on  $\mathtt B^\sigma$. But $f^*\xi_{\overline c}\vert_{\mathtt B^\sigma}$ coincides with $f^*\xi_c$ on $\mathtt B^\sigma$, concluding the proof in this case.

Next, assume that for any face of $\sigma$ where the edge $e$ gets contracted, a vertex of positive genus is formed. 
We can view $\mc C$ as the base of a family of $(n\sqcup \{\star\})$-marked stable curves 
\begin{equation}
    \pi_\star\colon \mc C_\star\to \mc C.
\end{equation}
The cycle rigidification of $\Tpi$ induced by the maps $f$ and $f_\star$ from diagram \eqref{diag:carttoV} produces a cycle rigidification of the  family
\begin{equation}
    \overline{\mathtt C}_\star\coloneqq \Sigma_{\mc C_\star} \xrightarrow{\Trop(\pi_\star)} \overline{\TC}\ ,
\end{equation}
that is a Cartesian \TPL-diagram
\begin{equation}
\label{eq:rigidification of family/family}
    \begin{tikzcd}
        \overline\TC_\star \arrow[r,"f_{\star,\bullet}"]\arrow[d]  &
            \Atlass{g,n\sqcup \{\star,\bullet\}} \arrow[d]\\
        \overline\TC \arrow[r,"f_\star"] &
            \Atlass{g,n\sqcup \{\star\}} \ .
    \end{tikzcd}
\end{equation}
Let $\sigma_e\in \Sigma$ be the cone over $\sigma$ whose fibers corresponds to $e$. The stratum $V(\sigma_e)$ is the node corresponding to $e$ in every fiber of $\mc C_{V(\sigma)}\to V(\sigma)$ and thus the restriction of $\pi$ defines an isomorphism $V(\sigma_e)\to V(\sigma)$. The pull-back  $\mu_\star^*\xi_c$  is a cross ratio on $\Atlass{g,n\sqcup\{\star\}}$, although not a primitive one. More precisely, if $G_\star$ denotes the rigidified $(n\sqcup\{\star\})$-marked graph, obtained by introducing a vertex $v$ on $e$ in $G$ and attaching a $\star-$marked leg $l_\star$ at $v$, then $f_\star(\sigma_e)\subseteq \sigma_{G_\star}$ and $c$ induces a cross ratio datum on $G_\star$ by ignoring the fact that $e$ has been subdivided. If $e_1$ and $e_2$ denote the two edges adjacent to $v$ in the order seen when traversing $e$ according to $c$, and $g_1$ and $g_2$ are the corresponding half-edges starting at $v$, as shown in Figure \ref{fig:subdivided edge}, then we have 
\begin{equation}\label{eq:sumoftwo}
    \mu_\star^*\xi_c = \xi_{(e_1,(f_s,l_\star),(f_{-1},g_2))}+ \xi_{(e_2,(l_\star,f_e),(g_2,f_1))}
\end{equation}
The two edges $e_1$ and $e_2$ are contracted on the two faces of $\sigma_e$ that map isomorphically onto $\sigma$ and correspond to the two vertices of $e$. Therefore, for each of the two primitive cross ratio summands in \eqref{eq:sumoftwo} we are in the special situation from before:  the edge of the edge-type primitive cross ratio is contracted to a genus zero vertex. Algebraically, this yields
\begin{multline}
\mc O_{V(\sigma_e)}\cong 
    \left(\mc O_{\mc C_{\sigma_e}}(f_\star^*\xi_{(e_1,(f_s,l_\star),(f_{-1},g_2))})\otimes \mc O_{\mc C_{\sigma_e}}(f_\star^*\xi_{(e_2,(l_\star,f_e),(g_2,f_1))})\right)\vert_{V(\sigma_e)} \cong \\
        \cong \mc O_{\mc C_{\sigma_e}}(f_\star^*\mu_\star^*\xi_c)\vert_{V(\sigma_e)}\cong
            \mc O_{\mc C_{\sigma_e}}(\Tpi^*f^*       \xi_c\vert_{\sigma_e})\vert_{V(\sigma_e)}\cong 
                \left(\pi^*\mc O_{\mc B_{\sigma}}(f^*\xi_c)\right)\Big\vert_{V(\sigma_e)} \ .        
\end{multline}
As we have noticed before, $\pi$ induces an isomorphism $V(\sigma_e)\to V(\sigma)$, so we conclude that
\begin{equation}
    \mc O_{\mc B_{\sigma}}(f^*\xi_c)\vert_{V(\sigma)}\cong \mc O_{V(\sigma)} \ ,
\end{equation}
which is equivalent to $f^\ast\xi_c$ being linear on $\mathtt B^\sigma$. 

To show that $f$ is also linear on the boundary of $\mathtt B$ it suffices to note that on both source (by Definition \ref{def:affine structure at infinity}) and target (by \cite[Lemma 4.16]{psi-classes}), the affine structure near the boundary is induced by the affine structure on the interior in the sense that the affine functions at the boundary are precisely the continuous extensions of affine functions on the interior.

To show that $f_\star$ is linear one applies the same arguments to the map $\pi_\star$ and the rigidification displayed in \eqref{eq:rigidification of family/family}.
\end{proof}

\begin{proposition}\label{prop:goodlocus}
With the set-up and notation of Proposition \ref{prop:cross-ratios are linear}, let $x\in {\overline\TC}$. If $x$ has genus $0$, then all functions in $\Aff_{{\TC},x}$ are linear on the fibers of $\Tpi$ close to $x$. Moreover, we have an exact sequence
\begin{equation}
    0 \to \Aff_{\overline\TB,\Tpi(x)}\to \Aff_{{\overline\TC},x} \to \Omega^1_{{{\overline\TC}}_{\Tpi(x)},x}\to 0 \ .
\end{equation}
If $x$ is a vertex (of genus $0$) in its fiber, or is on an edge adjacent to a vertex of genus $0$, then $\Omega^1_{{{\overline\TC}}_{\Tpi(x)},x}$ consists of the differentials of all harmonic functions.
\end{proposition}

\begin{figure}
    \centering
    \begin{tikzpicture}[thick]
        \coordinate (v) at (0,0);
        \coordinate (top) at (2,0);
        \coordinate (bot) at (-2,0);
        \node [fill=black, circle, inner sep= .05cm, label=45:{$v$}] at (v) {};
        \draw (v) -- +(0,2) node[at end,left]{$l_*$};
        \draw[->, very thick] (v) -- +(.7,0) node[pos=.6, below,yshift=-.2ex]{$g_2$};
        \draw (v) -- (top) node[pos=.6,above]{$e_2$}; 
        \draw[->, very thick] (v) -- +(-.7,0) node[pos=.6, below,yshift=-.2ex]{$g_1$};
        \draw (v) --(bot) node[pos=.6,above]{$e_1$} ; 
        
        \draw[->,very thick] (top) --  +(.7,.7) node[at end, above left]{$f_{1}$};
        \draw (top) -- +(1,1);
        \draw[->,very thick] (top) -- node[at end, below left]{$f_e$} +(.7,-.7);
        \draw (top) -- +(1,-1);

        \draw[->,very thick] (bot) -- node[at end, above right]{$f_{-1}$} +(-.7,.7);
        \draw (bot) -- +(-1,1);
        \draw[->,very thick] (bot) -- +(-.7,-.7)  node[at end, below right]{$f_s$};
        \draw (bot) -- +(-1,-1);
        \draw [decorate,decoration={brace,amplitude=15pt,mirror,raise=2ex}]
  (bot)--(top) node[midway,yshift=-3em]{$e$};
    \end{tikzpicture}
    \caption{the subdivision of the edge $e$ in $G_\star$}
    \label{fig:subdivided edge}
\end{figure}
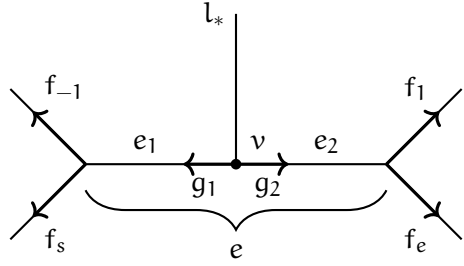

\begin{proof}
The injectivity of the pull-back is obvious, so is the surjectivity of restriction to the fiber (we consider the fiber with the induced affine structure). We show the harmonicity of affine functions on fibers and the exactness in the middle simultaneously.

We first prove the proposition in the case when  $x\in {{\TC}}$.  Suppose $x$ is on an edge in its fiber. In that case a function $\phi\in \Aff_{{{\TC}},x}$ is linear, and hence harmonic, on all fibers in a neighborhood of $x$. If $\phi\vert _{\Tpi^{-1}\{\Tpi(x)\}}$ is constant near $x$, then $\phi$ is constant on fibers on a small neighborhood of $x$. Let $\sigma\in{{\TC}}$ be the unique cone with $x\in \sigma^\diamond$ and let $\tau=\Tpi({{\sigma}})$. By definition of the affine structure on ${{\TC}}$, $\phi$ defines a piecewise linear function on ${{\TC}}^\sigma$, and $\phi$ is constant on all the fibers of the induced map ${{\TC}}^\sigma\to \TB^\tau$. Taking any section $\TB^\tau\to {{\TC}}^\sigma$ of this map (for example a section corresponding to one of the two vertices of the edge on which $x$ sits), we obtain a piecewise linear function $\chi$ on $\TB^\tau$ with $\Tpi^*\chi=\phi$. The stratum $V(\sigma)$ is a node in each fiber, and the projection $\pi\vert_{V(\sigma)} \colon V(\sigma)\to V(\tau)$ is an isomorphism. We have
\begin{equation}\label{eq:ses}
\pi\vert_{V(\sigma)}^* \mc O_{\mc B_\tau}(\chi)\vert_{V(\tau)}\cong 
    \mc O_{\mc C_\sigma}(\Tpi^*\chi)\vert_{V(\sigma)} \cong
        \mc O_{\mc C_\sigma}(\phi)\vert_{V(\sigma)} \cong
        \mc O_{V(\sigma)} \ ,
\end{equation}
where the last isomorphism follows because $\phi$ is affine; we conclude that $\chi$ is affine on $\TB$.

Now suppose $x$ is a vertex of genus $0$, let $\sigma$ and $\tau$ be as above, and let $\phi\in \Aff_{\TC,x}$. Let $e$ be an edge in the fiber of $x$ that is adjacent to $x$. After subtracting a suitable cross ratio, which is harmonic on fibers and affine by Proposition \ref{prop:cross-ratios are linear}, we may assume that $\phi$ is constant on $e$. So by the case where $x$ was on an edge, there exists an affine function $\chi\in \Aff_{\TB,\Tpi(x)}$ such that $\phi-\Tpi^*\chi$ vanishes on $e$. In particular, by replacing $\phi$ with this difference of affine functions,  we may assume that $\phi\vert_{\sigma}=0$. By Proposition \ref{prop:affine functions at infinity}, $\phi$ is induced by an affine function $\overline\phi$ on ${{\TC}}^\sigma/\sigma={{\TC}}_{V(\sigma)}$. Let $T_x({{\TC}}_{\Tpi(x)})$ denote the set of half-edges starting at $x$ in its fiber. For each $h\in T_x({{\TC}}_{\Tpi(x)})$ let $m_h$ be the slope of $\phi$ in the direction of $h$, and let $\sigma_h$ be the cone of ${{\TC}}$ spanned by $\sigma$ and $h$. Then the normalization $\widetilde V(\sigma)$ of the stratum $V(\sigma)$ is a family of $T_x({{\TC}}_{\Tpi(x)})$-marked rational stable curves over $V(\tau)$, and $V(\sigma_h)$ is the section of this family corresponding to $h\in T_x({{\TC}}_{\Tpi(x)})$. The fact that $\phi$ is affine means that the divisor 
\begin{equation}
    \sum_{h\in T_x({{\TC}}_{\Tpi(x)})} m_h V(\sigma_h)
\end{equation}
is rationally equivalent to $0$ on $\widetilde V(\sigma)$. Since the fiber of $\widetilde V(\sigma)\to V(\tau)$ over any $y\in O(\tau)$ is isomorphic to $\PP^1$, we conclude that $\sum_h m_h=0$, because every divisor on $\PP^1$ rationally equivalent to $0$ has degree $0$. This shows that $\phi$ is harmonic at $x$. Onto exactness in the middle of \eqref{eq:ses}, if $x$ is a vertex and $\phi\vert_{\Tpi^{-1}\{\Tpi(x)\}}$ is constant near $x$, then it is constant on an edge adjacent to $x$; by the case where $x$ was on an edge we obtain an affine function $\chi\in \Aff_{\TB,\Tpi(x)}$ with $\phi=\Tpi^*\chi$.

Let $x$ be either a rational vertex or adjacent to a genus $0$ vertex in its fiber. By the linearity of $f_\star$ (Proposition \ref{prop:cross-ratios are linear}), all (pull-backs of) cross ratios are linear on ${{\TC}}$. By the condition on $x$, by \cite[Proposition 4.24]{psi-classes} there are enough cross ratio functions defined at $x$ to realize all possible differentials of harmonic functions on the fiber.

It remains to analyze the case where $x\in \overline\TC\setminus \TC$. Let $\delta$ be the unique cone of $\TC$ with $x\in \TC^\delta/\delta$ and let $\theta=\Tpi(\delta)$, which is a cone in $\TB$. We can either have $\dim(\delta)=\dim(\theta)+1$ or $\dim(\delta)=\dim(\theta)$. In the first case,  the point $x$ is either the endpoint of a leg or a node at infinity in its fiber (see \cite[p.\ 14]{psi-classes}). It follows that $\Omega^1_{\Tpi(x),x}=0$ and we need to show that the pull-back $\Aff_{\overline\TB,\Pi(X)}\to \Aff_{\overline\TC,x}$ is an isomorphism. The morphism $V(\delta)\to V(\theta)$ induced by $\pi$ is an isomorphism because $V(\delta)$ either consists of a single marked point or a single node in each fiber over $V(\theta)$. So $\pi$ also induces an isomorphism 
$\Sigma_{V(\delta)}\to \Sigma_{V(\theta)}$ of tropical spaces and we are done by Proposition \ref{prop:affine functions at infinity}.
In the second case, i.e.\ when $\delta$ and $\theta$ have the same dimension, a neighborhood of $x$ in $\overline \TC_{\Pi(x)}$ is contained in $\TC^\delta/\delta$. It therefore suffices to prove the assertion for the family 
\[
\Sigma_{V(\delta)}\cong \TC^\delta/\delta\to \TB^\theta/\theta\cong \Sigma_{V(\theta)} \ ,
\]
where the canonical isomorphisms come from Proposition \ref{prop:affine functions at infinity}. This family is the tropicalization of the induced family of curves $V(\delta)\to V(\theta)$ and $x$ is in the interior of $\Sigma_{V(\delta)}$, so we have reduced to the case that we already proved.
\end{proof}

\begin{remark}
Consider the set $\TU\subseteq {{\TC}}$ of all genus-$0$ points that are either vertices or on an edge adjacent to a genus-$0$ vertex. The locus $\TU$ is closely related to the locus  $\GoodAtlas{g,n}$ introduced in \cite[Definition 4.22]{psi-classes}. More precisely, we have
\begin{equation}
    f^{-1}\GoodAtlas{g,n}= \{y \mid {{\TC}}_y\subseteq \TU\} \,
\end{equation}
where ${{\TC}}_y$ denotes the fiber over $y$.
One can view $\TU$ as the locus where ${{\TC}}\to \TB$ is a family of tropical curves in the sense of \cite{psi-classes}. The affine structure on $\TU$ is completely determined by that on $\TB$. Equivalently, the inclusion
\begin{equation}
    \TU\to \TB\times_{\Atlass{g,n}}\Atlass{g,n\sqcup \{\star\}} \ ,
\end{equation}
where the fiber product is taken in the category of tropical spaces, is an open immersion.
\end{remark}

\section{Tropicalizations of line bundles}

Let $X$ be a toroidal variety and let $\pi \colon L \to X$ be a  line bundle on $X$. We equip   the total space of $L$ with the toroidal structure whose boundary is $Z\cup \pi^{-1}\partial X$, where $Z$ denotes the zero-section of the bundle. The cone complex $\Sigma_L$ is naturally identified with $\Sigma_X\times\rho_Z$, where $\rho_Z$ is the ray of $\Sigma_L$ corresponding to the boundary divisor $Z$. Via the construction from Section \ref{sec:affstrtrop}, we obtain an affine structure on $\overline\Sigma_X\times\overline\rho_Z$.

\begin{definition}  
Extend the affine structure on $\overline\Sigma_X\times (\overline \rho_Z\setminus\{0\})$ to $\overline\Sigma_X\times \T$ by linearity, where we identify $\T_{>0}$ with $\overline{\rho}_Z\setminus\{0\}$.
We  call the morphism
\begin{equation}
    \overline\Sigma_X\times \T\to \overline\Sigma_X
\end{equation}
of spaces with affine structures the \textbf{tropicalization} of $L$ and denote it by $\Trop(L)$.
\end{definition} 

\begin{remark}\label{rem:tlbviapb}
We describe an equivalent way of defining $\Trop(L)$: first compactify the total space $L$ by adding a section at infinity, that is consider the projective bundle $P_L=\PP(L\oplus \mc O_X)$; the boundary consists of the union  of the  pull-back of $\partial X$ with the two sections. If $\rho_\infty$ is the ray of $\Sigma_{P_L}$ corresponding to the section at infinity, then  $\Trop(L)$ is the complement of $\Sigma_{P_L}^{\rho_\infty}/\rho_\infty$ in $\overline \Sigma_{P_L}$. More concisely, we add a section at infinity, tropicalize, and then remove the tropical section at infinity.
\end{remark}

The tropicalization of $L$ is not necessarily a tropical line bundle (see \S\ref{sec:tlb}), as we illustrate in the following example. 

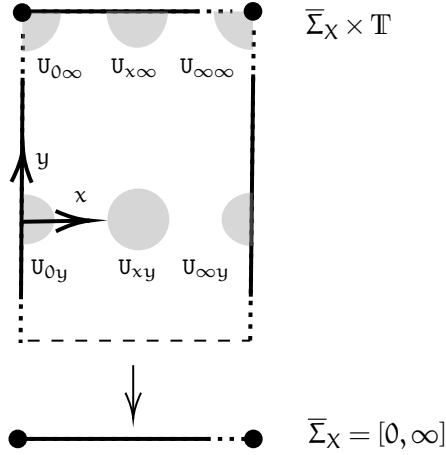
\begin{figure}
    \centering

\tikzset{every picture/.style={line width=0.75pt}} 

\begin{tikzpicture}[x=0.75pt,y=0.75pt,yscale=-1.2,xscale=1.2]

\draw [line width=1.5]    (31.83,281.83) -- (112.83,281.83) ;
\draw  [dash pattern={on 4.5pt off 4.5pt}]  (36,240.5) -- (129.5,240.5) ;
\draw [line width=1.5]  [dash pattern={on 1.69pt off 2.76pt}]  (42,102) -- (122,102) ;
\draw [line width=1.5]    (31.83,102) -- (108.25,102) ;
\draw [line width=1.5]  [dash pattern={on 1.69pt off 2.76pt}]  (131,110) -- (129.5,240.5) ;
\draw [line width=1.5]    (130.75,128.5) -- (129.75,222) ;
\draw [line width=1.5]  [dash pattern={on 1.69pt off 2.76pt}]  (34,111.5) -- (33,241.5) ;
\draw [line width=1.5]    (33.83,129.93) -- (33.17,223.07) ;
\draw  [fill={rgb, 255:red, 0; green, 0; blue, 0 }  ,fill opacity=1 ] (127.33,281.83) .. controls (127.33,279.9) and (128.9,278.33) .. (130.83,278.33) .. controls (132.77,278.33) and (134.33,279.9) .. (134.33,281.83) .. controls (134.33,283.77) and (132.77,285.33) .. (130.83,285.33) .. controls (128.9,285.33) and (127.33,283.77) .. (127.33,281.83) -- cycle ;
\draw [line width=1.5]  [dash pattern={on 1.69pt off 2.76pt}]  (31.83,281.83) -- (131.83,281.83) ;
\draw  [fill={rgb, 255:red, 0; green, 0; blue, 0 }  ,fill opacity=1 ] (28.33,281.83) .. controls (28.33,279.9) and (29.9,278.33) .. (31.83,278.33) .. controls (33.77,278.33) and (35.33,279.9) .. (35.33,281.83) .. controls (35.33,283.77) and (33.77,285.33) .. (31.83,285.33) .. controls (29.9,285.33) and (28.33,283.77) .. (28.33,281.83) -- cycle ;
\draw    (80.33,250.83) -- (80.33,271.33) ;
\draw [shift={(80.33,273.33)}, rotate = 270] [color={rgb, 255:red, 0; green, 0; blue, 0 }  ][line width=0.75]    (10.93,-3.29) .. controls (6.95,-1.4) and (3.31,-0.3) .. (0,0) .. controls (3.31,0.3) and (6.95,1.4) .. (10.93,3.29)   ;
\draw  [color={rgb, 255:red, 255; green, 255; blue, 255 }  ,draw opacity=0 ][fill={rgb, 255:red, 155; green, 155; blue, 155 }  ,fill opacity=0.4 ] (130.61,200.75) .. controls (130.57,200.75) and (130.54,200.75) .. (130.5,200.75) .. controls (123.32,200.75) and (117.5,195.71) .. (117.5,189.5) .. controls (117.5,183.29) and (123.32,178.25) .. (130.5,178.25) .. controls (130.54,178.25) and (130.57,178.25) .. (130.61,178.25) -- cycle ;
\draw  [color={rgb, 255:red, 255; green, 255; blue, 255 }  ,draw opacity=0 ][fill={rgb, 255:red, 155; green, 155; blue, 155 }  ,fill opacity=0.4 ] (34.46,179) .. controls (34.49,179) and (34.53,179) .. (34.56,179) .. controls (41.61,179.04) and (47.29,183.77) .. (47.26,189.57) .. controls (47.23,195.37) and (41.49,200.04) .. (34.44,200) .. controls (34.4,200) and (34.37,200) .. (34.34,200) -- cycle ;
\draw  [color={rgb, 255:red, 0; green, 0; blue, 0 }  ,draw opacity=0 ][fill={rgb, 255:red, 155; green, 155; blue, 155 }  ,fill opacity=0.41 ] (69.5,189.5) .. controls (69.5,182.32) and (75.32,176.5) .. (82.5,176.5) .. controls (89.68,176.5) and (95.5,182.32) .. (95.5,189.5) .. controls (95.5,196.68) and (89.68,202.5) .. (82.5,202.5) .. controls (75.32,202.5) and (69.5,196.68) .. (69.5,189.5) -- cycle ;
\draw  [color={rgb, 255:red, 0; green, 0; blue, 0 }  ,draw opacity=0 ][fill={rgb, 255:red, 155; green, 155; blue, 155 }  ,fill opacity=0.4 ] (130.67,118.25) .. controls (130.67,118.25) and (130.67,118.25) .. (130.67,118.25) .. controls (130.67,118.25) and (130.67,118.25) .. (130.67,118.25) .. controls (121.55,118.25) and (114.17,110.97) .. (114.17,102) -- (130.67,102) -- cycle ;
\draw  [fill={rgb, 255:red, 0; green, 0; blue, 0 }  ,fill opacity=1 ] (127.17,102) .. controls (127.17,100.07) and (128.73,98.5) .. (130.67,98.5) .. controls (132.6,98.5) and (134.17,100.07) .. (134.17,102) .. controls (134.17,103.93) and (132.6,105.5) .. (130.67,105.5) .. controls (128.73,105.5) and (127.17,103.93) .. (127.17,102) -- cycle ;
\draw  [color={rgb, 255:red, 0; green, 0; blue, 0 }  ,draw opacity=0 ][fill={rgb, 255:red, 155; green, 155; blue, 155 }  ,fill opacity=0.4 ] (49.92,102.45) .. controls (49.94,111.57) and (42.69,118.97) .. (33.71,119) -- (33.67,102.5) -- cycle ;
\draw  [fill={rgb, 255:red, 0; green, 0; blue, 0 }  ,fill opacity=1 ] (30.17,102.5) .. controls (30.17,100.57) and (31.73,99) .. (33.67,99) .. controls (35.6,99) and (37.17,100.57) .. (37.17,102.5) .. controls (37.17,104.43) and (35.6,106) .. (33.67,106) .. controls (31.73,106) and (30.17,104.43) .. (30.17,102.5) -- cycle ;
\draw  [color={rgb, 255:red, 0; green, 0; blue, 0 }  ,draw opacity=0 ][fill={rgb, 255:red, 155; green, 155; blue, 155 }  ,fill opacity=0.4 ] (94.91,101.99) .. controls (94.92,110.44) and (89.14,117.29) .. (82.01,117.3) .. controls (74.88,117.3) and (69.09,110.46) .. (69.09,102.01) .. controls (69.09,101.95) and (69.09,101.89) .. (69.09,101.83) -- cycle ;
\draw [line width=1.5]    (33.67,190.5) -- (59,190.05) ;
\draw [shift={(62,190)}, rotate = 178.99] [color={rgb, 255:red, 0; green, 0; blue, 0 }  ][line width=1.5]    (14.21,-4.28) .. controls (9.04,-1.82) and (4.3,-0.39) .. (0,0) .. controls (4.3,0.39) and (9.04,1.82) .. (14.21,4.28)   ;
\draw [shift={(34,159.67)}, rotate = 92.48] [color={rgb, 255:red, 0; green, 0; blue, 0 }  ][line width=1.5]    (14.21,-4.28) .. controls (9.04,-1.82) and (4.3,-0.39) .. (0,0) .. controls (4.3,0.39) and (9.04,1.82) .. (14.21,4.28)   ;

\draw (36.34,205) node [anchor=north west][inner sep=0.75pt]  [font=\footnotesize]  {$\TU_{0y}$};
\draw (73,205) node [anchor=north west][inner sep=0.75pt]  [font=\footnotesize]  {$\TU_{xy}$};
\draw (100,205) node [anchor=north west][inner sep=0.75pt]  [font=\footnotesize]  {$\TU_{\infty y}$};
\draw (39,120.73) node [anchor=north west][inner sep=0.75pt]  [font=\footnotesize]  {$\TU_{0 \infty}$};
\draw (70,120.73) node [anchor=north west][inner sep=0.75pt]  [font=\footnotesize]  {$\TU_{x \infty}$};
\draw (99,120.73) node [anchor=north west][inner sep=0.75pt]  [font=\footnotesize]  {$\TU_{\infty \infty}$};
\draw (152.5,271.57) node [anchor=north west][inner sep=0.75pt]    {$\overline\Sigma_X = [0,\infty]$};
\draw (151.5,99.4) node [anchor=north west][inner sep=0.75pt]    {$\overline\Sigma_X\times \T$};
\draw (54.34,176.4) node [anchor=north west][inner sep=0.75pt]  [font=\footnotesize]  {$x$};
\draw (38.34,160.4) node [anchor=north west][inner sep=0.75pt]  [font=\footnotesize]  {$y$};

\end{tikzpicture}
\caption{The tropicalization of a line bundle on a toroidal space with a unique boundary divisor $H$. As suggested by the orientation of the picture, we call $x$ a linear coordinate for the ray $\rho_H$ and $y$ the coordinate for the ray $\rho_Z$ dual to the zero section $Z\cong X$.}
    \label{fig:troplbex}
\end{figure}

\begin{example}
We present here three examples of tropicalization of a line bundle $\pi\colon L\to X$ on a toroidal variety $X$ with exactly one boundary divisor $H$. The extended tropicalization $\overline{\Sigma}_X$ of $X$ is the closed segment $[0, \infty]$ and we denote by $\{V_0, V_\infty\}$ the open cover $\{[0, \infty), (0, \infty]\}$. The map of extended cone complexes giving the tropicalization of $L$ is depicted in Figure \ref{fig:troplbex}; in all three cases the combinatorics of the tropicalization is the same, what differs is the affine structure on the total space $\overline\Sigma_X\times \T$. We describe the affine structure by giving the affine functions on the six types of open sets $\mtt U_{ij}$ shaded in Figure \ref{fig:troplbex}. In each case we will denote $\widetilde H\coloneqq \pi^*H$.

\noindent\textsc{Case 1:} $(X,H) = (\PP^1, pt. = c_1(\calO_{\PP^1}(1)))$, $L = \calO_{\PP^1}(m)$.
 
\noindent The local affine  functions  are the following:
 \begin{align*}
&(a)\quad    \Aff_{\Trop (L)}(\TU_{0 \infty}) = \R &
    &(d)\quad  \Aff_{\Trop (L)}(\TU_{0 y}) =  \langle y-mx\rangle_\ZZ\oplus \R \\
&(b)\quad     \Aff_{\Trop (L)}(\TU_{x \infty}) = \langle x\rangle_\ZZ\oplus \R  &
    &(e)\quad     \Aff_{\Trop (L)}(\TU_{x y}) = \langle x,y\rangle_\ZZ\oplus \R \\
&(c)\quad     \Aff_{\Trop (L)}(\TU_{\infty \infty}) = \R &
    &(f)\quad      \Aff_{\Trop (L)}(\TU_{ \infty y}) = \langle y\rangle_\ZZ\oplus \R
 \end{align*}

\begin{description}
\item[$(a)$] since the open set $\TU_{0 \infty}$ contains the cell at infinity $\rho_Z/\rho_Z$, any affine function must be constant in $y$. The local function $\alpha x+r$ is affine if and only if $ \calO_{L}(\alpha \widetilde H)\vert_{Z}\cong \calO_Z = \calO_{\PP^1}$. We have
\begin{equation}
    \calO_{L}(\alpha \widetilde H)\vert_{Z}\cong \calO_Z(\widetilde H\cdot Z\alpha) =  \calO_{\PP^1}(\alpha);
\end{equation}
it follows that $\alpha=0$;

\item[$(b)$]  affine functions need to be constant in $y$. We  want to see when $\calO_{L}(\alpha \widetilde H)\vert_{Z\cap \widetilde H}\cong \calO_{Z\cap \widetilde H}$, but $Z\cap \widetilde H$ consists of a single point, and therefore any line bundle restricted to it must be trivial. There is no condition on $\alpha$;

\item[$(c)$] the set $\TU_{\infty \infty}$ contains the point $(\infty, \infty)$ and therefore affine functions must be constant in $x$ and $y$ and therefore constant.

\item[$(d)$] we are looking for functions of the form $\alpha x+\beta y+r$ such that $ \calO_{L}(\alpha \widetilde H+\beta Z)\vert_{Z}\cong  \calO_{\PP^1}$
\begin{equation}
    \calO_{L}(\alpha \widetilde H+\beta Z)\vert_{Z}\cong \calO_Z( (Z\cdot \widetilde H)\alpha +c_1(L)\beta ) =  \calO_{\PP^1}(\alpha+m\beta); 
\end{equation}
setting  $\alpha +m\beta= 0$ shows that affine functions have the form $\beta(y-mx)+r$;

\item[$(e)$] in this case we are again restricting a line bundle on the total space of $L$ to the point $Z\cap \widetilde H$, therefore always obtaining a trivial bundle. There is no restriction on the integral slopes of affine functions.

\item[$(f)$] affine functions must be constant in $x$. The function $\beta y+r$ is affine if 
$\calO_{L}(\beta \widetilde H)\vert_{\widetilde H}$ is trivial. But since $\widetilde H\cong \CC$, any line bundle may be trivialized and the result follows.
\end{description}
We see that $\Trop(L)$ is a tropical line bundle on $\overline{\Sigma}_X$ with Cartier data $\{(V_0, -mx),(V_\infty,0)\}.$

\noindent\textsc{Case 2:} $(X,H) = (\PP^2, H = c_1(\calO_{\PP^2}(1)))$, $L = \calO_{\PP^2}(m)$.
 
\noindent The local affine  functions  are the following:
\begin{align*}
    &(a)\quad     \Aff_{\Trop (L)}(\TU_{0 \infty}) = \R &
        &(d)\quad    \Aff_{\Trop (L)}(\TU_{0 y}) =  \langle y-mx\rangle_\ZZ\oplus \R \\
    &(b)\quad    \Aff_{\Trop (L)}(\TU_{x \infty}) = \R &
        &(e)\quad    \Aff_{\Trop (L)}(\TU_{x y}) = \langle y-mx\rangle_\ZZ \oplus \R \\
    &(c)\quad     \Aff_{\Trop (L)}(\TU_{\infty \infty}) = \R &
        &(f)\quad    \Aff_{\Trop (L)}(\TU_{ \infty y}) = \R
\end{align*}
We don't go through the computations in detail, since they are very similar to those in the first case. To highlight what the difference is, we just look at what is different in computing the affine sections for the open set $\TU_{x y}$:  we are now looking for functions of the form $\alpha x+\beta y+r$ such that $\calO_{L}(\alpha \widetilde H+\beta Z)\vert_{Z\cap \widetilde H}\cong \calO_{Z\cap \widetilde H}$. Notice that $Z\cap \widetilde H$ is now a hyperplane in $\PP^2$, and it is therefore isomorphic to $\PP^1$.
We have
\begin{equation}
    \calO_{L}(\alpha \widetilde H+\beta Z)\vert_{Z\cap \widetilde H}\cong \calO_{Z}( (Z\cdot \widetilde H)\alpha +c_1(L)\beta )\vert_{\widetilde H} =  \calO_{\PP^1}(\alpha+m\beta);    
\end{equation}
 there must be a non-trivial relation between $\alpha$ and $\beta$ in order for the line bundle to trivialize, yielding the result from $(e)$.
 In this case $\Trop(L)$ is a (trivial) tropical line bundle on the cone complex $\Sigma_X$, but not on the extended space $\overline{\Sigma}_X$.

 \noindent\textsc{Case 3:} $(X,H) = (\PP^1\times \PP^1, H = c_1(\calO_{\PP^1}(1,0)))$, $L = \calO_{\PP^1\times \PP^1}(0,1)$.
 In this case the local affine functions are:
\begin{align*}
    &(a)\quad   \Aff_{\Trop (L)}(\TU_{0 \infty}) = \R &
        &(d)\quad     \Aff_{\Trop (L)}(\TU_{0 y}) =  \R  \\
    &(b)\quad     \Aff_{\Trop (L)}(\TU_{x \infty}) = \langle x\rangle_\ZZ\oplus \R &
        &(e)\quad     \Aff_{\Trop (L)}(\TU_{x y}) = \langle x\rangle_\ZZ \oplus \R \\
    &(c)\quad     \Aff_{\Trop (L)}(\TU_{\infty \infty}) = \R &
        &(f)\quad     \Aff_{\Trop (L)}(\TU_{ \infty y}) = \R 
\end{align*}
Again, the computations are similar to the first case, so we just show the most interesting  example. We compute the local sections $ \Aff_{\Trop (L)}(\TU_{0 y})$. These are functions of the form  $\alpha x+\beta y+r$ such that $\calO_{L}(\alpha \widetilde H+\beta Z)\vert_{Z}\cong \calO_{Z} = \calO_{\PP^1\times\PP^1}$.
In this case:
\begin{equation}
    \calO_{L}(\alpha \widetilde H+\beta Z)\vert_{Z} = \calO_{Z}( (Z\cdot \widetilde H)\alpha+c_1(L)\beta)= \calO_{\PP^1\times \PP^1}(\alpha \widetilde H+\beta V),
\end{equation}
where $V= c_1(L)$ denotes the class of the vertical fiber in $\PP^1\times \PP^1$. For the bundle to be trivial we must have $\alpha=\beta=0$. We thus see that $\Trop(L)$ is not a tropical line bundle on $\Sigma_X$.
\end{example}

We investigate when the tropicalization of a line bundle is indeed a tropical line bundle. We will say that a line bundle is \emph{tropicalizable} when this happens. 

\begin{proposition}
\label{prop:tropicalization of line bundle is tropical line bundle}
Let $L$ be a line bundle on $X$ with associated invertible sheaf $\mc L$ and let $\TU\subseteq~\overline{\Sigma}_X$ be an open subcomplex. Then $\Trop(L)$ is a tropical line bundle on $\TU$ if and only if for every cell $\sigma/\tau$ of $\TU$ there exists a function $\phi_{\sigma/\tau}\in \sPL(\overline\Sigma_X^{\sigma/\tau})$ (see page \pageref{page:sPL in extended cone complex}) with 
\begin{equation}
\label{eq:conditionfortropicalizationbeingatropicallinebundle}
    \left(\mc O_{X_{\sigma}}(\phi_{\sigma/\tau}) \otimes \mc L\right)\vert_{V(\sigma)} \cong \mc O_{V(\sigma)} \ .
\end{equation}
In that case, $\Trop(L)$ is the tropical line bundle associated to the tropical Cartier divisor $(\overline{\Sigma}_X^{\sigma/\tau},\phi_{\sigma/\tau})_{\sigma/\tau\subseteq \TU}$.
\end{proposition}

\begin{proof}
The total space  $\Trop(L)$ is a tropical line bundle on $\TU$ if and only if for every cell $\sigma/\tau$ of $\TU$ there exists exactly one element in  $\Aff_{\Sigma_L}({\overline\Sigma_L}^{(\sigma/\tau)\times\rho_Z})/\Aff_{ \overline\Sigma_X}( \overline\Sigma_X^{\sigma/\tau})$ with slope $1$ along the ray $\rho_Z$ dual to the zero section of $L$. Let $I$ denote the set of functions in $\sPL({\overline\Sigma_L}^{(\sigma/\tau)\times\rho_Z})$ with slope $1$ along $\rho_Z$. There is a bijection 
\begin{equation}
    \alpha\colon \sPL(\overline\Sigma_X^{\sigma/\tau})\to I, \;\; \phi\mapsto  \pr_{\overline\Sigma_X}^*\phi + \pr_{\rho_Z}^*\id \ ,
\end{equation}
where we identify $\rho_Z$ with $\R_{\geq 0}$.
The structure map $L\to X$ restricts to an isomorphism $V(\sigma\times\rho_Z)\to V(\sigma)$. Since $\mc O_L(\alpha(0))$ is the line bundle associated to the zero-section, and the normal bundle to the zero-section of $L$ is  $ L$ itself, we obtain an identification of
\begin{equation}
    \mc O_L(\alpha(\phi))\vert_{V((\sigma/\tau)\times \rho_Z)} = (\mc O_{X_\sigma}(\phi) \otimes \mc L)\vert_{V(\sigma)}
\end{equation}
for every $\phi\in \sPL(\overline\Sigma_X^{\sigma/\tau})$. It follows  that $I\cap \Aff_{\overline\Sigma_L}({\overline\Sigma_L}^{(\sigma/\tau)\times\rho_Z})$ is either empty or an $\Aff_{\overline\Sigma_X}(\overline\Sigma_X^{\sigma/\tau})$-torsor; the non-emptiness is equivalent to the existence of a function $\phi_{\sigma/\tau}\in \sPL(\overline\Sigma_X^{\sigma/\tau})$ with 
\begin{equation}\label{eq:boundaryexprforloglinebundle}
\left(\mc O_{X_\sigma}(\phi_{\sigma/\tau}) \otimes \mc L\right)\vert_{V(\sigma)} \cong \mc O_{V(\sigma)} \ .
\end{equation}

For the last statement, given two cells $\sigma_1/\tau_1$ and $\sigma_2/\tau_2$ of $\TU$, it is immediate from \eqref{eq:boundaryexprforloglinebundle} that the difference $\phi_{\sigma_2/\tau_2}-\phi_{\sigma_1/\tau_1}$ is affine on $\overline\Sigma_X^{{\sigma}_1/{\tau}_1}\cap \overline\Sigma_X^{\sigma_2/\tau_2}$ and hence $(\overline\Sigma_X^{\sigma/\tau},\phi_{\sigma/\tau})_{\sigma/\tau\subseteq \TU}$ represents a tropical Cartier divisor. It is the Cartier divisor associated to the canonical piecewise linear section 
\begin{equation}
    \overline \Sigma_X\xrightarrow{\cong} \overline \Sigma_X\times\{0\}\subseteq \overline \Sigma_X\times \overline\rho_Z\subseteq \Trop(L) 
\end{equation}
restricted to $\TU$, thus concluding the proof.
\end{proof}

\begin{proposition}
\label{prop:tropicalization of line bundles commutes with toroidal pull-backs}
Let $f\colon X\to Y$ be a morphism of toroidal varieties, and denote $\TF:=\Trop(f)$. Let $ L$ be a line bundle on $Y$, and let $\TU\subseteq \overline\Sigma_Y$ such that $\Trop(L)$ is a tropical line bundle on $\TU$. Then $\Trop(f^{*} L)$ is a line bundle on ${\TF^{-1}\TU}$ and we have
\begin{equation}
    \Trop(f^{\ast} L)\vert_{{\TF^{-1}\TU}}\cong {\TF^{\ast}(\Trop(L)}\vert_\TU) \ .
\end{equation}
\end{proposition}

\begin{proof}
Let $\mc L$ denote the invertible sheaf associated to $L$.
By Proposition \ref{prop:tropicalization of line bundle is tropical line bundle}, for every $\sigma/\tau$ of $\TU$ there exists a function $\phi_{\sigma/\tau}\in \sPL(\overline\Sigma_Y^{\sigma/\tau})$ such that 
\begin{equation}
\left(\mc O_{Y_{\sigma}}(\phi_{\sigma/\tau}) \otimes \mc L\right)\vert_{V(\sigma)} \cong \mc O_{V(\sigma)} \ ,
\end{equation}
and $\Trop(L)\vert_\TU$ corresponds to the Cartier divisor $D = (\overline\Sigma_Y^{\sigma/\tau},\phi_{\sigma/\tau})_{\sigma/\tau\subseteq \TU}$. For every cell $\tilde\sigma/\tilde\tau$ of ${\TF^{-1}\TU}$, we have ${\TF(\overline\Sigma_X^{\tilde\sigma/\tilde\tau})}\subseteq \overline\Sigma_Y^{\sigma/\tau}\subseteq \TU$ for the inclusion-minimal $\sigma/\tau\subseteq \overline\Sigma_Y$ with $\TF(\tilde\sigma/\tilde\tau)\subseteq \sigma/\tau$. On the algebraic side, $f$ induces a morphism $f\vert_{V( \tilde\sigma)}\colon V(\tilde\sigma)\to V(\sigma)$ and we have
\begin{multline}
    \mc O_{V(\tilde\sigma)} \cong  
        f\vert_{V(\tilde\sigma)}^{\ast} (\left(\mc O_{Y_{\sigma}}(\phi_{\sigma/\tau}) \otimes \mc L\right)\vert_{V(\sigma)})\cong \\
     \left(f^*\mc O_{Y_{\sigma}}(\phi_{\sigma/\tau}) \otimes f^{\ast}\mc    L\right)\vert_{V(\tilde\sigma)} \cong 
        \left(\mc O_{X_{\tilde\sigma}}((\TF^*\phi_{\sigma/\tau})\vert_{\overline\Sigma_X^{\tilde\sigma/\tilde\tau}}) \otimes f^{\ast}\mc L\right)\vert_{V(\tilde\sigma)} \ .
\end{multline}
Using Proposition \ref{prop:tropicalization of line bundle is tropical line bundle}, we see that $\Trop(f^*L)$ is a tropical line bundle on $\TF^{-1}\TU$ and that it is the line bundle associated to the tropical Cartier divisor
\begin{equation}
    E\coloneqq \left(\overline\Sigma_X^{\tilde\sigma/\tilde\tau}, (\TF^*\phi_{\sigma/\tau})\vert_{\overline\Sigma_X^{\tilde\sigma/\tilde\tau}}\right)_{\tilde\sigma/\tilde\tau\subseteq \TF^{-1}\TU}  \ .
\end{equation}
Noting that $E=\TF^*D$ concludes the proof.
\end{proof}

\begin{proposition}
\label{prop:if O(phi) is line bundle, then phi is CP}
Let $\phi\in \sPL(\Sigma_X)$. Then $\Trop(\mc O_X(\phi))$ is a tropical line bundle on $\overline \Sigma^{\sigma/\sigma}_X$ if and only if $\phi\in\CPL(\sigma)$.
\end{proposition}

\begin{proof} 
Assume $\Trop(\mc O_X(\phi))$ is a tropical line bundle on $\overline \Sigma^{\sigma/\sigma}_X$. By  Proposition \ref{prop:tropicalization of line bundle is tropical line bundle}, there is a function $\phi_\sigma \in \sPL({\overline\Sigma}_X^{\sigma/\sigma})$ such that 
\begin{equation}
    \mc O_{X_\sigma}(\phi_\sigma+\phi)\vert_{V(\sigma)}\cong 
        (\mc O_{X_\sigma}(\phi_\sigma)\otimes \mc O_X(\phi))\vert_{V(\sigma)}
            \cong \mc O_{V(\sigma)} \ .
\end{equation}
By definition of $\Aff_{\Sigma_X}$ , this means that 
\begin{equation}
    \phi_\sigma + \phi\in \Aff_{\Sigma_X}(\Sigma^\sigma_X) \ ,
\end{equation}
and after subtracting a constant we may assume that $\phi_\sigma$ vanishes on $\sigma$; that means $\phi_\sigma+\phi$ is an affine function on $\Sigma_X^\sigma$ that restricts to $\phi\vert_{\sigma}$ on $\sigma$,  which shows that $\phi\in \CPL(
\sigma)$.

Conversely, suppose that $\phi\in  \CPL(\sigma)$, and let $\xi\in \Aff_{\overline\Sigma_X}(\Sigma_X^\sigma)$ such that $\phi\vert_\sigma=\xi\vert_\sigma$. Then for every cell $\delta/\tau$ of $\overline\Sigma_X^{\sigma/\sigma}$ (that means for all cones $\delta$ and $\tau$ with $\tau\subseteq\sigma\subseteq \delta$), the function $\xi-\phi$ vanishes on $\tau$ and thus defines an element $(\xi-\phi)\vert_{\Sigma_X^{\delta}}\in \sPL(\overline\Sigma_X^{\delta/\tau})$. Moreover, we have 
\begin{equation}
    (\mc O_{X_\delta}((\xi-\phi)\vert_{\Sigma_X^{\delta}})\otimes \mc O_X(\phi))\vert_{V(\delta)}\cong 
        \mc O_{X_\delta}(\xi\vert_{\Sigma_X^\delta})\vert_{V(\delta)} \cong
            \mc O_{V(\delta)} \ ,  
\end{equation}
where the second isomorphism follows from the fact that $\xi\vert_{\Sigma_X^\delta}$ is affine at $\delta$ by Lemma \ref{lem: compatibility of affine functions}. Applying Proposition \ref{prop:tropicalization of line bundle is tropical line bundle} with $\phi_{\delta/\tau}=(\xi-\phi)\vert_{\Sigma_X^\delta}$ concludes the proof.
\end{proof}

\begin{corollary}
\label{cor:tropicalizable bundles}
Let $\mc L$ be an invertible sheaf on a toroidal variety $X$. Then $\Trop(\mc L)$ is a tropical line bundle on $\overline\Sigma_X$ if and only if $\mc L\cong\mc O_X(\phi)$ for some $\phi\in \CPL(\Sigma_X)$. In particular, if $\chi\in \sPL(\Sigma_X)$, then $\Trop(\mc O_X(\chi))$ is a tropical line bundle on $\overline \Sigma_X$ if and only if $\chi\in\CPL(\Sigma_X)$. 
\end{corollary}

\begin{proof}
If $\Trop(\mc L)$ is a tropical line bundle on $\overline\Sigma_X$, then $\mc L\cong \mc O_X(\phi)$ for some $\phi\in \sPL(\Sigma)$ by Proposition \ref{prop:tropicalization of line bundle is tropical line bundle} (take $\phi=-\phi_{0_\Sigma/0_\Sigma}$ in the notation of the proposition). By Proposition \ref{prop:if O(phi) is line bundle, then phi is CP}, we have $\phi\in \CPL(\Sigma_X)$.

Conversely, suppose $\phi\in \CPL(\Sigma_X)$. Then $\Trop(\mc O_X(\phi))$ is a tropical line bundle on $\overline\Sigma^{\sigma/\sigma}_X$ for all $\sigma\in \Sigma$ by Proposition \ref{prop:if O(phi) is line bundle, then phi is CP}. Since $\overline\Sigma_X=\bigcup_{\sigma\in \Sigma}\overline\Sigma_X^{\sigma/\sigma}$, it follows that $\Trop(\mc O_X(\phi))$ is a tropical line bundle on all of $\overline\Sigma_X$.

Finally, if $\Trop(\mc O_X(\chi))$ is a tropical line bundle, then by what we have just shown we have $\mc O_X(\chi)\cong \mc O_X(\phi)$ for some $\phi\in\CPL(\Sigma_X)$. This implies $\chi-\phi\in \Aff_{\Sigma_X}(\Sigma_X)$ and hence $\chi\in \CPL(\Sigma_X)$ because every affine function is combinatorially principal.
\end{proof}

\section{Tropicalization of $\psi$ classes}

We turn our attention to the tautological cotangent line bundle to a section of a family of curves and study its tropicalization.

\subsection{Tropicalization of the cotangent line bundle}

In \cite[Definition 6.16]{psi-classes}, tropical psi classes are defined as first Chern classes of certain torsors over the sheaf of affine functions of the base of a family of tropical curves. In this section we develop language and describe suitable conditions that  allow us to compare the tropicalization of the $i$-th cotangent line bundle with the torsors above.

\begin{definition}\label{def:Affphi}
Let $\Sigma$ be a cone complex with a sheaf of affine functions $\Aff_\Sigma$, let $\TU\subseteq \overline\Sigma$ be an open subcomplex, and let $\phi\in \sPL(\Sigma\cap \TU)$. We define the sheaf $\Aff_\TU(\phi)$ by 
\begin{equation}
    \Aff_\TU(\phi)(\TV)\coloneqq \{\chi\in \Aff_\Sigma(\TV\cap \Sigma)\mid \chi+\phi \text{ extends to a section of } \PLfin(\TV)\}
\end{equation}
for open subsets $\TV\subseteq \TU$. Note that $\Aff_\TU(\phi)$ is a pseudo-torsor over $\Aff_\TU$, that is it is a torsor whenever it is non-empty. If $\Aff_\TU(\phi)$ is a torsor, it defines a tropical line bundle, which we also denote by $\Aff_\TU(\phi)$ by abuse of notation.
\end{definition}

\begin{proposition}
\label{prop:torpicalization of boundary bundle is tropical boundary bundle}
Let $X$ be a toroidal variety,
let $\phi\in \sPL(\Sigma_X)$ and let $\TU\subseteq \overline\Sigma_X$. If $\Aff_\TU(\phi)$ is a tropical line bundle, then $\Trop(\mc O_X(\phi))\vert_\TU$ is a tropical line bundle as well and 
\begin{equation}
    \Trop(\mc O_X(\phi))\vert_\TU \cong \Aff_\TU(\phi)
\end{equation}
\end{proposition}

\begin{proof}
By Proposition \ref{prop:tropicalization of line bundle is tropical line bundle},  $\Trop(\mc O_X(\phi))\vert_\TU$ is a tropical line bundle if and only if for every cell $\sigma/\tau$ of $\TU$ there exists $\phi_{\sigma/\tau}\in \sPL(\overline\Sigma^{\sigma/\tau}_X)$ with
\begin{equation}\label{eq:isial}
    \mc O_{X_\sigma}(\phi_{\sigma/\tau}+\phi)\vert_{V(\sigma)}\cong \mc O_{V(\sigma)} \ ;
\end{equation}
the collection of all the $-\phi_{\sigma/\tau}$'s generates the pseudo-torsor $\Trop(\mc O_X(\phi))\vert_{\TU}(\overline\Sigma^{\sigma/\tau}_X)$ of local affine sections. 
By definition of $\Aff_{\overline\Sigma_X}$, equation \eqref{eq:isial} means that $\phi_{\sigma/\tau}+\phi\in \Aff_{\overline\Sigma_X}(\overline\Sigma_X^{\sigma/\tau}\cap \Sigma_X)$, which is equivalent to  $-(\phi_{\sigma/\tau}+\phi)\in \Aff_\TU(\phi)(\overline\Sigma_X^{\sigma/\tau})$. The $\Aff_{\TU}(\overline\Sigma^{\sigma/\tau}_X)$-equivariant map
\begin{equation}
\label{eq:isomorphism of torsors}
    \Trop(\mc O_X(\phi))\vert_{\TU}(\overline\Sigma^{\sigma/\tau}_X) \to \Aff_\TU(\phi)(\overline\Sigma_X^{\sigma/\tau}), \;\; -\phi_{\sigma/\tau} \mapsto  \chi = -\phi_{\sigma/\tau} -\phi     
\end{equation}
is an isomorphism of pseudo-torsors. Therefore, $\Trop(\mc O_X(\phi))\vert_\TU$ being a line bundle is equivalent to $\Aff_\TU(\phi)$ being a torsor, and  the collection of maps \eqref{eq:isomorphism of torsors} for $\sigma/\tau\subseteq \TU$ defines an isomorphism between $\Trop(\mc O_X(\phi))\vert_\TU$ and $\Aff_\TU(\phi)$. 
\end{proof}

We  now apply these construction to our main objects of interest: families of pointed curves and their tropicalizations.

\begin{proposition}
\label{prop:Aff(si) is boundary divisor}
Let $\mc C\to \mc B$ be a tropicalizable family of $n$-marked stable curves with associated family of tropical curves $ \overline\TC\to  \overline\TB$, and let $1\leq i \leq n$. Let $\phi$ be the piecewise linear function that has slope $1$ on the ray $\rho_i$ corresponding to the $i$-th section $s_i$ and is $0$ on all cones not containing $\rho_i$. Then there exists an isomorphism of pseudo $\Aff_{\overline\TC}$-torsors
\begin{equation}
    \Aff_{\overline{\mtt C}}(\ts_i)\cong \Aff_{\overline {\mtt C}}(\phi) \ .
\end{equation}
\end{proposition}

\begin{proof}
Let $\TU\subseteq \overline{\mtt C}$ be open. 
The sections of $\Aff_{\overline{\mtt C}}(\phi)(\TU)$ are  affine functions $\chi$ on $\TU\cap \mtt C$ for which $\chi+\phi$ can be extended continuously to $\TU$. Since $\phi$ has slope $1$ on $\rho_i$ and vanishes on cones that do not contain $\rho_i$, this is equivalent to saying that $\chi$ extends continuously to $\TU\setminus \ts_i$ and has slope $-1$ in the direction of $\rho_i$ near $\ts_i$. By Definition \ref{def:affine structure at infinity}, the continuous extension of $\chi$ to $\TU\setminus \ts_i$ is automatically affine on $\TU\setminus \ts_i$. We conclude that $\Aff_{\overline{\mtt C}}(\phi)(\TU)$ is in natural bijection with the subset of $\Aff_{\overline\TC}(\TU\setminus \ts_i)$ of affine functions that approach $\ts_i$ with slope $-1$. But this set equals $\Aff_{\overline\TC}(\ts_i)(\TU)$ by definition.
\end{proof}
    
\begin{definition}\label{def:Li}
 Let $\mc C \to \mc B$ be tropicalizable a family of $n$-marked stable curves with tropicalization $\overline{\TC}\to \overline\TB$. If  $\Aff_{\overline{\TC}}(\ts_i)$ is a tropical line bundle, we define the $i$-th tropical cotangent bundle analogously to \cite[Definition 6.16]{psi-classes}:
\begin{equation}
    \bL_i^\trop :=  (\ts_i)^*\Aff_{\overline \TC}(-\ts_i)  \ , 
\end{equation}  
where $\ts_i=\Trop(s_i)$ is the tropicalization of the $i$-th section
\end{definition}

\begin{theorem}
\label{thm:if psi is tropicalizable, then it tropicalizes to psi}
With setup and notation as in Definition \ref{def:Li}, assume $\Aff_{\overline{\TC}}(\ts_i)$ is a tropical line bundle. Then $\bL_i$ tropicalizes to a tropical line bundle and we have
\begin{equation}
    \Trop(\bL_i)\cong \bL_i^\trop \ .
\end{equation}
\end{theorem}

\begin{proof}
Let $\phi$ be the function on $\mtt C$ that has slope $1$ on the ray $\rho_i$ corresponding to $s_i$ and is $0$ on all cones of $\mtt C$ that do not contain $\rho_i$. By Proposition \ref{prop:Aff(si) is boundary divisor}, we have $\Aff_{\overline {\mtt C}}(\phi)\cong \Aff_{\overline {\mtt C}}(\ts_i)$. Using the assumption that $\Aff_{\overline {\mtt C}}(\ts_i)$ is a tropical line bundle, we conclude that $\Aff_{\overline {\mtt C}}(\phi)$ is a tropical line bundle as well. Applying Proposition \ref{prop:torpicalization of boundary bundle is tropical boundary bundle}, we see that the bundle $\mc O_{\mc C}(s_i)\cong  \mc O_{\mc C}(\phi)$ is tropicalizable and that the second isomorphism in the chain of isomorphisms
\begin{equation}
    \Trop(\mc O_{\mc C}(s_i))\cong \Trop(\mc O_{\mc C}(\phi)) \cong \Aff_{\overline \TC}(\phi)\cong \Aff_{\overline \TC}(\ts_i) 
\end{equation}
exists, the other ones being evident. Dualizing and applying Proposition \ref{prop:tropicalization of line bundles commutes with toroidal pull-backs} to the section $s_i$ then shows that $\Trop(\bL_i)$ is defined and that we have a sequence of isomorphisms
\begin{multline}
    \bL_i^\trop =
        \ts_i^*\Aff_{\overline \TC}(-\ts_i) \cong\\ 
    \cong(\ts_i)^*\Trop(\mc O_{\mc C}(-s_i))\cong   
        \Trop(s_i^*\mc O_{\mc C}(-s_i))\cong 
            \Trop(\bL_i) \ .
   \left.\right. 
\end{multline}
\end{proof}

\subsection{Tropicalizations of Cycles and Tropical Cycles}

In order to translate Theorem \ref{thm:if psi is tropicalizable, then it tropicalizes to psi} into a statement about $\psi$ classes, we need to take the first Chern class of a tropical line bundle. In this section we present the necessary bits of tropical intersection theory  for this task. Similar definitions and constructions appeared in \cite[Section 6]{psi-classes}; here we present them in the slightly larger generality needed for the current context.

\begin{definition}
Let $X$ be a complete toroidal variety with 
associated complex $\Sigma_X$, and let $c\in A^\ast(X)$. Then the tropicalization of $c$ is given by the weight
\begin{equation}
    \Trop_{X}(c)\colon \Sigma_X\to \Z,\;\;\sigma \mapsto \int_{X}  c\cdot[V(\sigma)]
\end{equation}
\end{definition}

\begin{proposition}
\label{prop:Tropicalizations are balanced}
Let $X$ be an $n$-dimensional complete toroidal variety and let $c\in A^k(X)$. Then the $(n-k)$-weight $\Trop(c)$ is balanced.
\end{proposition}

\begin{proof}
Balancing can be checked locally around each cone $\sigma\in \Sigma_{X}$ of dimension $n-k-1$. The weight $\Trop_X(c)$ induces a weight on $\Sigma^\sigma_{X}/\sigma$ and this weight coincides with $\Trop(\iota^*c)$, where $\iota\colon V(\sigma)\to X$ is the inclusion. Because the functions near $\sigma$ that vanish on $\sigma$ are  pulled back from $\Sigma_{X}^\sigma/\sigma\cong \Sigma_{V(\sigma)}$ by Lemma \ref{lem: relation between normal bundle functions}, it suffices to prove that $\Trop(\iota^*c)$ is balanced on $\Sigma_{V(\sigma)}$. We have reduced to the case where $k=n-1$ and $\Trop(c)$ is $1$-dimensional. Let $\phi\in \Aff_{\Sigma_X}(\Sigma_X)$. We have 
\begin{equation}
\label{eq:balancing of tropicalization}
    \sum_{\rho\in \Sigma_X(1)}\mathrm{slope}_\rho(\phi) c(\rho) =
        \int_X  c\cdot \left(\sum_{\rho\in\Sigma_X(1)} \mathrm{slope}_\rho(\phi) [V(\rho)]\right) = \int_X c \cdot c_1(\mc O_X(\phi)) 
            = 0 \ ,
\end{equation}
where the last equality follows from the fact that $\mc O_X(\phi)$ is trivial because $\phi$ is affine. The vanishing of the first term in \eqref{eq:balancing of tropicalization}  shows balancing, thus concluding the proof.
\end{proof}

\begin{corollary}
\label{cor:fundamental cycle exists on tropicalizations}
Let $X$ be an $n$-dimensional complete toroidal variety. The constant weight with value $1$ on the $n$-dimensional cones of $\Sigma_X$ is balanced. We denote by $[\Sigma_X]$ the corresponding tropical cycle.
\end{corollary}
    
\begin{proof}
The constant weight $1$ on $n$-dimensional cones is equal to $\Trop_X(1)$, where $1\in A^0(X)$ is the multiplicative identity of the Chow ring. The assertion follows from Proposition \ref{prop:Tropicalizations are balanced}.
\end{proof}

\begin{definition}
For a complete toroidal variety $X$ of dimension $n$ we call the tropical cycle represented by the balanced $n$-weight with value $1$ on all $n$-dimensional cones of $\Sigma_X$ the \textbf{fundamental class} of $\Sigma_X$ and denote it by $[\Sigma_X]$
\end{definition}

\begin{proposition} \label{prop:c1}
Let $X$ be a complete toroidal variety with cone complex $\Sigma_X$, let $c\in A^k(X)$, and let $\mc L$ be a tropicalizable line bundle on $X$ with tropicalization $\mtt L$ on $\overline \Sigma_X$. Then we have
\begin{equation}
\label{eq:trop commutes with intersection}
    [\Trop_X(c_1(\mc L) \cdot c)] = c_1(\mtt L) \frown [\Trop_X(c)]    
\end{equation}
\end{proposition}

\begin{proof}
By Corollary \ref{cor:tropicalizable bundles}, there exists a function $\phi\in \CPL(\Sigma_X)$ with $\mc L\cong \mc O_X(\phi)$. Let $\tau\in \Sigma_X$ be a cone of dimension $n-k-1$, where $n=\dim(X)$. We need to show that the  weights on both sides of \eqref{eq:trop commutes with intersection} have the same value on $\tau$. Since $\phi\in \CPL(\Sigma_X)$, we may assume that $\phi\vert_\tau=0$. We denote by $\phi^\tau$ the induced function on $\Sigma_X^\tau/\tau$. Then we have $\mc O_X(\phi)\vert_{V(\tau)}\cong \mc O_{V(\tau)}(\phi^\tau)$ and thus
\begin{equation}
    c_1(\mc L)\cdot [V(\tau)] = \sum_{\sigma} \mathrm{slope}_{\sigma/\tau}(\phi^\tau) [V(\sigma)] \quad \text{in }A^{k+1}(X) \ ,
\end{equation}
where the sum is taken over all $(n-k)$-dimensional cones of $\Sigma_X$ containing $\tau$. Therefore, we have
\begin{equation}
\label{eq:computation of intersection product}
    \Trop_X(c_1(\mc L)\cdot c)(\tau)= \int_{V(\tau)} c_1(\mc L)\cdot c = \sum_{\sigma} \mathrm{slope}_{\sigma/\tau}(\phi^\tau) \Trop_X(c)(\sigma) \ . 
\end{equation}
By Proposition \ref{prop:tropicalization of line bundle is tropical line bundle}, the function $-\phi^\tau$ is the local equation on $\overline\Sigma^{\tau/\tau}$ for a tropical Cartier divisor whose associated line bundle is $\TL$. Therefore, the last term in \eqref{eq:computation of intersection product} equals the local expression for the weight at $\tau$ of $c_1( \mtt L)\ \frown  [\Trop_X(c)]$ (compare with \eqref{eq:tropical intersection product}).
\end{proof}

\begin{corollary}\label{cor:psi}
Let $\mc C\to \mc B$ be a tropicalizable family of $n$-marked stable curves with tropicalization $ \mtt{C}\to  \mtt{B}$, and let $1\leq i \leq n$. Assume $\Aff_{\overline{\mathtt C}}(-\ts_i)$ is a tropical line bundle. Then 
 \begin{equation}\label{eq:pritsop}
     \Trop(\psi_i) = \psi_i^\trop \ .
\end{equation}
\end{corollary}

\begin{proof}
Combine Theorem \ref{thm:if psi is tropicalizable, then it tropicalizes to psi} and Proposition \ref{prop:c1}.    
\end{proof}

\section{An extended example in genus one}

This section is dedicated to an extended example of the theory set-up in the previous part of the paper. We study two $2$-dimensional families of genus-one curves with two marked points. They arise from a single family of admissible covers by forgetting different subsets of the marked ramification. We exhibit the tropical $\psi$ class as a $1$-dimensional cycle on the base and show it is the tropicalization of the corresponding algebraic cycle.

\subsection{Algebraic set-up}
\label{sec:algsetup}

We consider the (algebraic) moduli space of admissible covers
\begin{equation}
    \overline{Adm}_{1\to 0}((3), (2,1)^4) \ .
\end{equation}
Its closed points parameterize degree-$3$ admissible covers of a rational curve, where one branch point is the image of exactly one ramification point (of order $3$), whereas the other four branch points are simple: their inverse image consists of two points, one unramified and one simply ramified.
There is a natural branch morphism:
\begin{equation}
    \br:\overline{Adm}_{1\to 0}((3), (2,1)^4)\to \overline{\calM}_{0,5} \ ,
\end{equation}
recording the target of the cover together with the marked branch points. It is a finite morphism of degree $9$.
We adopt the convention that all inverse images of branch points are marked, and therefore we have a natural source morphism:
\begin{equation}
    \src:\overline{Adm}_{1\to 0}((3), (2,1)^4)\to \overline{\calM}_{1,9} \ .
\end{equation}
Let $r_1: \overline{\calM}_{1,9} \to \overline{\calM}_{1,2}$ (resp. $r_2$) be the morphism that forgets all marks except (the one corresponding to) the point of triple ramification and one of the unramified (resp. simply ramified) marked points. We adopt the notation that the first mark in $\overline{\calM}_{1,2}$ is the point of triple ramification of the cover. 

One may compose $r_i\circ \src$, and pull-back the universal family of $\overline{\calM}_{1,2}$ to obtain a family of genus-one, two-marked curves: 
\begin{equation}
     \xymatrix{ \mathcal{C}_i \ar[rr]\ar[d]& &\mathcal{U}_{1,2} \ar[d]\\
    *+[l]{\overline{Adm}_{1\to 0}((3), (2,1)^4) =  \mc B_i }\ar[rr]^{\varphi_i = r_i\circ \src}  & &\overline{\calM}_{1,2} \ .
    }
\end{equation}

We denote the space of admissible covers by $\mc B_i$ both to shorten the notation and to emphasize that we think of it as the base of a family of curves, determined by the map $\varphi_i: \mc B_i\to \overline{\calM}_{1,2}$.

\begin{lemma}
We determine the degree of the maps $\varphi_1, \varphi_2$.
\begin{enumerate}
    \item The degree of $\varphi_1: \mc B_1\to \overline{\calM}_{1,2}$ is equal to $24$, i.e.\ 
    \begin{equation}
        \varphi_{1, \ast}([\mc B_1]) = 24[\overline{\calM}_{1,2}] \ .
    \end{equation}
    \item The degree of $\varphi_2: \mc B_2\to \overline{\calM}_{1,2}$ is equal to $6$, i.e.\ 
    \begin{equation}
        \varphi_{2, \ast}([\mc B_2]) = 6[\overline{\calM}_{1,2}] \ .
    \end{equation}
  \end{enumerate}
\end{lemma}

\begin{proof} 
Consider a general point $[(E, p_1, p_2)]\in \overline{\calM}_{1,2}$, i.e.\ a smooth genus-one curve with two marked points.
A map $f:E\to \PP^1$ in the inverse image of $[(E, p_1, p_2)]$ vie $\varphi_1$ corresponds to a principal divisor of the form $3p_1 -p_2 - 2x$, where $x$ is a point  of $E$ to be determined. One can solve for $x$ in $Jac(E) \cong E$, and find that there are four solutions (pick any particular solution and then one can add any one of the four two-torsion points of $Jac(E)$). There are therefore four distinct maps that have triple ramification at $p_1$ and are unramified at $p_2$ (but where $p_2$ is the inverse image of a branch point). Generically, for each of these maps there are $3!$ ways to label the remaining simple ramification points, and once that is done the remaining marks are determined. There is therefore a total of $24$ inverse images, which proves $(1)$.

Similarly, to determine the inverse images via $\varphi_2$ of $[(E, p_1, p_2)]$ one must solve the equation $3p_1-2p_2-x = 0 \in Jac(E)$, which has a unique solution. There is only one map with triple ramification at $p_1$ and simple ramification at $p_2$ and $3! = 6$ ways to label the remaining simple ramification points, for a total of six inverse images. This proves (2).
\end{proof}

We conclude the description of the algebraic set up by spelling out the relation between the class $\psi_1$ on $\overline{\calM}_{1,2}$ and the pushforward of the class $\psi_1$ on $\mc B_i$ via the morphism $\varphi_i$.

\begin{lemma}
\label{lem:pushpsi}
Denote by $W$ the class of the Weierstrass divisor in $\overline{\calM}_{1,2}$ (parameterizing curves of genus one where both marks are Weierstrass points of the same $g_2^1$). We have the following equalities in $A^1(\overline{\calM}_{1,2})$:
\begin{equation}\label{eq:firstpushpullrel}
    \varphi_{1, \ast}(\psi_1) = 24 \psi_1 \ ,
\end{equation}
\begin{equation}\label{eq:secondpushpullrel}
    \varphi_{2, \ast}(\psi_1) = 6 (\psi_1 +W) \ .
\end{equation}
\end{lemma}

\begin{proof}
Since we need to pull-back and push-forward classes on different spaces that are all called $\psi_1$ on their respective space, we adopt the convention of just letting the context indicate which class is which.
There are two fundamental pull-back relations that we will be using:
\begin{equation}
    \psi_1 = r_i^\ast(\psi_1) + D_i \ ,
\end{equation}
where $D_i$ denotes the sum of all rational tails boundary divisors of $\overline{\calM}_{1,9}$ parameterizing nodal curves where the mark $1$ lies on the rational tail, and the mark $2$ lies on the genus-one component of the curve. This relation follows from iterated application of the standard pull-back relation for $\psi$ classes along forgetful morphisms (see e.g.\ \cite[Lemma 1.3.1]{k:pc}).
With respect to the morphism $\src$,
\begin{equation}
   \psi_1 = \src^\ast(\psi_1) \ , 
\end{equation}
follows from the fact that the first  mark is a point of full ramification for the cover, that can therefore not live in an unstable component of the source curve of the admissible cover.

\begin{figure}
    \centering
    
\tikzset{every picture/.style={line width=0.75pt}} 

\begin{tikzpicture}[x=0.75pt,y=0.75pt,yscale=-.8,xscale=.8]

\draw    (55,247) -- (200,228) ;
\draw    (139,228) -- (290,248) ;
\draw    (153,173) -- (304,193) ;
\draw    (153,173) .. controls (126,174) and (126,158) .. (133,154) .. controls (140,150) and (152.68,134.58) .. (143,130) .. controls (133.32,125.42) and (92.62,123.99) .. (75,124) .. controls (57.38,124.01) and (27,150) .. (46,171) .. controls (65,192) and (136.37,185.47) .. (153,173) -- cycle ;
\draw   (147,135) .. controls (145,111) and (302,132) .. (308,156) .. controls (314,180) and (149,159) .. (147,135) -- cycle ;
\draw    (200,142) .. controls (222,138) and (227,141) .. (242,149) ;
\draw    (190,135) .. controls (203,147) and (220,156) .. (253,145) ;
\draw    (168,195) -- (168,217) ;
\draw [shift={(168,219)}, rotate = 270] [color={rgb, 255:red, 0; green, 0; blue, 0 }  ][line width=0.75]    (10.93,-3.29) .. controls (6.95,-1.4) and (3.31,-0.3) .. (0,0) .. controls (3.31,0.3) and (6.95,1.4) .. (10.93,3.29)   ;
\draw  [color={rgb, 255:red, 208; green, 2; blue, 27 }  ,draw opacity=1 ][fill={rgb, 255:red, 208; green, 2; blue, 27 }  ,fill opacity=1 ] (36,157) .. controls (36,154.79) and (37.79,153) .. (40,153) .. controls (42.21,153) and (44,154.79) .. (44,157) .. controls (44,159.21) and (42.21,161) .. (40,161) .. controls (37.79,161) and (36,159.21) .. (36,157) -- cycle ;
\draw  [color={rgb, 255:red, 208; green, 2; blue, 27 }  ,draw opacity=1 ][fill={rgb, 255:red, 208; green, 2; blue, 27 }  ,fill opacity=1 ] (304,156) .. controls (304,153.79) and (305.79,152) .. (308,152) .. controls (310.21,152) and (312,153.79) .. (312,156) .. controls (312,158.21) and (310.21,160) .. (308,160) .. controls (305.79,160) and (304,158.21) .. (304,156) -- cycle ;
\draw  [color={rgb, 255:red, 2; green, 2; blue, 208 }  ,draw opacity=1 ][fill={rgb, 255:red, 2; green, 2; blue, 208 }  ,fill opacity=1 ] (125,164) .. controls (125,161.79) and (126.79,160) .. (129,160) .. controls (131.21,160) and (133,161.79) .. (133,164) .. controls (133,166.21) and (131.21,168) .. (129,168) .. controls (126.79,168) and (125,166.21) .. (125,164) -- cycle ;
\draw  [color={rgb, 255:red, 2; green, 2; blue, 208 }  ,draw opacity=1 ][fill={rgb, 255:red, 2; green, 2; blue, 208 }  ,fill opacity=1 ] (126,182) .. controls (126,179.79) and (127.79,178) .. (130,178) .. controls (132.21,178) and (134,179.79) .. (134,182) .. controls (134,184.21) and (132.21,186) .. (130,186) .. controls (127.79,186) and (126,184.21) .. (126,182) -- cycle ;
\draw  [color={rgb, 255:red, 2; green, 2; blue, 208 }  ,draw opacity=1 ][fill={rgb, 255:red, 2; green, 2; blue, 208 }  ,fill opacity=1 ] (196,143) .. controls (196,140.79) and (197.79,139) .. (200,139) .. controls (202.21,139) and (204,140.79) .. (204,143) .. controls (204,145.21) and (202.21,147) .. (200,147) .. controls (197.79,147) and (196,145.21) .. (196,143) -- cycle ;
\draw  [color={rgb, 255:red, 2; green, 2; blue, 208 }  ,draw opacity=1 ][fill={rgb, 255:red, 2; green, 2; blue, 208 }  ,fill opacity=1 ] (238,149) .. controls (238,146.79) and (239.79,145) .. (242,145) .. controls (244.21,145) and (246,146.79) .. (246,149) .. controls (246,151.21) and (244.21,153) .. (242,153) .. controls (239.79,153) and (238,151.21) .. (238,149) -- cycle ;
\draw  [color={rgb, 255:red, 2; green, 2; blue, 208 }  ,draw opacity=1 ][fill={rgb, 255:red, 2; green, 2; blue, 208 }  ,fill opacity=1 ] (197,179) .. controls (197,176.79) and (198.79,175) .. (201,175) .. controls (203.21,175) and (205,176.79) .. (205,179) .. controls (205,181.21) and (203.21,183) .. (201,183) .. controls (198.79,183) and (197,181.21) .. (197,179) -- cycle ;
\draw  [color={rgb, 255:red, 2; green, 2; blue, 208 }  ,draw opacity=1 ][fill={rgb, 255:red, 2; green, 2; blue, 208 }  ,fill opacity=1 ] (239,185) .. controls (239,182.79) and (240.79,181) .. (243,181) .. controls (245.21,181) and (247,182.79) .. (247,185) .. controls (247,187.21) and (245.21,189) .. (243,189) .. controls (240.79,189) and (239,187.21) .. (239,185) -- cycle ;
\draw  [color={rgb, 255:red, 208; green, 2; blue, 27 }  ,draw opacity=1 ][fill={rgb, 255:red, 208; green, 2; blue, 27 }  ,fill opacity=1 ] (51,247) .. controls (51,244.79) and (52.79,243) .. (55,243) .. controls (57.21,243) and (59,244.79) .. (59,247) .. controls (59,249.21) and (57.21,251) .. (55,251) .. controls (52.79,251) and (51,249.21) .. (51,247) -- cycle ;
\draw  [color={rgb, 255:red, 2; green, 2; blue, 208 }  ,draw opacity=1 ][fill={rgb, 255:red, 0; green, 0; blue, 0 }  ,fill opacity=1 ] (127.5,237.5) .. controls (127.5,235.29) and (129.29,233.5) .. (131.5,233.5) .. controls (133.71,233.5) and (135.5,235.29) .. (135.5,237.5) .. controls (135.5,239.71) and (133.71,241.5) .. (131.5,241.5) .. controls (129.29,241.5) and (127.5,239.71) .. (127.5,237.5) -- cycle ;
\draw  [color={rgb, 255:red, 2; green, 2; blue, 208 }  ,draw opacity=1 ][fill={rgb, 255:red, 0; green, 0; blue, 0 }  ,fill opacity=1 ] (201,236) .. controls (201,233.79) and (202.79,232) .. (205,232) .. controls (207.21,232) and (209,233.79) .. (209,236) .. controls (209,238.21) and (207.21,240) .. (205,240) .. controls (202.79,240) and (201,238.21) .. (201,236) -- cycle ;
\draw  [color={rgb, 255:red, 2; green, 2; blue, 208 }  ,draw opacity=1 ][fill={rgb, 255:red, 0; green, 0; blue, 0 }  ,fill opacity=1 ] (240,241) .. controls (240,238.79) and (241.79,237) .. (244,237) .. controls (246.21,237) and (248,238.79) .. (248,241) .. controls (248,243.21) and (246.21,245) .. (244,245) .. controls (241.79,245) and (240,243.21) .. (240,241) -- cycle ;
\draw  [color={rgb, 255:red, 208; green, 2; blue, 27 }  ,draw opacity=1 ][fill={rgb, 255:red, 208; green, 2; blue, 27 }  ,fill opacity=1 ] (290,248) .. controls (290,245.79) and (291.79,244) .. (294,244) .. controls (296.21,244) and (298,245.79) .. (298,248) .. controls (298,250.21) and (296.21,252) .. (294,252) .. controls (291.79,252) and (290,250.21) .. (290,248) -- cycle ;
\draw  [color={rgb, 255:red, 2; green, 2; blue, 208 }  ,draw opacity=1 ][fill={rgb, 255:red, 2; green, 2; blue, 208 }  ,fill opacity=1 ] (300,193) .. controls (300,190.79) and (301.79,189) .. (304,189) .. controls (306.21,189) and (308,190.79) .. (308,193) .. controls (308,195.21) and (306.21,197) .. (304,197) .. controls (301.79,197) and (300,195.21) .. (300,193) -- cycle ;

\draw (334,239.4) node [anchor=north west][inner sep=0.75pt]    {$T$};
\draw (334,146.4) node [anchor=north west][inner sep=0.75pt]    {$C$};
\draw (24,133.4) node [anchor=north west][inner sep=0.75pt]    {$\textcolor[rgb]{0.82,0.01,0.11}{1}$};
\draw (306,129.4) node [anchor=north west][inner sep=0.75pt]    {$\textcolor[rgb]{0.82,0.01,0.11}{2}$};

\end{tikzpicture}

\caption{A topological cartoon of the admissible covers parameterized by the general points of $\src^\ast (D_2)$. The blue points are the marked points that get forgotten by the map $r_2$, which we did not label to avoid cluttering the figure. There are three possible ways to order the pairs of blue points, giving rise to three irreducible components.}
\label{fig:ibd}
\end{figure}
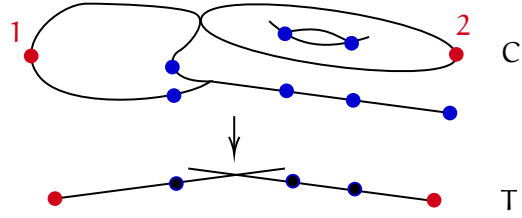

Next we analyze the divisor $\src^\ast (D_i)$. When $i=1$, i.e.\ when the second mark corresponds to a non-ramified point of the cover, then  one may observe that the image of $\src$ is disjoint from $D_1$, and therefore $\src^\ast (D_1) = 0$.

In the other case, $\src^\ast (D_2)$ is the sum of three irreducible boundary divisors in the space of admissible covers, depicted in Figure \ref{fig:ibd}. They parameterize the same topological type of covers, and differ only in the labeling of the  inverse images of the branch points. 

The final step is to apply the projection formula.
In the case $i=1$, we have that $\psi_1 = \varphi_1^\ast(\psi_1)$, and we have shown that $\varphi_1$ is a finite morphism of degree $24$, from which \eqref{eq:firstpushpullrel} follows.

In the second case, we have
\begin{equation}
    \varphi_{2,\ast}(\psi_1) = 6\psi_1 + \varphi_{2,\ast}(\src^\ast(D_2)) \ .   
\end{equation}
One finally observes that each of the three boundary divisors comprising $\src^\ast(D_2)$ pushes forward to $2W$ (the factor of $2$ coming from the gluing ghost automorphisms of the node), thus yielding \eqref{eq:secondpushpullrel}.
\end{proof}

\subsection{Tropical family}
\label{sec:tropfam}

We consider the moduli space of tropical admissible covers as introduced in \cite{Caporaso,CMRadmissible}
\begin{equation}
    \overline{Adm}^{\trop}_{1\to 0}((3), (2,1)^4) \ .
\end{equation}
It is a two-dimensional cone complex, with $20$ rays and $45$ two-dimensional cones which we call faces for simplicity. 
Forgetting the labeling of the five points on the target trees, one can group the rays into $4$ distinct types; similarly, the faces are grouped into $5$ types. The topological types of tropical covers corresponding to these types are depicted in Figure \ref{fig:raysandcones}. We  describe the various incidence properties of rays and faces.

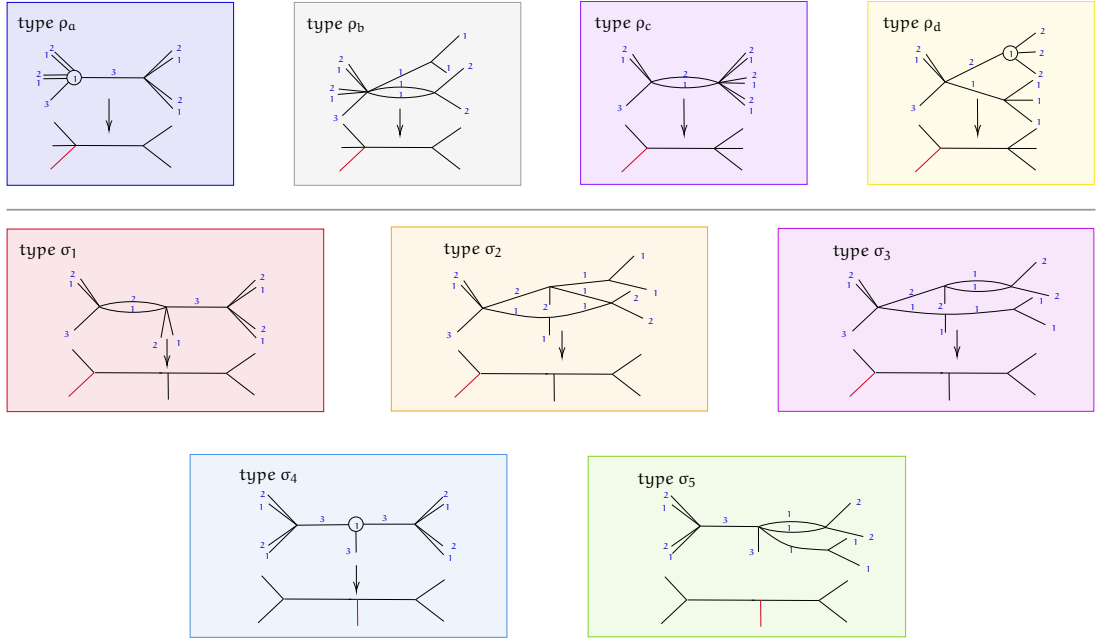
\begin{figure}
\centering
     \resizebox{\textwidth}{!}{\input{raysandcones.tikz}}
\caption{The topological types of tropical covers corresponding to the  rays and two-dimensional cones of $\overline{Adm}^{\trop}_{1\to 0}((3), (2,1)^4)$. Types are labeled as in  their description in Section \ref{sec:tropfam}. The color coding is meant to facilitate the understanding of the structure of the cone complex of tropical admissible covers, represented in Figure \ref{fig:brtrop}. }
    \label{fig:raysandcones}
\end{figure}

\noindent{\sc Rays:}
\begin{description}
    \item[type $\rho_a$] there are six rays of type $\rho_a$; each is adjacent to two faces of type $\sigma_1$, a face of type $\sigma_4$ and one of type $\sigma_5$;
        \item[type $\rho_b$] there are six rays of type $\rho_b$; each is adjacent to two faces of type $\sigma_2$, two faces of type $\sigma_3$ and one of type $\sigma_5$;
            \item[type $\rho_c$] there are four rays of type $\rho_c$; each is adjacent to three faces of type $\sigma_1$, and three faces of type $\sigma_2$;
                \item[type $\rho_d$] there are four rays of type $\rho_d$; each is adjacent to three faces of type $\sigma_3$. 
\end{description}

\noindent{\sc Faces:}
\begin{description}
    \item[type $\sigma_1$] there are twelve faces of type $\sigma_1$, each bounded by a ray of type $\rho_a$ and one of type $\rho_c$;
        \item[type $\sigma_2$] there are twelve faces of type $\sigma_2$, each bounded by a ray of type $\rho_b$ and one of type $\rho_c$;
          \item[type $\sigma_3$] there are twelve faces of type $\sigma_3$, each bounded by a ray of type $\rho_b$ and one of type $\rho_d$;
            \item[type $\sigma_4$] there are three faces of type $\sigma_4$, each bounded by two rays of type $\rho_a$;
              \item[type $\sigma_5$] there are six faces of type $\sigma_5$, each bounded by a ray of type $\rho_a$ and one of type $\rho_b$.
\end{description}

\subsection{Tropicalization and the fundamentalish cycle.}

The  tropicalization of the Hurwitz space (compactified in the space of admissible covers) is a cone complex which may be naturally identified with the Berkovich skeleton of the analytic Hurwitz  space. In \cite{CMRadmissible}, the authors exhibit a natural tropicalization map $\mtt T$ from the Berkovich analytic skeleton $\Sigma$ of the Hurwitz space ($\Sigma$ agrees with the cone complex associated to the toroidal compactification given by the space of admissible covers) to the cone complex of tropical admissible covers. This map is a strict morphism of cone complexes, but  in general it is not an isomorphism. Surjectivity fails if some tropical admissible cover is not dual to any algebraic one. This situation is detected by the vanishing of some local Hurwitz number associated to a vertex of some tropical admissible cover. Injectivity may fail for two reasons: multiple strata may give rise to the same topological type of tropical cover if they arise from the normalization of a singular locus in the space of admissible covers: this happens when in a tropical admissible cover, there are multiple edges  mapping to a given edge of the target graph, and their expansion factors  are not pairwise coprime. Alternatively, it may be that there are multiple algebraic covers of a component of the target curve: this situation is a bit more delicate and it is assessed by observing the local Hurwitz numbers and the automorphism factors around the vertices of a tropical admissible covers.

The tropical source and target morphisms factor via the tropicalization map $\mtt T$; in the specific case we are studying they give rise to the following commutative diagram:

\begin{equation}\label{eq:diagram}
    \xymatrix{
    \Sigma=\Sigma_{\mc B_i}=\Sigma_{\overline{Adm}_{1\to 0}((3), (2,1)^4)} \ar[dr]_{\mtt T}
    \ar[drrr]^{\mathtt{Src}} \ar[dddr]_{\mathtt{Br}}
    & & &
    \\ 
    &  \overline{Adm}^{\trop}_{1\to 0}((3), (2,1)^4) \ar[rr]^{\mathtt{src}} \ar[dd]_{\mathtt{br}}
    & & \overline{\calM}_{1,9}^\trop  
    \\
     & & & 
     \\
    &   \overline{\calM}_{0,5}^\trop  & &
    }
\end{equation}

By the functoriality of tropicalization (Proposition \ref{prop:functoriality}), any (local, affine) linear function on $\overline{\calM}_{1,9}^\trop$ (resp. $ \overline{\calM}_{0,5}^\trop$) pulls-back via $\mathtt{Src}$ (resp. $\mathtt{Br}$) to a linear function on the cone complex $\Sigma$ of $\overline{Adm}_{1\to 0}((3), (2,1)^4)$. 

We  declare that linear functions pulled back via $\mathtt{src}$ and $\mtt{br}$ are linear functions on the tropical moduli space $\overline{Adm}^{\trop}_{1\to 0}((3), (2,1)^4)$. By the commutativity of \eqref{eq:diagram} such a declaration is compatible with further pulling back via $\mtt T$. As a further consequence, any intersection-theoretic computation on $\Sigma$ involving cycles pulled back from source or target morphism may be reproduced, by the projection formula, on the moduli space $\overline{Adm}^{\trop}_{1\to 0}((3), (2,1)^4)$.

It is then necessary to study the push-forward of the fundamental class of $\Sigma$ via $\mtt T$.
We have 
\begin{equation}
    \mtt T_\ast([\Sigma]) = \sum_\sigma \omega(\sigma) \sigma \ , 
\end{equation}
where the sum runs over all top-dimensional cones of $\overline{Adm}^{\trop}_{1\to 0}((3), (2,1)^4)$. The weight $\omega(\sigma)$, for a cone $\sigma$ parameterizing tropical admissible covers $\Gamma \to \Delta$ of combinatorial type $\Theta$ is as in \cite[Definition 22]{CMRadmissible} the product of the following factors: 
\begin{enumerate}[(W1)]
    \item a factor of $\frac{1}{|\Aut_0(\Theta)|}$, with $\Aut_0(\Theta)$ denoting automorphisms of the tropical cover fixing the source curve;
    \item a factor of local Hurwitz numbers $\prod_{v\in V(\Gamma)} H(v)$;
    \item a factor of $\prod_{e\in E(\Delta)} w_e$, where $w_e$ is the product of the expansion factors above the edge $e$, divided by their LCM.
\end{enumerate}

For the moduli space $\overline{Adm}^{\trop}_{1\to 0}((3), (2,1)^4)$ that we are studying, all local Hurwitz numbers appearing are non-zero, hence the map $\mtt T$ is surjective.
Every time there are multiple edges mapping to an edge of a target graph, their expansion factors are pairwise coprime. All local Hurwitz numbers appearing are equal to one, except for two:
\begin{itemize}
    \item in the tropical cover parameterized by cone $\sigma_4$ the genus-one vertex has local Hurwitz number equal to $1/3$; this corresponds to the unique genus-one curve admitting a degree $3$ automorphism;
    \item  the central vertex in the tropical cover parameterized by cone $\sigma_5$ has local Hurwitz number equal to $2$, corresponding to different labelings of the three unramified points: there are in fact six possible labelings, but a factor of three is divided out because of the automorphism of the cover. 
\end{itemize}
The covers parameterized by cones of type $\sigma_3$ and $\sigma_5$ have a group of automorphisms of order two, while all other types have no non-trivial automorphisms. This information is systematically collected in Table \ref{tab:fundishclas}. 

In conclusion, if we  denote by $\sigma^\tot_i$ the formal sum of all two-dimensional cones of type $\sigma_i$,
we obtain:
\begin{equation} \label{eq:fundamentalish}
    \mtt T_\ast([\Sigma]) = \sigma^\tot_1+\sigma^\tot_2+\frac{1}{2}\sigma^\tot_3+\frac{1}{3}\sigma^\tot_4+\sigma^\tot_5 \ .    
\end{equation}
The cycle in \eqref{eq:fundamentalish} plays the role of the fundamental class in the forthcoming computations, and we therefore dub it the \emph{fundamentalish cycle} of the moduli space of tropical admissible covers. 

\begin{table}[tb]
    \centering
    \begin{tabular}{|c|c|c|c|c|}
    \hline
           Face type & W1 & W2 & W3 & $\omega(\sigma)$  \\
          \hline
         $\sigma_1$  & $1$ & $1$ & $1$ & $1$\\
            \hline
         $\sigma_2$  & $1$ & $1$ & $1$ & $1$\\
            \hline
         $\sigma_3$  & $\frac{1}{2}$ & $1$ & $1$ & $\frac{1}{2}$\\
            \hline
                     $\sigma_4$  & $1$ & $\frac{1}{3}$ & $1$ & $\frac{1}{3}$\\
            \hline
                     $\sigma_5$  & $\frac{1}{2}$ & $2$ & $1$ & $1$\\
            \hline
    \end{tabular}
    \caption{The factors computing the weights of two-dimensional cones in $\overline{Adm}^{\trop}_{1\to 0}((3), (2,1)^4)$.}
    \label{tab:fundishclas}
\end{table}

\subsection{The tropical $\psi$ class}

In this section we study the class $\psi_1$, supported on the mark with expansion factor $3$, in the two-dimensional family of tropical admissible covers introduced in the previous sections. Our first result is that the tropical psi class is in fact the tropicalization of the algebraic psi class.

\begin{proposition}
\label{prop:psi1works}
Consider the families of genus-one, $9$-marked algebraic and tropical curves given by the morphisms $\src$ and $\mathtt{src}$ introduced in the previous sections. We have:
\begin{equation}
    \psi_1^\trop = \Trop(\psi_1) \ ,    
\end{equation}
where we have omitted explicitly writing the pull-back via the source morphism in order to not burden the notation unnecessarily. 
\end{proposition}

\begin{proof}
By Theorem \ref{thm:if psi is tropicalizable, then it tropicalizes to psi}, it suffices to show that the (partial) affine structure on the  family of tropical admissible covers induced by  tropicalization  contains enough  information to define $\psi_1^\trop$.    Denote by $\hat{\psi}^\trop_1$ the psi class at the first mark of  $\overline{\calM}^\trop_{0,5}$.
By Theorem \ref{thm:cross-ratios on M0n}  we have have that the affine structure on  $\overline{\calM}^\trop_{0,5}$ induced by  tropicalization agrees with the one induced by tropical cross-ratios, which implies that
\begin{equation}
\hat\psi_1^\trop = \Trop(\hat\psi_1) \ .
\end{equation}

Consider the tropical branch morphisms $ \mathtt{br}$.  By the same argument as in the proof of  \cite[Prop 7.8, Eq. (145)]{psi-classes}
\begin{equation}
\label{eq:pbpsiviabrtrop}
    \mathtt{br}^\ast(\hat\psi^\trop_1) = 3 \psi^\trop_1 \ ,
\end{equation}
hence there are enough affine functions to define the class $3\psi_1^\trop$. Since we work with cycles with $\Q$ coefficients, this suffices to conclude the proof.
\end{proof}

\begin{remark}
As a sanity check, we observe that the steps in the proof of Proposition \ref{prop:psi1works} are compatible with functoriality of logarithmic tropicalization (\cite[Proposition 6.3]{U:art}) and Ionel's lemma \cite[Lemma 1.17]{ionel} which provides the comparison statement on the algebraic side: \begin{equation}
    \br^\ast(\hat\psi_1) = 3 \psi_1 \ .
\end{equation}
\end{remark}

Having established the tropicalization statement, we now describe the cycle obtained by capping the $\psi_1^{\trop}$ with the fundamentalish cycle \eqref{eq:fundamentalish}. 

\begin{proposition}\label{prop:psicomp}
Consider the family of tropical curves corresponding to the moduli function $\mathtt{src}: \overline{Adm}^\trop_{1\to 0}((3), (2,1)^4)\to \overline{\calM}^\trop_{1,9}. $
For $x = a,b,c,d$, denote by $\rho^\tot_x$ the formal sum of all rays of type $\rho_x$, as described in Section \ref{sec:tropfam}. We have
\begin{equation}
    \psi_1^\trop \frown \mtt T_\ast([\Sigma]) = \frac{2}{3}\rho^\tot_a+\rho^\tot_b \ .
\end{equation}
\end{proposition}

\begin{proof}
The  tropical cycle $\psi_1^\trop \frown \mtt T_\ast([\Sigma]) = \sum c_\rho \rho$ is a linear combination of the rays of $\overline{Adm}^\trop_{1\to 0}((3), (2,1)^4)$. The coefficient $ c_\rho$ in front of a given ray $\rho$ is obtained as follows: let $\psi_{1,\rho}$ be a piecewise linear function representing the cycle $\psi_1^{\trop}$ with the property that it has slope zero along $\rho$; then
\begin{equation}\label{eq:coeffofray}
    c_\rho = -\sum_{\sigma\succ \rho} \omega(\sigma) \psi_{1, \rho}(u_{\rho/\sigma}) \ ,
\end{equation}
where $u_{\rho/\sigma}$, the lattice normal vector to $\rho$ in $\sigma$, may in this case be taken to be the primitive vector for the other ray of $\sigma$.
We perform this computation for the four types of rays described in Section \ref{sec:tropfam}. First we introduce some notation:  for $i, j\in [5]$, denote by $\hat{\rho}_{i,j}$ the ray in $\overline{\calM}^\trop_{0,5}$ dual to the divisor $D_{\{i,j\}}$\footnote{the marks labeled $i,j$ are on one component of a nodal rational curve, and the remaining marks on the other.}, and $\varphi_{\hat{\rho}_{i,j}}$ the piecewise linear function with slope $-1$ along $\hat{\rho}_{i,j}$ and zero along all other rays. For $x\in \{a,b,c,d\}$, we denote by $\rho_{x|i,j}$ the ray of type $\rho_x$ in $\mathtt{br}^{-1}(\hat{\rho}_{i,j})$, and by $\varphi_{\rho_{x|i,j}}$ the corresponding piecewise linear function with slope $-1$ along the ray, and zero along all other rays.  Figure \ref{fig:brtrop} illustrates the situation and it may be helpful in following the computation.

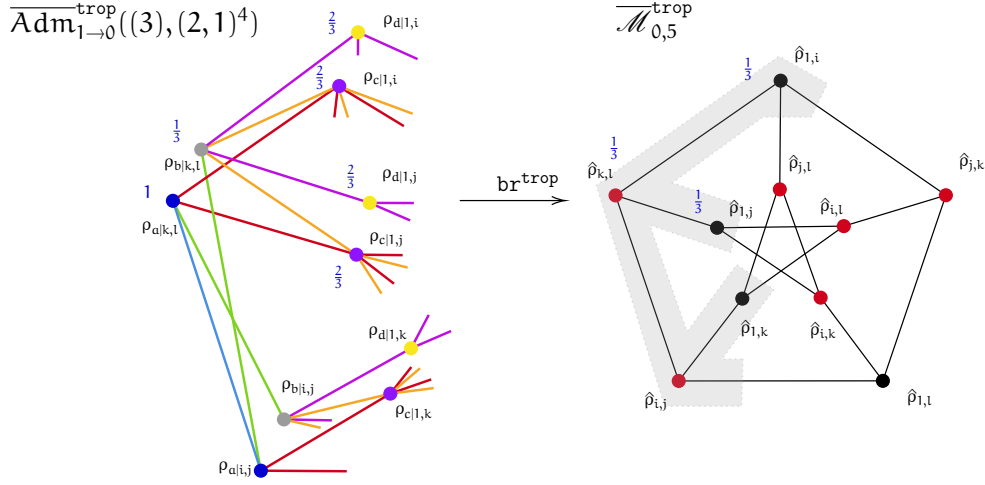
\begin{figure}
\centering
     \resizebox{ .9\textwidth}{!}{\input{brtrop.tikz}}
\caption{The restriction of the map $\mathtt{br}^{trop}$ to the inverse image of the shaded region. In blue are the non-zero slopes of the piecewise linear function representing $\frac{1}{3}\psi_1^\trop$ and of its pullback to the space of tropical admissible covers. The rays and faces of the space of tropical admissible covers are color coded by the type of cones as in Figure \ref{fig:raysandcones}.}
\label{fig:brtrop}
\end{figure}

\vspace{0.2cm}
\noindent\textsc{Type $\rho_a$:} for $i, j\not= 1$, consider $\rho_{a|i,j}$. In $\overline{M}_{0,5}$ one may represent the class $\psi_1$ as the following linear combination of the boundary divisors:
\begin{equation}
    \psi_1 = D_{1,i}+D_{1,j}+D_{k,l} \ ,
\end{equation}
where $\{i,j,k,l\} = \{2,3,4,5\}$.
Correspondingly, the piecewise linear function on $\overline{\calM}^\trop_{0,5}$ given by 
\begin{equation}\label{eq:goodrepa}
    \hat{\rho}_{1,i}+\hat{\rho}_{1,j}+\hat{\rho}_{k,l} \ ,
\end{equation}
may be chosen to represent $\hat{\psi}^\trop_1$.
By \eqref{eq:pbpsiviabrtrop}, $\psi_1^\trop = \frac{1}{3}\mathtt{br}^\ast(\hat{\psi}^\trop_1)$, and pulling back the representative from \eqref{eq:goodrepa} we have a piecewise linear function representing $\psi_1^\trop$ which has slope zero along the ray $\rho_{a|i,j}$. Precisely:, $\psi_1^\trop$ is represented by
\begin{equation}\label{eq:psionetrop}
    \frac{2}{3}\varphi_{\rho_{c|1,i}}+\frac{2}{3}\varphi_{\rho_{d|1,i}} +
        \frac{2}{3}\varphi_{\rho_{c|1,j}}+\frac{2}{3}\varphi_{\rho_{d|1,j}} +
            \varphi_{\rho_{a|k,l}}+\frac{1}{3}\varphi_{\rho_{b|k,l}} \ . 
\end{equation}
There are only two two-dimensional cones containing $\rho_{a|i,j}$ such that the function representing $\psi_1^\trop$ is not identically zero: a cone of type $\sigma_4$ whose other ray is $\rho_{a|k,l}$ and a cone of type $\sigma_5$ whose other ray is $\rho_{b|k,l}$. We evaluate \eqref{eq:coeffofray} to obtain:
\begin{equation}\label{eq:coeffa}
    c_{\rho_{a|i,j}} = \frac{1}{3}\cdot 1+ 1\cdot \frac{1}{3} = \frac{2}{3} \ . 
\end{equation}

\vspace{0.2cm}
\noindent\textsc{Type $\rho_b$:} for $i, j\not= 1$, consider $\rho_{b|i,j}$. We choose the same representative for $\psi_1^\trop$ as in \eqref{eq:psionetrop}. 
There is only one two-dimensional cone containing $\rho_{b|i,j}$ such that the function representing $\psi_1^\trop$ is not identically zero: it is of type  $\sigma_5$ and its other ray is $\rho_{a|k,l}$. From \eqref{eq:coeffofray} we have:
\begin{equation}\label{eq:coeffb}
    c_{\rho_{b|i,j}} =  1\cdot 1 = 1 \ . 
\end{equation}

\vspace{0.2cm}
\noindent\textsc{Types $\rho_c, \rho_d$:} let $\{i,j,k,l\} = \{2,3,4,5\}$; rays of these two types live in the inverse image via $\mathtt{br}$ of rays of the form $\hat{\rho}_{1,k}$. The piecewise linear function \eqref{eq:goodrepa} representing $\hat{\psi}^\trop_1$ is identically zero on the star of $\hat{\rho}_{1,k}$, and therefore its pullback via $\mathtt{br}$ is identically zero on the stars of the rays $\rho_{c|1,k}, \rho_{d|1,k}$. It follows immediately that:
\begin{equation} \label{eq:coeffcd}
    c_{\rho_{c|1,k}} = c_{\rho_{d|1,k}} =0 \ .
\end{equation}

The statement in the Proposition is  obtained by adding up \eqref{eq:coeffa}, \eqref{eq:coeffb} and \eqref{eq:coeffcd}.
\end{proof}

\subsection{Families on $\mathcal{M}_{1,2}^\trop$}

We  push-forward the $\psi$ class and the fundamentalish cycle via two different forgetful morphisms to obtain cycles on $\mathcal{M}_{1,2}^\trop$. We verify that the resulting tropical cycles satisfy relations analogous to those in  Lemma \ref{lem:pushpsi}.
\subsubsection{Remember a $1$-end}

We consider the tropical forgetful morphisms analogous to the map  $\varphi_1$ from Section \ref{sec:algsetup}: we denote $\varphi_1^\trop$ the map that assigns to a tropical cover the stabilization of the source curve after forgetting all the ends with expansion factor two and all but one of the marked ends with expansion factor one. The end with expansion factor  $3$ is labeled to be the first end, and the remaining end with expansion factor one is labeled to be the second marked end.

\begin{figure}
    \centering

\tikzset{every picture/.style={line width=0.75pt}} 

\begin{tikzpicture}[x=0.75pt,y=0.75pt,yscale=-1,xscale=1]

\draw    (135.3,151.46) -- (90.33,151.19) ;
\draw [color={rgb, 255:red, 0; green, 0; blue, 0 }  ,draw opacity=1 ]   (72.7,172.86) -- (89.72,151.19) ;
\draw    (75.13,130.07) -- (89.72,151.19) ;
\draw    (178.45,151.74) -- (133.48,151.46) ;
\draw    (178.45,151.74) -- (197.9,133.63) ;
\draw    (178.45,151.74) -- (197.29,171.49) ;
\draw    (144.72,112.23) -- (144.75,134.38) ;
\draw [shift={(144.75,136.38)}, rotate = 269.92] [color={rgb, 255:red, 0; green, 0; blue, 0 }  ][line width=0.75]    (10.93,-3.29) .. controls (6.95,-1.4) and (3.31,-0.3) .. (0,0) .. controls (3.31,0.3) and (6.95,1.4) .. (10.93,3.29)   ;
\draw    (139.35,151.46) -- (139.55,176.43) ;
\draw [color={rgb, 255:red, 208; green, 2; blue, 27 }  ,draw opacity=1 ][line width=1.5]    (74.53,113.47) -- (91.54,91.8) ;
\draw    (76.96,70.67) -- (91.54,91.8) ;
\draw    (78.78,66.56) -- (91.54,91.8) ;
\draw [color={rgb, 255:red, 245; green, 166; blue, 35 }  ,draw opacity=1 ][line width=1.5]    (91.54,91.8) -- (136.52,72.18) ;
\draw    (91.54,91.8) .. controls (132.87,100.44) and (145.63,99.34) .. (182.7,93.58) ;
\draw    (136.52,88.64) -- (136.52,72.18) ;
\draw  [draw opacity=0][line width=1.5]  (133.53,73.72) .. controls (139.12,70.07) and (148.1,67.7) .. (158.23,67.7) .. controls (169.07,67.7) and (178.6,70.41) .. (184.06,74.51) -- (158.23,82.52) -- cycle ; \draw  [color={rgb, 255:red, 144; green, 19; blue, 254 }  ,draw opacity=1 ][line width=1.5]  (133.53,73.72) .. controls (139.12,70.07) and (148.1,67.7) .. (158.23,67.7) .. controls (169.07,67.7) and (178.6,70.41) .. (184.06,74.51) ;  
\draw  [draw opacity=0][line width=1.5]  (181.7,72.64) .. controls (176.14,75.47) and (168.34,77.28) .. (159.7,77.37) .. controls (149.52,77.48) and (140.48,75.18) .. (134.87,71.55) -- (159.62,62.56) -- cycle ; \draw  [color={rgb, 255:red, 144; green, 19; blue, 254 }  ,draw opacity=1 ][line width=1.5]  (181.7,72.64) .. controls (176.14,75.47) and (168.34,77.28) .. (159.7,77.37) .. controls (149.52,77.48) and (140.48,75.18) .. (134.87,71.55) ;  
\draw    (136.9,114.7) -- (136.9,108.39) -- (136.9,98.24) ;
\draw    (180.88,73.01) -- (197.9,50.78) ;
\draw    (180.88,73.01) -- (206.41,81.24) ;
\draw    (182.7,93.58) -- (195.47,82.88) ;
\draw [color={rgb, 255:red, 65; green, 117; blue, 5 }  ,draw opacity=1 ][line width=1.5]    (182.7,93.58) -- (203.98,107.57) ;
\draw [color={rgb, 255:red, 65; green, 117; blue, 5 }  ,draw opacity=1 ][line width=1.5]    (394.14,82.43) -- (430.14,118.43) ;
\draw [color={rgb, 255:red, 208; green, 2; blue, 27 }  ,draw opacity=1 ][line width=1.5]    (430.14,118.43) -- (400.14,156.43) ;
\draw [color={rgb, 255:red, 245; green, 166; blue, 35 }  ,draw opacity=1 ][fill={rgb, 255:red, 245; green, 166; blue, 35 }  ,fill opacity=1 ][line width=1.5]    (430.14,118.43) -- (493.14,117.43) ;
\draw  [color={rgb, 255:red, 144; green, 19; blue, 254 }  ,draw opacity=1 ][line width=1.5]  (493.14,117.43) .. controls (493.14,96.4) and (510.19,79.36) .. (531.21,79.36) .. controls (552.24,79.36) and (569.29,96.4) .. (569.29,117.43) .. controls (569.29,138.45) and (552.24,155.5) .. (531.21,155.5) .. controls (510.19,155.5) and (493.14,138.45) .. (493.14,117.43) -- cycle ;
\draw [line width=1.5]    (257.14,106.43) -- (325.14,105.47) ;
\draw [shift={(328.14,105.43)}, rotate = 179.19] [color={rgb, 255:red, 0; green, 0; blue, 0 }  ][line width=1.5]    (19.89,-5.99) .. controls (12.65,-2.54) and (6.02,-0.55) .. (0,0) .. controls (6.02,0.55) and (12.65,2.54) .. (19.89,5.99)   ;
\draw    (138.79,491.08) -- (93.7,490.79) ;
\draw [color={rgb, 255:red, 0; green, 0; blue, 0 }  ,draw opacity=1 ]   (76.03,513.91) -- (93.09,490.79) ;
\draw    (78.46,468.25) -- (93.09,490.79) ;
\draw    (182.06,491.37) -- (136.96,491.08) ;
\draw    (182.06,491.37) -- (201.56,472.06) ;
\draw    (182.06,491.37) -- (200.95,512.45) ;
\draw    (148.24,449.23) -- (148.27,472.99) ;
\draw [shift={(148.27,474.99)}, rotate = 269.92] [color={rgb, 255:red, 0; green, 0; blue, 0 }  ][line width=0.75]    (10.93,-3.29) .. controls (6.95,-1.4) and (3.31,-0.3) .. (0,0) .. controls (3.31,0.3) and (6.95,1.4) .. (10.93,3.29)   ;
\draw    (142.85,491.08) -- (143.06,517.71) ;
\draw [color={rgb, 255:red, 208; green, 2; blue, 27 }  ,draw opacity=1 ][line width=1.5]    (77.85,450.55) -- (94.92,427.43) ;
\draw    (80.29,404.89) -- (94.92,427.43) ;
\draw    (82.12,400.5) -- (94.92,427.43) ;
\draw [color={rgb, 255:red, 144; green, 19; blue, 254 }  ,draw opacity=1 ][line width=1.5]    (94.92,427.43) -- (140.01,406.5) ;
\draw [color={rgb, 255:red, 245; green, 166; blue, 35 }  ,draw opacity=1 ][line width=1.5]    (94.92,427.43) .. controls (133.92,439.87) and (141.23,440.74) .. (181.45,422.31) ;
\draw [color={rgb, 255:red, 245; green, 166; blue, 35 }  ,draw opacity=1 ][line width=1.5]    (140.01,406.5) -- (181.45,422.31) ;
\draw    (140.01,406.5) -- (179.62,400.36) ;
\draw [color={rgb, 255:red, 65; green, 117; blue, 5 }  ,draw opacity=1 ][line width=1.5]    (179.62,400.36) -- (196.68,376.65) ;
\draw    (179.62,400.36) -- (205.21,409.14) ;
\draw    (181.45,422.31) -- (194.24,410.89) ;
\draw    (181.45,422.31) -- (202.77,437.23) ;
\draw    (140.01,424.06) -- (140.01,406.5) ;
\draw [color={rgb, 255:red, 0; green, 0; blue, 0 }  ,draw opacity=1 ][line width=0.75]    (139.79,453.62) -- (139.79,446.89) -- (139.79,436.06) ;
\draw    (140.79,327.08) -- (95.7,326.79) ;
\draw [color={rgb, 255:red, 0; green, 0; blue, 0 }  ,draw opacity=1 ]   (78.03,349.91) -- (95.09,326.79) ;
\draw [color={rgb, 255:red, 0; green, 0; blue, 0 }  ,draw opacity=1 ]   (80.46,304.25) -- (95.09,326.79) ;
\draw    (184.06,327.37) -- (138.96,327.08) ;
\draw    (184.06,327.37) -- (203.56,308.06) ;
\draw    (184.06,327.37) -- (202.95,348.45) ;
\draw    (150.24,285.23) -- (150.27,308.99) ;
\draw [shift={(150.27,310.99)}, rotate = 269.92] [color={rgb, 255:red, 0; green, 0; blue, 0 }  ][line width=0.75]    (10.93,-3.29) .. controls (6.95,-1.4) and (3.31,-0.3) .. (0,0) .. controls (3.31,0.3) and (6.95,1.4) .. (10.93,3.29)   ;
\draw    (144.85,327.08) -- (145.06,353.71) ;
\draw [color={rgb, 255:red, 208; green, 2; blue, 27 }  ,draw opacity=1 ][line width=1.5]    (79.85,286.55) -- (96.92,263.43) ;
\draw    (82.29,240.89) -- (96.92,263.43) ;
\draw    (84.12,236.5) -- (96.92,263.43) ;
\draw [color={rgb, 255:red, 144; green, 19; blue, 254 }  ,draw opacity=1 ][line width=1.5]    (96.92,263.43) -- (142.01,242.5) ;
\draw [color={rgb, 255:red, 144; green, 19; blue, 254 }  ,draw opacity=1 ][line width=1.5]    (141.79,272.06) .. controls (159,274.89) and (167,266.89) .. (183.45,258.31) ;
\draw [color={rgb, 255:red, 144; green, 19; blue, 254 }  ,draw opacity=1 ][line width=1.5]    (142.01,242.5) -- (183.45,258.31) ;
\draw    (142.01,242.5) -- (181.62,236.36) ;
\draw    (181.62,236.36) -- (198.68,212.65) ;
\draw    (181.62,236.36) -- (207.21,245.14) ;
\draw    (183.45,258.31) -- (196.24,246.89) ;
\draw    (183.45,258.31) -- (204.77,273.23) ;
\draw    (142.01,260.06) -- (142.01,242.5) ;
\draw [color={rgb, 255:red, 65; green, 117; blue, 5 }  ,draw opacity=1 ][line width=1.5]    (141.79,289.62) -- (141.79,282.89) -- (141.79,272.06) ;
\draw [line width=1.5]    (260.14,436.43) -- (328.14,435.47) ;
\draw [shift={(331.14,435.43)}, rotate = 179.19] [color={rgb, 255:red, 0; green, 0; blue, 0 }  ][line width=1.5]    (19.89,-5.99) .. controls (12.65,-2.54) and (6.02,-0.55) .. (0,0) .. controls (6.02,0.55) and (12.65,2.54) .. (19.89,5.99)   ;
\draw [line width=1.5]    (260.14,273.57) -- (328.14,272.61) ;
\draw [shift={(331.14,272.57)}, rotate = 179.19] [color={rgb, 255:red, 0; green, 0; blue, 0 }  ][line width=1.5]    (19.89,-5.99) .. controls (12.65,-2.54) and (6.02,-0.55) .. (0,0) .. controls (6.02,0.55) and (12.65,2.54) .. (19.89,5.99)   ;
\draw [color={rgb, 255:red, 65; green, 117; blue, 5 }  ,draw opacity=1 ][line width=1.5]    (402.45,434.23) -- (446.87,434.58) ;
\draw [color={rgb, 255:red, 208; green, 2; blue, 27 }  ,draw opacity=1 ][line width=1.5]    (572.57,434.29) -- (519.57,434.64) ;
\draw [color={rgb, 255:red, 65; green, 117; blue, 5 }  ,draw opacity=1 ][line width=1.5]    (401.45,273.23) -- (423.56,273.4) -- (445.87,273.58) ;
\draw [color={rgb, 255:red, 208; green, 2; blue, 27 }  ,draw opacity=1 ][line width=1.5]    (571.57,273.29) -- (518.57,273.64) ;
\draw  [draw opacity=0][line width=1.5]  (446.87,434.58) .. controls (450.61,420.9) and (465.41,410.66) .. (483.13,410.62) .. controls (500.95,410.57) and (515.87,420.86) .. (519.57,434.64) -- (483.21,441.34) -- cycle ; \draw  [color={rgb, 255:red, 144; green, 19; blue, 254 }  ,draw opacity=1 ][line width=1.5]  (446.87,434.58) .. controls (450.61,420.9) and (465.41,410.66) .. (483.13,410.62) .. controls (500.95,410.57) and (515.87,420.86) .. (519.57,434.64) ;  
\draw  [draw opacity=0][line width=1.5]  (445.87,273.58) .. controls (449.61,259.9) and (464.41,249.66) .. (482.13,249.62) .. controls (499.95,249.57) and (514.87,259.86) .. (518.57,273.64) -- (482.21,280.34) -- cycle ; \draw  [color={rgb, 255:red, 144; green, 19; blue, 254 }  ,draw opacity=1 ][line width=1.5]  (445.87,273.58) .. controls (449.61,259.9) and (464.41,249.66) .. (482.13,249.62) .. controls (499.95,249.57) and (514.87,259.86) .. (518.57,273.64) ;  
\draw  [draw opacity=0][line width=1.5]  (519.57,435.32) .. controls (515.7,448.96) and (500.81,459.06) .. (483.1,458.94) .. controls (465.28,458.83) and (450.45,448.4) .. (446.87,434.58) -- (483.3,428.22) -- cycle ; \draw  [color={rgb, 255:red, 245; green, 166; blue, 35 }  ,draw opacity=1 ][line width=1.5]  (519.57,435.32) .. controls (515.7,448.96) and (500.81,459.06) .. (483.1,458.94) .. controls (465.28,458.83) and (450.45,448.4) .. (446.87,434.58) ;  
\draw  [draw opacity=0][line width=1.5]  (518.57,274.32) .. controls (514.7,287.96) and (499.81,298.06) .. (482.1,297.94) .. controls (464.28,297.83) and (449.45,287.4) .. (445.87,273.58) -- (482.3,267.22) -- cycle ; \draw  [color={rgb, 255:red, 245; green, 166; blue, 35 }  ,draw opacity=1 ][line width=1.5]  (518.57,274.32) .. controls (514.7,287.96) and (499.81,298.06) .. (482.1,297.94) .. controls (464.28,297.83) and (449.45,287.4) .. (445.87,273.58) ;  
\draw [color={rgb, 255:red, 245; green, 166; blue, 35 }  ,draw opacity=1 ][line width=1.5]    (96.92,263.43) .. controls (108.6,270.89) and (118.2,274.09) .. (141.79,272.06) ;

\draw (67.41,67.52) node [anchor=north west][inner sep=0.75pt]  [font=\tiny,color={rgb, 255:red, 2; green, 2; blue, 208 }  ,opacity=1 ]  {$1$};
\draw (68.63,57.65) node [anchor=north west][inner sep=0.75pt]  [font=\tiny,color={rgb, 255:red, 2; green, 2; blue, 208 }  ,opacity=1 ]  {$2$};
\draw (63.97,111.42) node [anchor=north west][inner sep=0.75pt]  [font=\tiny,color={rgb, 255:red, 2; green, 2; blue, 208 }  ,opacity=1 ]  {$3$};
\draw (128.49,86.25) node [anchor=north west][inner sep=0.75pt]  [font=\tiny,color={rgb, 255:red, 2; green, 2; blue, 208 }  ,opacity=1 ]  {$2$};
\draw (109.04,73.9) node [anchor=north west][inner sep=0.75pt]  [font=\tiny,color={rgb, 255:red, 2; green, 2; blue, 208 }  ,opacity=1 ]  {$2$};
\draw (109.95,87.28) node [anchor=north west][inner sep=0.75pt]  [font=\tiny,color={rgb, 255:red, 2; green, 2; blue, 208 }  ,opacity=1 ]  {$1$};
\draw (129.4,110.32) node [anchor=north west][inner sep=0.75pt]  [font=\tiny,color={rgb, 255:red, 2; green, 2; blue, 208 }  ,opacity=1 ]  {$1$};
\draw (151.89,56.83) node [anchor=north west][inner sep=0.75pt]  [font=\tiny,color={rgb, 255:red, 2; green, 2; blue, 208 }  ,opacity=1 ]  {$1$};
\draw (151.89,69.17) node [anchor=north west][inner sep=0.75pt]  [font=\tiny,color={rgb, 255:red, 2; green, 2; blue, 208 }  ,opacity=1 ]  {$1$};
\draw (153.1,88.92) node [anchor=north west][inner sep=0.75pt]  [font=\tiny,color={rgb, 255:red, 2; green, 2; blue, 208 }  ,opacity=1 ]  {$1$};
\draw (196.25,81.52) node [anchor=north west][inner sep=0.75pt]  [font=\tiny,color={rgb, 255:red, 2; green, 2; blue, 208 }  ,opacity=1 ]  {$1$};
\draw (204.76,106.21) node [anchor=north west][inner sep=0.75pt]  [font=\tiny,color={rgb, 255:red, 2; green, 2; blue, 208 }  ,opacity=1 ]  {$1$};
\draw (198.38,45.1) node [anchor=north west][inner sep=0.75pt]  [font=\tiny,color={rgb, 255:red, 2; green, 2; blue, 208 }  ,opacity=1 ]  {$2$};
\draw (206.89,77.2) node [anchor=north west][inner sep=0.75pt]  [font=\tiny,color={rgb, 255:red, 2; green, 2; blue, 208 }  ,opacity=1 ]  {$2$};
\draw (550,106.73) node [anchor=north west][inner sep=0.75pt]    {$x$};
\draw (456,96.4) node [anchor=north west][inner sep=0.75pt]    {$y$};
\draw (106,57.4) node [anchor=north west][inner sep=0.75pt]    {$u$};
\draw (156,102.4) node [anchor=north west][inner sep=0.75pt]    {$v$};
\draw (264,72.4) node [anchor=north west][inner sep=0.75pt]    {$\varphi_1^\trop$};
\draw (269,138.4) node [anchor=north west][inner sep=0.75pt]  [font=\footnotesize]  {$y\ =\ u$};
\draw (268,123.4) node [anchor=north west][inner sep=0.75pt]  [font=\footnotesize]  {$x\ =\ 2v$};
\draw (112.38,425.7) node [anchor=north west][inner sep=0.75pt]  [font=\tiny,color={rgb, 255:red, 2; green, 2; blue, 208 }  ,opacity=1 ]  {$1$};
\draw (113.09,406.75) node [anchor=north west][inner sep=0.75pt]  [font=\tiny,color={rgb, 255:red, 2; green, 2; blue, 208 }  ,opacity=1 ]  {$2$};
\draw (70.74,401.71) node [anchor=north west][inner sep=0.75pt]  [font=\tiny,color={rgb, 255:red, 2; green, 2; blue, 208 }  ,opacity=1 ]  {$1$};
\draw (71.96,391.17) node [anchor=north west][inner sep=0.75pt]  [font=\tiny,color={rgb, 255:red, 2; green, 2; blue, 208 }  ,opacity=1 ]  {$2$};
\draw (67.28,448.53) node [anchor=north west][inner sep=0.75pt]  [font=\tiny,color={rgb, 255:red, 2; green, 2; blue, 208 }  ,opacity=1 ]  {$3$};
\draw (202.87,435.73) node [anchor=north west][inner sep=0.75pt]  [font=\tiny,color={rgb, 255:red, 2; green, 2; blue, 208 }  ,opacity=1 ]  {$2$};
\draw (193.12,413.56) node [anchor=north west][inner sep=0.75pt]  [font=\tiny,color={rgb, 255:red, 2; green, 2; blue, 208 }  ,opacity=1 ]  {$2$};
\draw (197.69,372.15) node [anchor=north west][inner sep=0.75pt]  [font=\tiny,color={rgb, 255:red, 2; green, 2; blue, 208 }  ,opacity=1 ]  {$1$};
\draw (206.22,408.14) node [anchor=north west][inner sep=0.75pt]  [font=\tiny,color={rgb, 255:red, 2; green, 2; blue, 208 }  ,opacity=1 ]  {$1$};
\draw (158.69,390.59) node [anchor=north west][inner sep=0.75pt]  [font=\tiny,color={rgb, 255:red, 2; green, 2; blue, 208 }  ,opacity=1 ]  {$1$};
\draw (158.69,406.1) node [anchor=north west][inner sep=0.75pt]  [font=\tiny,color={rgb, 255:red, 2; green, 2; blue, 208 }  ,opacity=1 ]  {$1$};
\draw (158.69,421.31) node [anchor=north west][inner sep=0.75pt]  [font=\tiny,color={rgb, 255:red, 2; green, 2; blue, 208 }  ,opacity=1 ]  {$1$};
\draw (131.98,421.68) node [anchor=north west][inner sep=0.75pt]  [font=\tiny,color={rgb, 255:red, 2; green, 2; blue, 208 }  ,opacity=1 ]  {$2$};
\draw (131.27,452.78) node [anchor=north west][inner sep=0.75pt]  [font=\tiny,color={rgb, 255:red, 2; green, 2; blue, 208 }  ,opacity=1 ]  {$1$};
\draw (114.38,261.7) node [anchor=north west][inner sep=0.75pt]  [font=\tiny,color={rgb, 255:red, 2; green, 2; blue, 208 }  ,opacity=1 ]  {$1$};
\draw (115.09,242.75) node [anchor=north west][inner sep=0.75pt]  [font=\tiny,color={rgb, 255:red, 2; green, 2; blue, 208 }  ,opacity=1 ]  {$2$};
\draw (72.74,237.71) node [anchor=north west][inner sep=0.75pt]  [font=\tiny,color={rgb, 255:red, 2; green, 2; blue, 208 }  ,opacity=1 ]  {$1$};
\draw (73.96,227.17) node [anchor=north west][inner sep=0.75pt]  [font=\tiny,color={rgb, 255:red, 2; green, 2; blue, 208 }  ,opacity=1 ]  {$2$};
\draw (69.28,284.53) node [anchor=north west][inner sep=0.75pt]  [font=\tiny,color={rgb, 255:red, 2; green, 2; blue, 208 }  ,opacity=1 ]  {$3$};
\draw (204.87,271.73) node [anchor=north west][inner sep=0.75pt]  [font=\tiny,color={rgb, 255:red, 2; green, 2; blue, 208 }  ,opacity=1 ]  {$2$};
\draw (195.12,249.56) node [anchor=north west][inner sep=0.75pt]  [font=\tiny,color={rgb, 255:red, 2; green, 2; blue, 208 }  ,opacity=1 ]  {$2$};
\draw (199.69,208.15) node [anchor=north west][inner sep=0.75pt]  [font=\tiny,color={rgb, 255:red, 2; green, 2; blue, 208 }  ,opacity=1 ]  {$1$};
\draw (208.22,244.14) node [anchor=north west][inner sep=0.75pt]  [font=\tiny,color={rgb, 255:red, 2; green, 2; blue, 208 }  ,opacity=1 ]  {$1$};
\draw (160.69,226.59) node [anchor=north west][inner sep=0.75pt]  [font=\tiny,color={rgb, 255:red, 2; green, 2; blue, 208 }  ,opacity=1 ]  {$1$};
\draw (160.69,242.1) node [anchor=north west][inner sep=0.75pt]  [font=\tiny,color={rgb, 255:red, 2; green, 2; blue, 208 }  ,opacity=1 ]  {$1$};
\draw (160.69,257.31) node [anchor=north west][inner sep=0.75pt]  [font=\tiny,color={rgb, 255:red, 2; green, 2; blue, 208 }  ,opacity=1 ]  {$1$};
\draw (133.98,257.68) node [anchor=north west][inner sep=0.75pt]  [font=\tiny,color={rgb, 255:red, 2; green, 2; blue, 208 }  ,opacity=1 ]  {$2$};
\draw (133.27,288.78) node [anchor=north west][inner sep=0.75pt]  [font=\tiny,color={rgb, 255:red, 2; green, 2; blue, 208 }  ,opacity=1 ]  {$1$};
\draw (267,402.4) node [anchor=north west][inner sep=0.75pt]    {$\varphi_1^\trop$};
\draw (272,468.4) node [anchor=north west][inner sep=0.75pt]  [font=\footnotesize]  {$x_{2} \ =\ 2u+2v$};
\draw (271,453.4) node [anchor=north west][inner sep=0.75pt]  [font=\footnotesize]  {$x_{1} \ =\ u$};
\draw (267,239.54) node [anchor=north west][inner sep=0.75pt]    {$\varphi_1^\trop$};
\draw (272,305.54) node [anchor=north west][inner sep=0.75pt]  [font=\footnotesize]  {$x_{2} \ =\ 2u$};
\draw (271,290.54) node [anchor=north west][inner sep=0.75pt]  [font=\footnotesize]  {$x_{1} \ =\ u+2v$};
\draw (474,383.4) node [anchor=north west][inner sep=0.75pt]    {$x_{1}$};
\draw (473.33,435.73) node [anchor=north west][inner sep=0.75pt]    {$x_{2}$};
\draw (473,222.4) node [anchor=north west][inner sep=0.75pt]    {$x_{1}$};
\draw (473.67,273.4) node [anchor=north west][inner sep=0.75pt]    {$x_{2}$};
\draw (110,389.4) node [anchor=north west][inner sep=0.75pt]    {$u$};
\draw (113,223.4) node [anchor=north west][inner sep=0.75pt]    {$u$};
\draw (157,431.4) node [anchor=north west][inner sep=0.75pt]    {$v$};
\draw (162,271.4) node [anchor=north west][inner sep=0.75pt]    {$v$};

\end{tikzpicture}

\caption{The type of admissible covers contributing to the computation of the degree of the map $\varphi_1^{\trop}$. We have colored the ends that are remembered, as well as the edges whose lengths are relevant for the computation of the degree.}
\label{fig:deg24}
\end{figure}
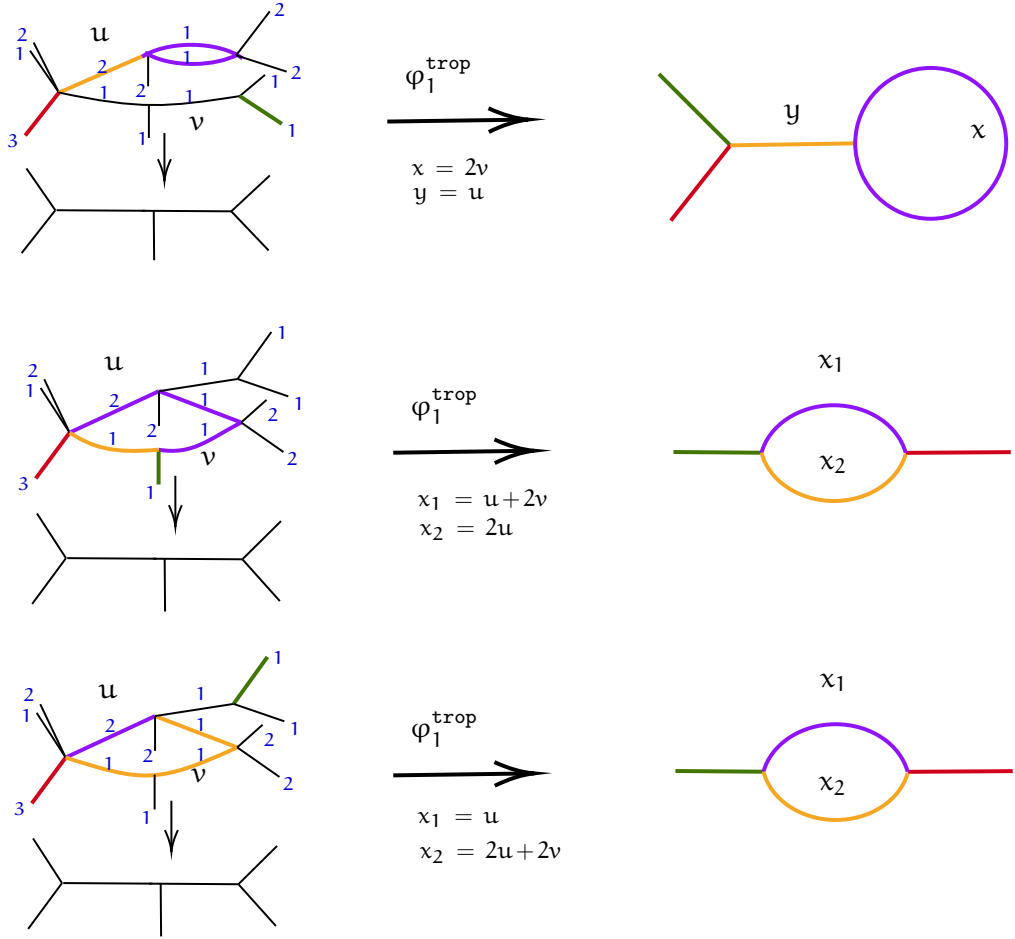

\begin{lemma}
\label{lem:degphione}
The function  $\varphi_1^\trop$ has degree $24$.
\end{lemma}

\begin{proof}
We recall that $\mathcal{M}_{1,2}^\trop$ is a cone stack (see \cite{CCUW}). Each of the two top-dimensional cones is the target of a degree two quotient map from $\R^2_{\geq 0}$. For the cone parameterizing curves where the two marked ends are on different vertices, the $\mu_2$ action is reflection along the diagonal. For the cone where both marked ends are adjacent to the same vertex, the $\mu_2$ action is trivial.

In order to show that the degree of $\varphi_1^\trop$ is  equal to $24$, we show that for each general interior point of $\mathcal{M}_{1,2}^\trop$, the sum of the local degrees of the map at its inverse images equals $24$.

Recall that for a cover $\xi = [f:\Gamma\to T]$ belonging to a two-dimensional cone $\sigma_\xi$, and such that  $\varphi_1^\trop(\xi) =\tilde\Gamma$ is in a two-dimensional cone $\tilde\sigma_{\varphi_1^\trop(\xi)}$, the local degree of $\varphi_1^\trop$ is defined to be:
\begin{equation}
    \deg_\xi(\varphi_1^\trop) = \frac{\omega(\sigma_\xi)} {\omega(\tilde\sigma_{\varphi_1^\trop(\xi)})} |\det(M)| \ , 
\end{equation}
where $\omega(\sigma_\xi)$ is the weight of the cone $\sigma_\xi$ in the fundamentalish cycle of $\Adm$, $\omega(\tilde\sigma_{\varphi_1^\trop(\xi)}) = 1/|\Aut(\tilde\Gamma)|$ and $M$ is the matrix giving a local expression for $\varphi_1^\trop$ in integral lattice bases for the two cones. 
There are three distinct situations that need to be considered. Refer to Figure \ref{fig:deg24}.

\noindent{\textsc{Case I:}} $\tilde\xi$ belongs to the interior of the two-dimensional cone of $\mathcal{M}_{1,2}^\trop$ parameterizing tropical curves where the two marks emanate from the same vertex.
Then there are twelve inverse images for $\tilde\xi$, one belonging to each cone of type $\sigma_3$ (see the first line of Figure \ref{fig:deg24} for one example). For any such inverse image $\xi$, the local degree of $\varphi_1^\trop$ is
\begin{equation}
    \deg_\xi(\varphi_1^\trop) = \frac{2}{2}\left|\det
    \left[
    \begin{array}{cc}
       0  & 2 \\
        1 & 0
    \end{array}
    \right]
    \right|= 2 \ .
\end {equation}
Since there are twelve inverse images, the sum of the local degrees equals $24$, as desired.

\noindent{\textsc{Case II:}} $\tilde\xi$ belongs to the interior of the two-dimensional  folded cone $\tilde\sigma$ of $\mathcal{M}_{1,2}^\trop$. Consider the orthant with coordinates $(x_1, x_2)$ mapping with degree  two to $\tilde\sigma$ and assume $\tilde\xi$ lies in the region whose inverse image is  $\{x_2<x_1/2\}\cup \{x_2>2x_1\}$ \ .

Then $\tilde\xi$ has a total of $9$ inverse images. Three inverse images belong to cones of type $\sigma_2$ where the $1$-end that is remembered maps to the middle edge of the target graph (see second line of Figure \ref{fig:deg24}).
For each of these inverse images the local degree is:
\begin{equation}
    \deg_\xi(\varphi_1^\trop) = \frac{1}{1}\left|\det
    \left[
    \begin{array}{cc}
       1  & 2 \\
        2 & 0
    \end{array}
    \right]
    \right|= 4 \ .
\end {equation}
The remaining six inverse images belong to cones of type $\sigma_2$ with the $1$-end mapped to an external edge of the target tree (third line of Figure \ref{fig:deg24}). The local degree for these inverse images is :
\begin{equation}
    \label{eq:locdeg}
    \deg_\xi(\varphi_1^\trop) = \frac{1}{1}\left|\det
    \left[
    \begin{array}{cc}
       1  & 0 \\
        2 & 2
    \end{array}
    \right]
    \right|= 2 \ .
\end{equation}
It follows that the sum of the local degrees over all inverse images equals $24$.

\noindent{\textsc{Case III:}} $\tilde\xi$ belongs to the interior of the two-dimensional  folded cone $\tilde\sigma$ of $\mathcal{M}_{1,2}^\trop$. Now assume $\tilde\xi$ lies in the region whose inverse image is  $\{x_1/2<x_2<2x_1\}$. There are $12$ inverse images, two for each of the six cones of type $\sigma_2$ where the $1$-edge maps to an external end of the target tree (third line of Figure \ref{fig:deg24}). The local degree computation is identical to \eqref{eq:locdeg}, which again yields a total of $24$.

We have thus verified that the degree of $\varphi_1^\trop$ is  equal to $24$.
\end{proof}

\begin{lemma}\label{lem:pfone}
Denote by $\rho_{irr}$  the ray of ${\calM}^\trop_{1,2}$ parameterizing curves with one vertex of genus zero and a self loop. 
We have 
\begin{equation}
    \varphi^\trop_{1, \ast}(\psi_1^\trop \frown \mtt T_\ast([\Sigma])) = 12\ \rho_{irr} \ .
\end{equation}
\end{lemma}

\begin{proof}
From the computation in Proposition \ref{prop:psicomp}, we only need to compute the pushforward of the rays of type $\rho_a$ and $\rho_b$. It is immediate to see that all rays of type $\rho_a$ are contracted to the cone point of $\overline{\calM}^\trop_{1,2}$, and therefore their pushforward vanishes.
All the rays of type $\rho_b$ map to the ray we have denoted $\rho_{irr}$, and the function $\varphi_1^\trop$ restricts to each of these rays as a linear function of slope $2$. Both the tropical covers $\Gamma \to T$ parameterized by the rays of type $\rho_b$ and the tropical curves $\tilde\Gamma$ parameterized by $\rho_{irr}$ have  a group of automorphisms of order two. Recalling that we have $6$ rays of type $\rho_b$ we can conclude
\begin{equation}
    \varphi_{1,\ast}(\rho_b^\tot) 
        = \left(\frac{|\Aut(\tilde{\Gamma})|}{|\Aut(\Gamma\to T)|}\cdot \mathrm{slope}_{\rho_b}({\varphi_1}) \cdot |\text{rays of type $\rho_b$}|\right) \rho_{irr}     = \left(\frac{2}{2}\cdot 2\cdot 6\right)\rho_{irr} = 12 \rho_{irr} \ .
\end{equation}
This computation concludes the proof of the lemma.
\end{proof}

We  observe that this result is compatible with the algebraic computation from Lemma \ref{lem:pushpsi}. Denote by $\rho_{sec}$ the other ray of $\calM_{1,2}^\trop$ (corresponding to graphs with a genus-one vertex and a genus zero vertex with the two legs attached), and for $x = sec, irr$, denote by $D_{x}$ the divisor in $\Mbar_{1,2}$ parameterizing curves whose dual graph correspond to the ray $\rho_x$.
It  is a standard fact from the intersection theory of $\psi$ classes that $\psi_1 \cdot D_{sec} = 0$ and $\psi_1\cdot D_{irr} = \frac{1}{2}.$
By \eqref{eq:firstpushpullrel}, we have 
$\varphi_{1,\ast}(\psi_1) = 24 \psi_1$, hence its operational tropicalization agrees with the computation from Lemma \ref{lem:pfone}.

\subsubsection{Remember a $2$-end}
In this section we compute the pushforward of $\psi_1^\trop$ to $\mathcal{M}_{1,2}^\trop$ via the forgetful morphism $\varphi_2^\trop$ (remembering the degree $3$ end and one of the degree $2$ ends of the tropical covers), and check that the result is consistent with the operational tropicalization of  \eqref{eq:secondpushpullrel}. We begin by computing the operational tropicalization of the Weierstrass divisor $W$. Since $W$ is not dimensionally transverse to the boundary of $\Mgnbar{1,2}$, before we apply tropicalization we resolve this issue by blowing up the stratum $\delta_{00}$ of self-intersection of the irreducible divisor in $\overline{\calM}_{1,2}$.

\begin{lemma}
\label{lem:tropW}
Denote by $\rho_{irr}$ and $\rho_{sec}$ the rays of $\mathcal{M}_{1,2}^\trop$ dual to the irreducible divisor and the section of $\overline{\calM}_{1,2}$, and by $\rho_E$ the ray coinciding with the diagonal of the folded cone (corresponding to the exceptional divisor $E$ in the blow up of the point $\delta_{00}$). We have
\begin{equation}
    \Trop(W) = \frac{1}{2}(\rho_{irr}+\rho_E) \ .
\end{equation}
\end{lemma}
\begin{proof} 
To compute the operational tropicalization of $W$ we must compute the intersection of $W$ with all boundary curves of $Bl_{\delta_{00}}\overline{\calM}_{1,2}$. We identify $W$ with the (proper transform of the) locus of genus-one admissible covers of $\mathds{P}^1$ of degree $2$ with two labeled and two unlabeled branch points. This family has exactly two points parameterizing singular curves. One point $p_1$ corresponds to covers where one of the unlabeled branch points has collided with one of the labeled  branch points. The source curve of such a cover is a banana curve. The cover has an  automorphisms group of order two. All together $p_1$ contributes $\frac{1}{2}$ to the intersection of $W$ and $E$.
The second point $p_2$ arises when the two unlabeled branch points have collided. In this case the source curve of the admissible cover is unstable. Its stabilization is a nodal rational curve with two marked points, i.e.\ it is a general point in the interior of $\delta_{irr}$. The automorphism group of the cover has order two, and therefore $p_2$ contributes $\frac{1}{2}$ to the intersection of $W$ with 
$\delta_{irr}$. Since $W$ does not contain any other point parameterizing singular curves, we deduce that 
\begin{equation}
    W\cdot\delta_{sec} = 0, \quad
    W\cdot\delta_{irr} = \frac{1}{2}, \quad
     W\cdot\delta_{E} = \frac{1}{2}  \ ,
\end{equation}
which proves the assertion of the lemma.
\end{proof}

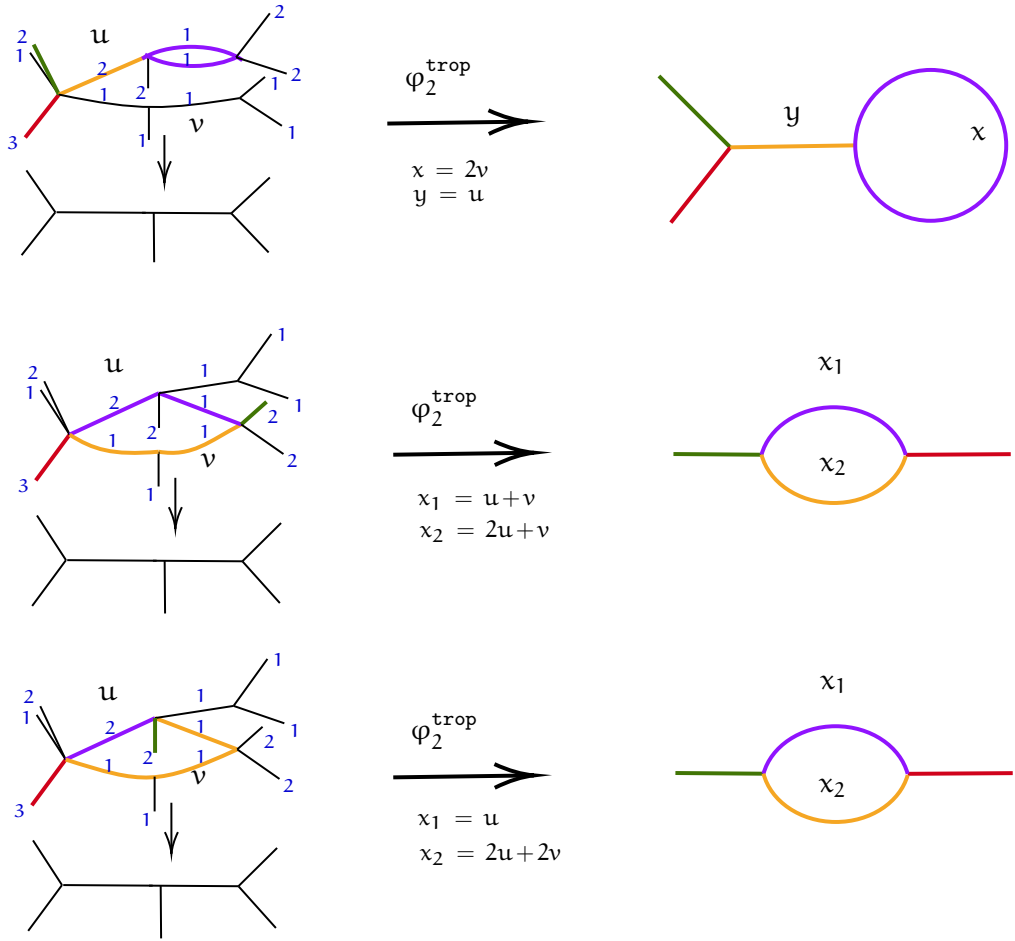
\begin{figure}
    \centering
  
\tikzset{every picture/.style={line width=0.75pt}} 

\begin{tikzpicture}[x=0.75pt,y=0.75pt,yscale=-1,xscale=1]

\draw    (155.3,171.46) -- (110.33,171.19) ;
\draw [color={rgb, 255:red, 0; green, 0; blue, 0 }  ,draw opacity=1 ]   (92.7,192.86) -- (109.72,171.19) ;
\draw    (95.13,150.07) -- (109.72,171.19) ;
\draw    (198.45,171.74) -- (153.48,171.46) ;
\draw    (198.45,171.74) -- (217.9,153.63) ;
\draw    (198.45,171.74) -- (217.29,191.49) ;
\draw    (164.72,132.23) -- (164.75,154.38) ;
\draw [shift={(164.75,156.38)}, rotate = 269.92] [color={rgb, 255:red, 0; green, 0; blue, 0 }  ][line width=0.75]    (10.93,-3.29) .. controls (6.95,-1.4) and (3.31,-0.3) .. (0,0) .. controls (3.31,0.3) and (6.95,1.4) .. (10.93,3.29)   ;
\draw    (159.35,171.46) -- (159.55,196.43) ;
\draw [color={rgb, 255:red, 208; green, 2; blue, 27 }  ,draw opacity=1 ][line width=1.5]    (94.53,133.47) -- (111.54,111.8) ;
\draw    (96.96,90.67) -- (111.54,111.8) ;
\draw [color={rgb, 255:red, 65; green, 117; blue, 5 }  ,draw opacity=1 ][line width=1.5]    (98.78,86.56) -- (111.54,111.8) ;
\draw [color={rgb, 255:red, 245; green, 166; blue, 35 }  ,draw opacity=1 ][line width=1.5]    (111.54,111.8) -- (156.52,92.18) ;
\draw    (111.54,111.8) .. controls (152.87,120.44) and (165.63,119.34) .. (202.7,113.58) ;
\draw    (156.52,108.64) -- (156.52,92.18) ;
\draw  [draw opacity=0][line width=1.5]  (153.53,93.72) .. controls (159.12,90.07) and (168.1,87.7) .. (178.23,87.7) .. controls (189.07,87.7) and (198.6,90.41) .. (204.06,94.51) -- (178.23,102.52) -- cycle ; \draw  [color={rgb, 255:red, 144; green, 19; blue, 254 }  ,draw opacity=1 ][line width=1.5]  (153.53,93.72) .. controls (159.12,90.07) and (168.1,87.7) .. (178.23,87.7) .. controls (189.07,87.7) and (198.6,90.41) .. (204.06,94.51) ;  
\draw  [draw opacity=0][line width=1.5]  (201.7,92.64) .. controls (196.14,95.47) and (188.34,97.28) .. (179.7,97.37) .. controls (169.52,97.48) and (160.48,95.18) .. (154.87,91.55) -- (179.62,82.56) -- cycle ; \draw  [color={rgb, 255:red, 144; green, 19; blue, 254 }  ,draw opacity=1 ][line width=1.5]  (201.7,92.64) .. controls (196.14,95.47) and (188.34,97.28) .. (179.7,97.37) .. controls (169.52,97.48) and (160.48,95.18) .. (154.87,91.55) ;  
\draw    (156.9,134.7) -- (156.9,128.39) -- (156.9,118.24) ;
\draw    (200.88,93.01) -- (217.9,70.78) ;
\draw    (200.88,93.01) -- (226.41,101.24) ;
\draw    (202.7,113.58) -- (215.47,102.88) ;
\draw [color={rgb, 255:red, 0; green, 0; blue, 0 }  ,draw opacity=1 ][line width=0.75]    (202.7,113.58) -- (223.98,127.57) ;
\draw [color={rgb, 255:red, 65; green, 117; blue, 5 }  ,draw opacity=1 ][line width=1.5]    (414.14,102.43) -- (450.14,138.43) ;
\draw [color={rgb, 255:red, 208; green, 2; blue, 27 }  ,draw opacity=1 ][line width=1.5]    (450.14,138.43) -- (420.14,176.43) ;
\draw [color={rgb, 255:red, 245; green, 166; blue, 35 }  ,draw opacity=1 ][fill={rgb, 255:red, 245; green, 166; blue, 35 }  ,fill opacity=1 ][line width=1.5]    (450.14,138.43) -- (513.14,137.43) ;
\draw  [color={rgb, 255:red, 144; green, 19; blue, 254 }  ,draw opacity=1 ][line width=1.5]  (513.14,137.43) .. controls (513.14,116.4) and (530.19,99.36) .. (551.21,99.36) .. controls (572.24,99.36) and (589.29,116.4) .. (589.29,137.43) .. controls (589.29,158.45) and (572.24,175.5) .. (551.21,175.5) .. controls (530.19,175.5) and (513.14,158.45) .. (513.14,137.43) -- cycle ;
\draw [line width=1.5]    (277.14,126.43) -- (345.14,125.47) ;
\draw [shift={(348.14,125.43)}, rotate = 179.19] [color={rgb, 255:red, 0; green, 0; blue, 0 }  ][line width=1.5]    (19.89,-5.99) .. controls (12.65,-2.54) and (6.02,-0.55) .. (0,0) .. controls (6.02,0.55) and (12.65,2.54) .. (19.89,5.99)   ;
\draw    (158.79,511.08) -- (113.7,510.79) ;
\draw [color={rgb, 255:red, 0; green, 0; blue, 0 }  ,draw opacity=1 ]   (96.03,533.91) -- (113.09,510.79) ;
\draw    (98.46,488.25) -- (113.09,510.79) ;
\draw    (202.06,511.37) -- (156.96,511.08) ;
\draw    (202.06,511.37) -- (221.56,492.06) ;
\draw    (202.06,511.37) -- (220.95,532.45) ;
\draw    (168.24,469.23) -- (168.27,492.99) ;
\draw [shift={(168.27,494.99)}, rotate = 269.92] [color={rgb, 255:red, 0; green, 0; blue, 0 }  ][line width=0.75]    (10.93,-3.29) .. controls (6.95,-1.4) and (3.31,-0.3) .. (0,0) .. controls (3.31,0.3) and (6.95,1.4) .. (10.93,3.29)   ;
\draw    (162.85,511.08) -- (163.06,537.71) ;
\draw [color={rgb, 255:red, 208; green, 2; blue, 27 }  ,draw opacity=1 ][line width=1.5]    (97.85,470.55) -- (114.92,447.43) ;
\draw    (100.29,424.89) -- (114.92,447.43) ;
\draw    (102.12,420.5) -- (114.92,447.43) ;
\draw [color={rgb, 255:red, 144; green, 19; blue, 254 }  ,draw opacity=1 ][line width=1.5]    (114.92,447.43) -- (160.01,426.5) ;
\draw [color={rgb, 255:red, 245; green, 166; blue, 35 }  ,draw opacity=1 ][line width=1.5]    (114.92,447.43) .. controls (153.92,459.87) and (161.23,460.74) .. (201.45,442.31) ;
\draw [color={rgb, 255:red, 245; green, 166; blue, 35 }  ,draw opacity=1 ][line width=1.5]    (160.01,426.5) -- (201.45,442.31) ;
\draw    (160.01,426.5) -- (199.62,420.36) ;
\draw [color={rgb, 255:red, 0; green, 0; blue, 0 }  ,draw opacity=1 ][line width=0.75]    (199.62,420.36) -- (216.68,396.65) ;
\draw    (199.62,420.36) -- (225.21,429.14) ;
\draw    (201.45,442.31) -- (214.24,430.89) ;
\draw    (201.45,442.31) -- (222.77,457.23) ;
\draw [color={rgb, 255:red, 65; green, 117; blue, 5 }  ,draw opacity=1 ][line width=1.5]    (160.01,444.06) -- (160.01,426.5) ;
\draw [color={rgb, 255:red, 0; green, 0; blue, 0 }  ,draw opacity=1 ][line width=0.75]    (159.79,473.62) -- (159.79,466.89) -- (159.79,456.06) ;
\draw    (160.79,347.08) -- (115.7,346.79) ;
\draw [color={rgb, 255:red, 0; green, 0; blue, 0 }  ,draw opacity=1 ]   (98.03,369.91) -- (115.09,346.79) ;
\draw [color={rgb, 255:red, 0; green, 0; blue, 0 }  ,draw opacity=1 ]   (100.46,324.25) -- (115.09,346.79) ;
\draw    (204.06,347.37) -- (158.96,347.08) ;
\draw    (204.06,347.37) -- (223.56,328.06) ;
\draw    (204.06,347.37) -- (222.95,368.45) ;
\draw    (170.24,305.23) -- (170.27,328.99) ;
\draw [shift={(170.27,330.99)}, rotate = 269.92] [color={rgb, 255:red, 0; green, 0; blue, 0 }  ][line width=0.75]    (10.93,-3.29) .. controls (6.95,-1.4) and (3.31,-0.3) .. (0,0) .. controls (3.31,0.3) and (6.95,1.4) .. (10.93,3.29)   ;
\draw    (164.85,347.08) -- (165.06,373.71) ;
\draw [color={rgb, 255:red, 208; green, 2; blue, 27 }  ,draw opacity=1 ][line width=1.5]    (99.85,306.55) -- (116.92,283.43) ;
\draw    (102.29,260.89) -- (116.92,283.43) ;
\draw    (104.12,256.5) -- (116.92,283.43) ;
\draw [color={rgb, 255:red, 144; green, 19; blue, 254 }  ,draw opacity=1 ][line width=1.5]    (116.92,283.43) -- (162.01,262.5) ;
\draw [color={rgb, 255:red, 245; green, 166; blue, 35 }  ,draw opacity=1 ][line width=1.5]    (161.79,292.06) .. controls (179,294.89) and (187,286.89) .. (203.45,278.31) ;
\draw [color={rgb, 255:red, 144; green, 19; blue, 254 }  ,draw opacity=1 ][line width=1.5]    (162.01,262.5) -- (203.45,278.31) ;
\draw    (162.01,262.5) -- (201.62,256.36) ;
\draw    (201.62,256.36) -- (218.68,232.65) ;
\draw    (201.62,256.36) -- (227.21,265.14) ;
\draw [color={rgb, 255:red, 65; green, 117; blue, 5 }  ,draw opacity=1 ][line width=1.5]    (203.45,278.31) -- (216.24,266.89) ;
\draw    (203.45,278.31) -- (224.77,293.23) ;
\draw    (162.01,280.06) -- (162.01,262.5) ;
\draw [color={rgb, 255:red, 0; green, 0; blue, 0 }  ,draw opacity=1 ][line width=0.75]    (161.79,309.62) -- (161.79,302.89) -- (161.79,292.06) ;
\draw [line width=1.5]    (280.14,456.43) -- (348.14,455.47) ;
\draw [shift={(351.14,455.43)}, rotate = 179.19] [color={rgb, 255:red, 0; green, 0; blue, 0 }  ][line width=1.5]    (19.89,-5.99) .. controls (12.65,-2.54) and (6.02,-0.55) .. (0,0) .. controls (6.02,0.55) and (12.65,2.54) .. (19.89,5.99)   ;
\draw [line width=1.5]    (280.14,293.57) -- (348.14,292.61) ;
\draw [shift={(351.14,292.57)}, rotate = 179.19] [color={rgb, 255:red, 0; green, 0; blue, 0 }  ][line width=1.5]    (19.89,-5.99) .. controls (12.65,-2.54) and (6.02,-0.55) .. (0,0) .. controls (6.02,0.55) and (12.65,2.54) .. (19.89,5.99)   ;
\draw [color={rgb, 255:red, 65; green, 117; blue, 5 }  ,draw opacity=1 ][line width=1.5]    (422.45,454.23) -- (466.87,454.58) ;
\draw [color={rgb, 255:red, 208; green, 2; blue, 27 }  ,draw opacity=1 ][line width=1.5]    (592.57,454.29) -- (539.57,454.64) ;
\draw [color={rgb, 255:red, 65; green, 117; blue, 5 }  ,draw opacity=1 ][line width=1.5]    (421.45,293.23) -- (443.56,293.4) -- (465.87,293.58) ;
\draw [color={rgb, 255:red, 208; green, 2; blue, 27 }  ,draw opacity=1 ][line width=1.5]    (591.57,293.29) -- (538.57,293.64) ;
\draw  [draw opacity=0][line width=1.5]  (466.87,454.58) .. controls (470.61,440.9) and (485.41,430.66) .. (503.13,430.62) .. controls (520.95,430.57) and (535.87,440.86) .. (539.57,454.64) -- (503.21,461.34) -- cycle ; \draw  [color={rgb, 255:red, 144; green, 19; blue, 254 }  ,draw opacity=1 ][line width=1.5]  (466.87,454.58) .. controls (470.61,440.9) and (485.41,430.66) .. (503.13,430.62) .. controls (520.95,430.57) and (535.87,440.86) .. (539.57,454.64) ;  
\draw  [draw opacity=0][line width=1.5]  (465.87,293.58) .. controls (469.61,279.9) and (484.41,269.66) .. (502.13,269.62) .. controls (519.95,269.57) and (534.87,279.86) .. (538.57,293.64) -- (502.21,300.34) -- cycle ; \draw  [color={rgb, 255:red, 144; green, 19; blue, 254 }  ,draw opacity=1 ][line width=1.5]  (465.87,293.58) .. controls (469.61,279.9) and (484.41,269.66) .. (502.13,269.62) .. controls (519.95,269.57) and (534.87,279.86) .. (538.57,293.64) ;  
\draw  [draw opacity=0][line width=1.5]  (539.57,455.32) .. controls (535.7,468.96) and (520.81,479.06) .. (503.1,478.94) .. controls (485.28,478.83) and (470.45,468.4) .. (466.87,454.58) -- (503.3,448.22) -- cycle ; \draw  [color={rgb, 255:red, 245; green, 166; blue, 35 }  ,draw opacity=1 ][line width=1.5]  (539.57,455.32) .. controls (535.7,468.96) and (520.81,479.06) .. (503.1,478.94) .. controls (485.28,478.83) and (470.45,468.4) .. (466.87,454.58) ;  
\draw  [draw opacity=0][line width=1.5]  (538.57,294.32) .. controls (534.7,307.96) and (519.81,318.06) .. (502.1,317.94) .. controls (484.28,317.83) and (469.45,307.4) .. (465.87,293.58) -- (502.3,287.22) -- cycle ; \draw  [color={rgb, 255:red, 245; green, 166; blue, 35 }  ,draw opacity=1 ][line width=1.5]  (538.57,294.32) .. controls (534.7,307.96) and (519.81,318.06) .. (502.1,317.94) .. controls (484.28,317.83) and (469.45,307.4) .. (465.87,293.58) ;  
\draw [color={rgb, 255:red, 245; green, 166; blue, 35 }  ,draw opacity=1 ][line width=1.5]    (116.92,283.43) .. controls (128.6,290.89) and (138.2,294.09) .. (161.79,292.06) ;

\draw (87.41,87.52) node [anchor=north west][inner sep=0.75pt]  [font=\tiny,color={rgb, 255:red, 2; green, 2; blue, 208 }  ,opacity=1 ]  {$1$};
\draw (88.63,77.65) node [anchor=north west][inner sep=0.75pt]  [font=\tiny,color={rgb, 255:red, 2; green, 2; blue, 208 }  ,opacity=1 ]  {$2$};
\draw (83.97,131.42) node [anchor=north west][inner sep=0.75pt]  [font=\tiny,color={rgb, 255:red, 2; green, 2; blue, 208 }  ,opacity=1 ]  {$3$};
\draw (148.49,106.25) node [anchor=north west][inner sep=0.75pt]  [font=\tiny,color={rgb, 255:red, 2; green, 2; blue, 208 }  ,opacity=1 ]  {$2$};
\draw (129.04,93.9) node [anchor=north west][inner sep=0.75pt]  [font=\tiny,color={rgb, 255:red, 2; green, 2; blue, 208 }  ,opacity=1 ]  {$2$};
\draw (129.95,107.28) node [anchor=north west][inner sep=0.75pt]  [font=\tiny,color={rgb, 255:red, 2; green, 2; blue, 208 }  ,opacity=1 ]  {$1$};
\draw (149.4,130.32) node [anchor=north west][inner sep=0.75pt]  [font=\tiny,color={rgb, 255:red, 2; green, 2; blue, 208 }  ,opacity=1 ]  {$1$};
\draw (171.89,76.83) node [anchor=north west][inner sep=0.75pt]  [font=\tiny,color={rgb, 255:red, 2; green, 2; blue, 208 }  ,opacity=1 ]  {$1$};
\draw (171.89,89.17) node [anchor=north west][inner sep=0.75pt]  [font=\tiny,color={rgb, 255:red, 2; green, 2; blue, 208 }  ,opacity=1 ]  {$1$};
\draw (173.1,108.92) node [anchor=north west][inner sep=0.75pt]  [font=\tiny,color={rgb, 255:red, 2; green, 2; blue, 208 }  ,opacity=1 ]  {$1$};
\draw (216.25,101.52) node [anchor=north west][inner sep=0.75pt]  [font=\tiny,color={rgb, 255:red, 2; green, 2; blue, 208 }  ,opacity=1 ]  {$1$};
\draw (224.76,126.21) node [anchor=north west][inner sep=0.75pt]  [font=\tiny,color={rgb, 255:red, 2; green, 2; blue, 208 }  ,opacity=1 ]  {$1$};
\draw (218.38,65.1) node [anchor=north west][inner sep=0.75pt]  [font=\tiny,color={rgb, 255:red, 2; green, 2; blue, 208 }  ,opacity=1 ]  {$2$};
\draw (226.89,97.2) node [anchor=north west][inner sep=0.75pt]  [font=\tiny,color={rgb, 255:red, 2; green, 2; blue, 208 }  ,opacity=1 ]  {$2$};
\draw (570,126.73) node [anchor=north west][inner sep=0.75pt]    {$x$};
\draw (476,116.4) node [anchor=north west][inner sep=0.75pt]    {$y$};
\draw (126,77.4) node [anchor=north west][inner sep=0.75pt]    {$u$};
\draw (176,122.4) node [anchor=north west][inner sep=0.75pt]    {$v$};
\draw (284,92.4) node [anchor=north west][inner sep=0.75pt]    {$\varphi_2^\trop$};
\draw (289,158.4) node [anchor=north west][inner sep=0.75pt]  [font=\footnotesize]  {$y\ =\ u$};
\draw (288,143.4) node [anchor=north west][inner sep=0.75pt]  [font=\footnotesize]  {$x\ =\ 2v$};
\draw (132.38,445.7) node [anchor=north west][inner sep=0.75pt]  [font=\tiny,color={rgb, 255:red, 2; green, 2; blue, 208 }  ,opacity=1 ]  {$1$};
\draw (133.09,426.75) node [anchor=north west][inner sep=0.75pt]  [font=\tiny,color={rgb, 255:red, 2; green, 2; blue, 208 }  ,opacity=1 ]  {$2$};
\draw (90.74,421.71) node [anchor=north west][inner sep=0.75pt]  [font=\tiny,color={rgb, 255:red, 2; green, 2; blue, 208 }  ,opacity=1 ]  {$1$};
\draw (91.96,411.17) node [anchor=north west][inner sep=0.75pt]  [font=\tiny,color={rgb, 255:red, 2; green, 2; blue, 208 }  ,opacity=1 ]  {$2$};
\draw (87.28,468.53) node [anchor=north west][inner sep=0.75pt]  [font=\tiny,color={rgb, 255:red, 2; green, 2; blue, 208 }  ,opacity=1 ]  {$3$};
\draw (222.87,455.73) node [anchor=north west][inner sep=0.75pt]  [font=\tiny,color={rgb, 255:red, 2; green, 2; blue, 208 }  ,opacity=1 ]  {$2$};
\draw (213.12,433.56) node [anchor=north west][inner sep=0.75pt]  [font=\tiny,color={rgb, 255:red, 2; green, 2; blue, 208 }  ,opacity=1 ]  {$2$};
\draw (217.69,392.15) node [anchor=north west][inner sep=0.75pt]  [font=\tiny,color={rgb, 255:red, 2; green, 2; blue, 208 }  ,opacity=1 ]  {$1$};
\draw (226.22,428.14) node [anchor=north west][inner sep=0.75pt]  [font=\tiny,color={rgb, 255:red, 2; green, 2; blue, 208 }  ,opacity=1 ]  {$1$};
\draw (178.69,410.59) node [anchor=north west][inner sep=0.75pt]  [font=\tiny,color={rgb, 255:red, 2; green, 2; blue, 208 }  ,opacity=1 ]  {$1$};
\draw (178.69,426.1) node [anchor=north west][inner sep=0.75pt]  [font=\tiny,color={rgb, 255:red, 2; green, 2; blue, 208 }  ,opacity=1 ]  {$1$};
\draw (178.69,441.31) node [anchor=north west][inner sep=0.75pt]  [font=\tiny,color={rgb, 255:red, 2; green, 2; blue, 208 }  ,opacity=1 ]  {$1$};
\draw (151.98,441.68) node [anchor=north west][inner sep=0.75pt]  [font=\tiny,color={rgb, 255:red, 2; green, 2; blue, 208 }  ,opacity=1 ]  {$2$};
\draw (151.27,472.78) node [anchor=north west][inner sep=0.75pt]  [font=\tiny,color={rgb, 255:red, 2; green, 2; blue, 208 }  ,opacity=1 ]  {$1$};
\draw (134.38,281.7) node [anchor=north west][inner sep=0.75pt]  [font=\tiny,color={rgb, 255:red, 2; green, 2; blue, 208 }  ,opacity=1 ]  {$1$};
\draw (135.09,262.75) node [anchor=north west][inner sep=0.75pt]  [font=\tiny,color={rgb, 255:red, 2; green, 2; blue, 208 }  ,opacity=1 ]  {$2$};
\draw (92.74,257.71) node [anchor=north west][inner sep=0.75pt]  [font=\tiny,color={rgb, 255:red, 2; green, 2; blue, 208 }  ,opacity=1 ]  {$1$};
\draw (93.96,247.17) node [anchor=north west][inner sep=0.75pt]  [font=\tiny,color={rgb, 255:red, 2; green, 2; blue, 208 }  ,opacity=1 ]  {$2$};
\draw (89.28,304.53) node [anchor=north west][inner sep=0.75pt]  [font=\tiny,color={rgb, 255:red, 2; green, 2; blue, 208 }  ,opacity=1 ]  {$3$};
\draw (224.87,291.73) node [anchor=north west][inner sep=0.75pt]  [font=\tiny,color={rgb, 255:red, 2; green, 2; blue, 208 }  ,opacity=1 ]  {$2$};
\draw (215.12,269.56) node [anchor=north west][inner sep=0.75pt]  [font=\tiny,color={rgb, 255:red, 2; green, 2; blue, 208 }  ,opacity=1 ]  {$2$};
\draw (219.69,228.15) node [anchor=north west][inner sep=0.75pt]  [font=\tiny,color={rgb, 255:red, 2; green, 2; blue, 208 }  ,opacity=1 ]  {$1$};
\draw (228.22,264.14) node [anchor=north west][inner sep=0.75pt]  [font=\tiny,color={rgb, 255:red, 2; green, 2; blue, 208 }  ,opacity=1 ]  {$1$};
\draw (180.69,246.59) node [anchor=north west][inner sep=0.75pt]  [font=\tiny,color={rgb, 255:red, 2; green, 2; blue, 208 }  ,opacity=1 ]  {$1$};
\draw (180.69,262.1) node [anchor=north west][inner sep=0.75pt]  [font=\tiny,color={rgb, 255:red, 2; green, 2; blue, 208 }  ,opacity=1 ]  {$1$};
\draw (180.69,277.31) node [anchor=north west][inner sep=0.75pt]  [font=\tiny,color={rgb, 255:red, 2; green, 2; blue, 208 }  ,opacity=1 ]  {$1$};
\draw (153.98,277.68) node [anchor=north west][inner sep=0.75pt]  [font=\tiny,color={rgb, 255:red, 2; green, 2; blue, 208 }  ,opacity=1 ]  {$2$};
\draw (153.27,308.78) node [anchor=north west][inner sep=0.75pt]  [font=\tiny,color={rgb, 255:red, 2; green, 2; blue, 208 }  ,opacity=1 ]  {$1$};
\draw (287,422.4) node [anchor=north west][inner sep=0.75pt]    {$\varphi_2^\trop$};
\draw (292,488.4) node [anchor=north west][inner sep=0.75pt]  [font=\footnotesize]  {$x_{2} \ =\ 2u+2v$};
\draw (291,473.4) node [anchor=north west][inner sep=0.75pt]  [font=\footnotesize]  {$x_{1} \ =\ u$};
\draw (287,259.54) node [anchor=north west][inner sep=0.75pt]    {$\varphi_2^\trop$};
\draw (292,325.54) node [anchor=north west][inner sep=0.75pt]  [font=\footnotesize]  {$x_{2} \ =\ 2u+v$};
\draw (291,310.54) node [anchor=north west][inner sep=0.75pt]  [font=\footnotesize]  {$x_{1} \ =\ u+v$};
\draw (494,403.4) node [anchor=north west][inner sep=0.75pt]    {$x_{1}$};
\draw (493.33,455.73) node [anchor=north west][inner sep=0.75pt]    {$x_{2}$};
\draw (493,242.4) node [anchor=north west][inner sep=0.75pt]    {$x_{1}$};
\draw (493.67,293.4) node [anchor=north west][inner sep=0.75pt]    {$x_{2}$};
\draw (130,409.4) node [anchor=north west][inner sep=0.75pt]    {$u$};
\draw (133,243.4) node [anchor=north west][inner sep=0.75pt]    {$u$};
\draw (177,451.4) node [anchor=north west][inner sep=0.75pt]    {$v$};
\draw (182,291.4) node [anchor=north west][inner sep=0.75pt]    {$v$};
\end{tikzpicture}

\caption{The type of admissible covers contributing to the computation of the degree of the map $\varphi_2^{\trop}$. We have colored the ends that are remembered, as well as the edges whose lengths are relevant for the computation of the degree.}
\label{fig:deg6}
\end{figure}

\begin{lemma} 
The function  $\varphi_2^\trop$ has degree $6$.
\end{lemma}

\begin{proof}
We proceed analogously to Lemma \ref{lem:degphione}. Again, we need to analyze three distinct situations. Refer to Figure \ref{fig:deg6}.

\noindent{\textsc{Case I:}} $\tilde\xi$ belongs to the interior of the two-dimensional cone of $\mathcal{M}_{1,2}^\trop$ parameterizing tropical curves where the two marks emanate from the same vertex.
There are three inverse images for $\tilde\xi$,  belonging to  cones of type $\sigma_3$ where the $2$-end we remember is on the same vertex as the $3$-end (see the first line of Figure \ref{fig:deg6}). For any such inverse image $\xi$, the local degree of $\varphi_1^\trop$ is
\begin{equation}
    \deg_\xi(\varphi_2^\trop) = \frac{2}{2}\left|\det
    \left[
    \begin{array}{cc}
       0  & 2 \\
        1 & 0
    \end{array}
    \right]
    \right|= 2 \ .
\end{equation}
Since there are three inverse images, the sum of the local degrees equals $6$.

\noindent{\textsc{Case II:}} $\tilde\xi$ belongs to the interior of the two-dimensional  folded cone $\tilde\sigma$ of $\mathcal{M}_{1,2}^\trop$. Assume $\tilde\xi$ lies in the region whose inverse image is  $\{x_1/2<x_2<2x_1\}$. Then $\tilde\xi$ has a total of $6$ inverse images, belonging to cones of type $\sigma_2$ with the $2$-end mapped to an external edge of the target tree (second line of Figure \ref{fig:deg6}). The local degree for these inverse images is:
\begin{equation}
    \deg_\xi(\varphi_2^\trop) = \frac{1}{1}\left|\det
    \left[
    \begin{array}{cc}
       1  & 1 \\
        2 & 1
    \end{array}
    \right]
    \right|= 1 \ .
\end{equation}
It follows that the sum of the local degrees over all inverse images equals $6$.

\noindent{\textsc{Case III:}} $\tilde\xi$ belongs to the interior of the two-dimensional  folded cone $\tilde\sigma$ of $\mathcal{M}_{1,2}^\trop$, in the region  $\{x_2<x_1/2\}\cup \{x_2>2x_1\}$. There are three inverse images, that
belong to cones of type $\sigma_2$ where the $2$-end that is remembered maps to the middle edge of the target graph (see third line of Figure \ref{fig:deg6}).
For each of these inverse images the local degree is:
\begin{equation}
    \deg_\xi(\varphi_2^\trop) = \frac{1}{1}\left|\det
    \left[
    \begin{array}{cc}
       1  & 0 \\
        2 & 2
    \end{array}
    \right]
    \right|= 2 \ .
\end{equation}
We have thus verified that the degree of $\varphi_2^\trop$ is  equal to $6$.
\end{proof}

\begin{lemma}\label{lem:pftwo}
Let $\rho_{irr}, \rho_{sec}$ and $\rho_E$  be as in Lemma \ref{lem:tropW}. We have 
\begin{equation}
    \varphi^\trop_{2, \ast}(\psi_1^\trop \frown \mtt T_\ast([\Sigma])) =  6\rho_{irr}+ 3\rho_E \ .
\end{equation}
\end{lemma}

\begin{proof}
As before, all rays of type $\rho_a$ are contracted to the cone point of $\overline{\calM}^\trop_{1,2}$, and therefore their pushforward vanishes.

Three  of the rays of type $\rho_b$, where the $2$-end is on the same vertex as the $3$-end,  map to the ray we have denoted $\rho_{irr}$, and the function $\varphi_2^\trop$ restricts to each of these rays as a linear function of slope $2$. 
Three  of the rays of type $\rho_b$, where the $2$-end is on a different vertex than the $3$-end,  map to the ray we have denoted $\rho_{E}$, and the function $\varphi_2^\trop$ restricts to each of these rays as a linear function of slope $1$. 

Both the tropical covers $\Gamma \to T$ parameterized by the rays of type $\rho_b$ and the tropical curves $\tilde\Gamma$ parameterized by $\rho_{irr}$ and $\rho_E$ have  a group of automorphisms of order two. Hence:
\begin{equation}
    \varphi^\trop_{2,\ast}(\rho_b^\tot) = \frac{|\Aut(\tilde{\Gamma})|}{|\Aut(\Gamma\to T)|}\cdot (2\cdot 3\rho_{irr}+1\cdot3 \rho_E) = 6\rho_{irr}+3 \rho_E \ .
\end{equation}
\end{proof}

We now observe that this result is compatible with the algebraic computation \eqref{eq:secondpushpullrel} from Lemma \ref{lem:pushpsi}. 
We have computed the operational tropicalization of $\psi$ to be $\frac{1}{2}\rho_{irr}$ in Lemma \ref{lem:pfone}, and of $W$ to be $\frac{1}{2}(\rho_{irr}+\rho_E)$ in Lemma \ref{lem:tropW}.

It follows immediately that 
\begin{equation}
    \varphi^\trop_{2,\ast}(\psi_1^\trop) =  \Trop(6(\psi_1+W)) \ .
\end{equation}

\bibliographystyle{alpha}
\bibliography{lib}

\end{document}

%% file: normal.tikz.tex
\tikzset{every picture/.style={line width=0.75pt}} 

\begin{tikzpicture}[x=0.75pt,y=0.75pt,yscale=-1,xscale=1]

\draw    (160.8,21.2) -- (160.6,210.4) ;
\draw [shift={(160.8,19.2)}, rotate = 90.06] [color={rgb, 255:red, 0; green, 0; blue, 0 }  ][line width=0.75]    (10.93,-3.29) .. controls (6.95,-1.4) and (3.31,-0.3) .. (0,0) .. controls (3.31,0.3) and (6.95,1.4) .. (10.93,3.29)   ;
\draw    (304.46,117.27) -- (42,181.2) ;
\draw [shift={(306.4,116.8)}, rotate = 166.31] [color={rgb, 255:red, 0; green, 0; blue, 0 }  ][line width=0.75]    (10.93,-3.29) .. controls (6.95,-1.4) and (3.31,-0.3) .. (0,0) .. controls (3.31,0.3) and (6.95,1.4) .. (10.93,3.29)   ;
\draw    (310.4,161.48) -- (38.8,145.2) ;
\draw [shift={(312.4,161.6)}, rotate = 183.43] [color={rgb, 255:red, 0; green, 0; blue, 0 }  ][line width=0.75]    (10.93,-3.29) .. controls (6.95,-1.4) and (3.31,-0.3) .. (0,0) .. controls (3.31,0.3) and (6.95,1.4) .. (10.93,3.29)   ;
\draw    (511.2,21.2) -- (511.4,211.6) ;
\draw [shift={(511.2,19.2)}, rotate = 89.94] [color={rgb, 255:red, 0; green, 0; blue, 0 }  ][line width=0.75]    (10.93,-3.29) .. controls (6.95,-1.4) and (3.31,-0.3) .. (0,0) .. controls (3.31,0.3) and (6.95,1.4) .. (10.93,3.29)   ;
\draw    (652,140.01) -- (394.4,140.8) ;
\draw [shift={(654,140)}, rotate = 179.82] [color={rgb, 255:red, 0; green, 0; blue, 0 }  ][line width=0.75]    (10.93,-3.29) .. controls (6.95,-1.4) and (3.31,-0.3) .. (0,0) .. controls (3.31,0.3) and (6.95,1.4) .. (10.93,3.29)   ;
\draw    (318.2,139.2) -- (374.4,139.59) ;
\draw [shift={(376.4,139.6)}, rotate = 180.39] [color={rgb, 255:red, 0; green, 0; blue, 0 }  ][line width=0.75]    (10.93,-3.29) .. controls (6.95,-1.4) and (3.31,-0.3) .. (0,0) .. controls (3.31,0.3) and (6.95,1.4) .. (10.93,3.29)   ;
\draw [color={rgb, 255:red, 208; green, 2; blue, 27 }  ,draw opacity=1 ][line width=2.25]    (160.2,152.4) -- (160.56,105.4) ;
\draw [shift={(160.6,100.4)}, rotate = 90.44] [fill={rgb, 255:red, 208; green, 2; blue, 27 }  ,fill opacity=1 ][line width=0.08]  [draw opacity=0] (14.29,-6.86) -- (0,0) -- (14.29,6.86) -- cycle    ;
\draw [color={rgb, 255:red, 208; green, 2; blue, 27 }  ,draw opacity=1 ][line width=2.25]    (160.2,152.4) -- (103.93,196.58) ;
\draw [shift={(100,199.67)}, rotate = 321.86] [fill={rgb, 255:red, 208; green, 2; blue, 27 }  ,fill opacity=1 ][line width=0.08]  [draw opacity=0] (14.29,-6.86) -- (0,0) -- (14.29,6.86) -- cycle    ;
\draw [color={rgb, 255:red, 208; green, 2; blue, 27 }  ,draw opacity=1 ][line width=2.25]    (160.2,152.4) -- (211.61,155.68) ;
\draw [shift={(216.6,156)}, rotate = 183.65] [fill={rgb, 255:red, 208; green, 2; blue, 27 }  ,fill opacity=1 ][line width=0.08]  [draw opacity=0] (14.29,-6.86) -- (0,0) -- (14.29,6.86) -- cycle    ;
\draw [color={rgb, 255:red, 208; green, 2; blue, 27 }  ,draw opacity=1 ][line width=2.25]    (160.2,152.4) -- (223.35,136.8) ;
\draw [shift={(228.2,135.6)}, rotate = 166.12] [fill={rgb, 255:red, 208; green, 2; blue, 27 }  ,fill opacity=1 ][line width=0.08]  [draw opacity=0] (14.29,-6.86) -- (0,0) -- (14.29,6.86) -- cycle    ;
\draw [color={rgb, 255:red, 208; green, 2; blue, 27 }  ,draw opacity=1 ][line width=2.25]    (511.4,141.2) -- (511.76,187.4) ;
\draw [shift={(511.8,192.4)}, rotate = 269.55] [fill={rgb, 255:red, 208; green, 2; blue, 27 }  ,fill opacity=1 ][line width=0.08]  [draw opacity=0] (14.29,-6.86) -- (0,0) -- (14.29,6.86) -- cycle    ;
\draw [color={rgb, 255:red, 208; green, 2; blue, 27 }  ,draw opacity=1 ][line width=2.25]    (511.4,141.2) -- (511.76,94.2) ;
\draw [shift={(511.8,89.2)}, rotate = 90.44] [fill={rgb, 255:red, 208; green, 2; blue, 27 }  ,fill opacity=1 ][line width=0.08]  [draw opacity=0] (14.29,-6.86) -- (0,0) -- (14.29,6.86) -- cycle    ;
\draw [color={rgb, 255:red, 208; green, 2; blue, 27 }  ,draw opacity=1 ][line width=2.25]    (511.4,141.2) -- (556.4,140.48) ;
\draw [shift={(561.4,140.4)}, rotate = 179.08] [fill={rgb, 255:red, 208; green, 2; blue, 27 }  ,fill opacity=1 ][line width=0.08]  [draw opacity=0] (14.29,-6.86) -- (0,0) -- (14.29,6.86) -- cycle    ;
\draw [color={rgb, 255:red, 208; green, 2; blue, 27 }  ,draw opacity=1 ][line width=2.25]    (511.4,141.2) -- (466,141.2) ;
\draw [shift={(461,141.2)}, rotate = 360] [fill={rgb, 255:red, 208; green, 2; blue, 27 }  ,fill opacity=1 ][line width=0.08]  [draw opacity=0] (14.29,-6.86) -- (0,0) -- (14.29,6.86) -- cycle    ;

\draw (287,164.8) node [anchor=north west][inner sep=0.75pt]    {$x$};
\draw (287.8,123.2) node [anchor=north west][inner sep=0.75pt]    {$y$};
\draw (171.4,22) node [anchor=north west][inner sep=0.75pt]    {$z$};
\draw (627,145.6) node [anchor=north west][inner sep=0.75pt]    {$u$};
\draw (523.4,23.6) node [anchor=north west][inner sep=0.75pt]    {$v$};
\draw (331.2,96.2) node [anchor=north west][inner sep=0.75pt]  [font=\footnotesize]  {$ \begin{array}{l}
u\ =\ x-z\\
v\ =\ y
\end{array}$};

\end{tikzpicture}

%% file: raysandcones.tikz.tex
\tikzset{every picture/.style={line width=0.75pt}} 

\begin{tikzpicture}[x=0.75pt,y=0.75pt,yscale=-1,xscale=1]

\draw  [color={rgb, 255:red, 208; green, 2; blue, 27 }  ,draw opacity=1 ][fill={rgb, 255:red, 208; green, 2; blue, 27 }  ,fill opacity=0.1 ] (50,400.33) -- (399.44,400.33) -- (399.44,602) -- (50,602) -- cycle ;
\draw  [color={rgb, 255:red, 245; green, 166; blue, 35 }  ,draw opacity=1 ][fill={rgb, 255:red, 245; green, 166; blue, 35 }  ,fill opacity=0.1 ] (474.11,398) -- (823.17,398) -- (823.17,601.33) -- (474.11,601.33) -- cycle ;
\draw  [color={rgb, 255:red, 189; green, 16; blue, 224 }  ,draw opacity=1 ][fill={rgb, 255:red, 189; green, 16; blue, 224 }  ,fill opacity=0.1 ] (901.11,399.67) -- (1249.5,399.67) -- (1249.5,601.67) -- (901.11,601.67) -- cycle ;
\draw  [color={rgb, 255:red, 74; green, 144; blue, 226 }  ,draw opacity=1 ][fill={rgb, 255:red, 74; green, 144; blue, 226 }  ,fill opacity=0.1 ] (252,649.33) -- (601.67,649.33) -- (601.67,851.33) -- (252,851.33) -- cycle ;
\draw  [color={rgb, 255:red, 126; green, 211; blue, 33 }  ,draw opacity=1 ][fill={rgb, 255:red, 126; green, 211; blue, 33 }  ,fill opacity=0.1 ] (691.33,651.33) -- (1041.67,651.33) -- (1041.67,850.78) -- (691.33,850.78) -- cycle ;
\draw  [color={rgb, 255:red, 2; green, 2; blue, 208 }  ,draw opacity=1 ][fill={rgb, 255:red, 2; green, 2; blue, 208 }  ,fill opacity=0.1 ] (50,150) -- (300,150) -- (300,352) -- (50,352) -- cycle ;
\draw  [color={rgb, 255:red, 155; green, 155; blue, 155 }  ,draw opacity=1 ][fill={rgb, 255:red, 155; green, 155; blue, 155 }  ,fill opacity=0.1 ] (367,150) -- (617,150) -- (617,352) -- (367,352) -- cycle ;
\draw  [color={rgb, 255:red, 144; green, 19; blue, 254 }  ,draw opacity=1 ][fill={rgb, 255:red, 144; green, 19; blue, 254 }  ,fill opacity=0.1 ] (683,149) -- (933,149) -- (933,351) -- (683,351) -- cycle ;
\draw  [color={rgb, 255:red, 248; green, 231; blue, 28 }  ,draw opacity=1 ][fill={rgb, 255:red, 248; green, 231; blue, 28 }  ,fill opacity=0.1 ] (1000,149) -- (1250,149) -- (1250,351) -- (1000,351) -- cycle ;
\draw   (116,233) .. controls (116,228.58) and (119.58,225) .. (124,225) .. controls (128.42,225) and (132,228.58) .. (132,233) .. controls (132,237.42) and (128.42,241) .. (124,241) .. controls (119.58,241) and (116,237.42) .. (116,233) -- cycle ;
\draw    (201,233.33) -- (132,233) ;
\draw    (201,233.33) -- (232,203.33) ;
\draw    (201,233.33) -- (233,211.33) ;
\draw    (201,233.33) -- (232,257.33) ;
\draw    (201,233.33) -- (233,267.33) ;
\draw    (200,308.33) -- (126,308) ;
\draw    (200,308.33) -- (232,286.33) ;
\draw    (200,308.33) -- (231,332.33) ;
\draw    (162.5,256.33) -- (162.5,289.33) ;
\draw [shift={(162.5,291.33)}, rotate = 270] [color={rgb, 255:red, 0; green, 0; blue, 0 }  ][line width=0.75]    (10.93,-3.29) .. controls (6.95,-1.4) and (3.31,-0.3) .. (0,0) .. controls (3.31,0.3) and (6.95,1.4) .. (10.93,3.29)   ;
\draw [color={rgb, 255:red, 208; green, 2; blue, 27 }  ,draw opacity=1 ]   (98,334.33) -- (126,308) ;
\draw    (100,308.33) -- (126,308) ;
\draw    (102,282.33) -- (126,308) ;
\draw    (99,208) -- (119,226) ;
\draw    (90,230.33) -- (116,230) ;
\draw    (98,258) -- (118,238) ;
\draw    (102.5,204.67) -- (124,225) ;
\draw    (90,234.33) -- (116,234) ;
\draw    (519,310.33) -- (445,310) ;
\draw    (519,310.33) -- (551,288.33) ;
\draw    (519,310.33) -- (550,334.33) ;
\draw [color={rgb, 255:red, 208; green, 2; blue, 27 }  ,draw opacity=1 ]   (417,336.33) -- (445,310) ;
\draw    (419,310.33) -- (445,310) ;
\draw    (421,284.33) -- (445,310) ;
\draw    (518,215.67) -- (448,248.83) ;
\draw [color={rgb, 255:red, 0; green, 0; blue, 0 }  ,draw opacity=1 ]   (420,275.17) -- (448,248.83) ;
\draw    (415,251.67) -- (448,248.83) ;
\draw    (416,245.67) -- (448,248.83) ;
\draw    (424,223.17) -- (448,248.83) ;
\draw    (427,218.17) -- (448,248.83) ;
\draw  [draw opacity=0] (448,248.83) .. controls (457.14,245.57) and (469.77,243.55) .. (483.73,243.55) .. controls (498.94,243.55) and (512.58,245.95) .. (521.84,249.74) -- (483.73,261.55) -- cycle ; \draw   (448,248.83) .. controls (457.14,245.57) and (469.77,243.55) .. (483.73,243.55) .. controls (498.94,243.55) and (512.58,245.95) .. (521.84,249.74) ;  
\draw  [draw opacity=0] (521.84,249.74) .. controls (512.72,253.08) and (500.1,255.19) .. (486.15,255.3) .. controls (470.94,255.42) and (457.28,253.13) .. (448,249.41) -- (486.01,237.3) -- cycle ; \draw   (521.84,249.74) .. controls (512.72,253.08) and (500.1,255.19) .. (486.15,255.3) .. controls (470.94,255.42) and (457.28,253.13) .. (448,249.41) ;  
\draw    (518,215.67) -- (549,186.67) ;
\draw    (518,215.67) -- (535,227.67) ;
\draw    (521.84,249.74) -- (554,222.67) ;
\draw    (521.84,249.74) -- (551,267.67) ;
\draw    (483.67,267.33) -- (483.67,294.33) ;
\draw [shift={(483.67,296.33)}, rotate = 270] [color={rgb, 255:red, 0; green, 0; blue, 0 }  ][line width=0.75]    (10.93,-3.29) .. controls (6.95,-1.4) and (3.31,-0.3) .. (0,0) .. controls (3.31,0.3) and (6.95,1.4) .. (10.93,3.29)   ;
\draw    (830.33,311.33) -- (756.33,311) ;
\draw    (830.33,311.33) -- (862.33,289.33) ;
\draw    (830.33,311.33) -- (861.33,335.33) ;
\draw [color={rgb, 255:red, 208; green, 2; blue, 27 }  ,draw opacity=1 ]   (728.33,337.33) -- (756.33,311) ;
\draw    (830.33,311.33) -- (862.33,311.33) ;
\draw    (732.33,285.33) -- (756.33,311) ;
\draw [color={rgb, 255:red, 0; green, 0; blue, 0 }  ,draw opacity=1 ]   (733.33,264.17) -- (761.33,237.83) ;
\draw    (737.33,212.17) -- (761.33,237.83) ;
\draw    (740.33,207.17) -- (761.33,237.83) ;
\draw  [draw opacity=0] (761.33,237.83) .. controls (770.47,234.57) and (783.11,232.55) .. (797.06,232.55) .. controls (812.27,232.55) and (825.91,234.95) .. (835.17,238.74) -- (797.06,250.55) -- cycle ; \draw   (761.33,237.83) .. controls (770.47,234.57) and (783.11,232.55) .. (797.06,232.55) .. controls (812.27,232.55) and (825.91,234.95) .. (835.17,238.74) ;  
\draw  [draw opacity=0] (835.17,238.74) .. controls (826.05,242.08) and (813.44,244.19) .. (799.48,244.3) .. controls (784.27,244.42) and (770.62,242.13) .. (761.33,238.41) -- (799.34,226.3) -- cycle ; \draw   (835.17,238.74) .. controls (826.05,242.08) and (813.44,244.19) .. (799.48,244.3) .. controls (784.27,244.42) and (770.62,242.13) .. (761.33,238.41) ;  
\draw    (835.17,238.74) -- (867.33,211.67) ;
\draw    (835.17,238.74) -- (864.33,256.67) ;
\draw    (797,256.33) -- (797.32,294.33) ;
\draw [shift={(797.33,296.33)}, rotate = 269.52] [color={rgb, 255:red, 0; green, 0; blue, 0 }  ][line width=0.75]    (10.93,-3.29) .. controls (6.95,-1.4) and (3.31,-0.3) .. (0,0) .. controls (3.31,0.3) and (6.95,1.4) .. (10.93,3.29)   ;
\draw    (835.17,238.74) -- (865.33,205.67) ;
\draw    (835.17,238.74) -- (866.33,233.33) ;
\draw    (835.17,238.74) -- (864.33,239.33) ;
\draw    (835.17,238.74) -- (864.33,263.67) ;
\draw    (1154.33,311.33) -- (1080.33,311) ;
\draw    (1154.33,311.33) -- (1186.33,289.33) ;
\draw    (1154.33,311.33) -- (1185.33,335.33) ;
\draw [color={rgb, 255:red, 208; green, 2; blue, 27 }  ,draw opacity=1 ]   (1052.33,337.33) -- (1080.33,311) ;
\draw    (1154.33,311.33) -- (1186.33,311.33) ;
\draw    (1056.33,285.33) -- (1080.33,311) ;
\draw [color={rgb, 255:red, 0; green, 0; blue, 0 }  ,draw opacity=1 ]   (1057.33,264.17) -- (1085.33,237.83) ;
\draw    (1061.33,212.17) -- (1085.33,237.83) ;
\draw    (1064.33,207.17) -- (1085.33,237.83) ;
\draw    (1121,256.33) -- (1121.32,294.33) ;
\draw [shift={(1121.33,296.33)}, rotate = 269.52] [color={rgb, 255:red, 0; green, 0; blue, 0 }  ][line width=0.75]    (10.93,-3.29) .. controls (6.95,-1.4) and (3.31,-0.3) .. (0,0) .. controls (3.31,0.3) and (6.95,1.4) .. (10.93,3.29)   ;
\draw    (1152.78,258.56) -- (1085.33,237.83) ;
\draw    (1149.33,206) -- (1085.33,237.83) ;
\draw    (1150.33,258) -- (1182.33,236) ;
\draw    (1150.33,258) -- (1181.33,282) ;
\draw    (1150.33,258) -- (1182.33,258) ;
\draw    (1163.33,200) -- (1185.33,183.67) ;
\draw    (1163,212.33) -- (1185.33,228.67) ;
\draw    (1166,205.67) -- (1185.33,204.67) ;
\draw   (1149.33,206) .. controls (1149.33,201.58) and (1152.92,198) .. (1157.33,198) .. controls (1161.75,198) and (1165.33,201.58) .. (1165.33,206) .. controls (1165.33,210.42) and (1161.75,214) .. (1157.33,214) .. controls (1152.92,214) and (1149.33,210.42) .. (1149.33,206) -- cycle ;
\draw    (221,559.33) -- (147,559) ;
\draw [color={rgb, 255:red, 208; green, 2; blue, 27 }  ,draw opacity=1 ]   (118,585.33) -- (146,559) ;
\draw    (122,533.33) -- (146,559) ;
\draw [color={rgb, 255:red, 0; green, 0; blue, 0 }  ,draw opacity=1 ]   (124,512.17) -- (152,485.83) ;
\draw    (128,460.17) -- (152,485.83) ;
\draw    (131,455.17) -- (152,485.83) ;
\draw  [draw opacity=0] (152,485.83) .. controls (161.14,482.57) and (173.77,480.55) .. (187.73,480.55) .. controls (202.94,480.55) and (216.58,482.95) .. (225.84,486.74) -- (187.73,498.55) -- cycle ; \draw   (152,485.83) .. controls (161.14,482.57) and (173.77,480.55) .. (187.73,480.55) .. controls (202.94,480.55) and (216.58,482.95) .. (225.84,486.74) ;  
\draw  [draw opacity=0] (225.84,486.74) .. controls (216.72,490.08) and (204.1,492.19) .. (190.15,492.3) .. controls (174.94,492.42) and (161.28,490.13) .. (152,486.41) -- (190.01,474.3) -- cycle ; \draw   (225.84,486.74) .. controls (216.72,490.08) and (204.1,492.19) .. (190.15,492.3) .. controls (174.94,492.42) and (161.28,490.13) .. (152,486.41) ;  
\draw    (293,486.67) -- (224,486.33) ;
\draw    (293,486.67) -- (324,456.67) ;
\draw    (293,486.67) -- (325,464.67) ;
\draw    (293,486.67) -- (324,510.67) ;
\draw    (293,486.67) -- (325,520.67) ;
\draw    (292,559.67) -- (218,559.33) ;
\draw    (292,559.67) -- (324,537.67) ;
\draw    (292,559.67) -- (323,583.67) ;
\draw    (226.5,521.67) -- (226.55,549) ;
\draw [shift={(226.56,551)}, rotate = 269.89] [color={rgb, 255:red, 0; green, 0; blue, 0 }  ][line width=0.75]    (10.93,-3.29) .. controls (6.95,-1.4) and (3.31,-0.3) .. (0,0) .. controls (3.31,0.3) and (6.95,1.4) .. (10.93,3.29)   ;
\draw    (227.67,559.33) -- (228,588.67) ;
\draw    (225.84,486.74) -- (219.89,520.33) ;
\draw    (225.84,486.74) -- (233.22,517.67) ;
\draw    (647,560.33) -- (573,560) ;
\draw [color={rgb, 255:red, 208; green, 2; blue, 27 }  ,draw opacity=1 ]   (544,586.33) -- (572,560) ;
\draw    (548,534.33) -- (572,560) ;
\draw    (718,560.67) -- (644,560.33) ;
\draw    (718,560.67) -- (750,538.67) ;
\draw    (718,560.67) -- (749,584.67) ;
\draw    (662.5,512.67) -- (662.55,540) ;
\draw [shift={(662.56,542)}, rotate = 269.89] [color={rgb, 255:red, 0; green, 0; blue, 0 }  ][line width=0.75]    (10.93,-3.29) .. controls (6.95,-1.4) and (3.31,-0.3) .. (0,0) .. controls (3.31,0.3) and (6.95,1.4) .. (10.93,3.29)   ;
\draw    (653.67,560.33) -- (654,590.67) ;
\draw [color={rgb, 255:red, 0; green, 0; blue, 0 }  ,draw opacity=1 ]   (547,514.17) -- (575,487.83) ;
\draw    (551,462.17) -- (575,487.83) ;
\draw    (554,457.17) -- (575,487.83) ;
\draw    (575,487.83) -- (649,464) ;
\draw    (575,487.83) .. controls (639,502) and (651,503) .. (717,482) ;
\draw    (649,464) -- (717,482) ;
\draw    (649,464) -- (714,457) ;
\draw    (714,457) -- (742,430) ;
\draw    (714,457) -- (756,467) ;
\draw    (717,482) -- (738,469) ;
\draw    (717,482) -- (752,499) ;
\draw    (649,484) -- (649,464) ;
\draw    (648.64,517.67) -- (648.64,510) -- (648.64,497.67) ;
\draw    (1083,559.33) -- (1009,559) ;
\draw [color={rgb, 255:red, 208; green, 2; blue, 27 }  ,draw opacity=1 ]   (980,585.33) -- (1008,559) ;
\draw    (984,533.33) -- (1008,559) ;
\draw    (1154,559.67) -- (1080,559.33) ;
\draw    (1154,559.67) -- (1186,537.67) ;
\draw    (1154,559.67) -- (1185,583.67) ;
\draw    (1098.5,511.67) -- (1098.55,539) ;
\draw [shift={(1098.56,541)}, rotate = 269.89] [color={rgb, 255:red, 0; green, 0; blue, 0 }  ][line width=0.75]    (10.93,-3.29) .. controls (6.95,-1.4) and (3.31,-0.3) .. (0,0) .. controls (3.31,0.3) and (6.95,1.4) .. (10.93,3.29)   ;
\draw    (1089.67,559.33) -- (1090,589.67) ;
\draw [color={rgb, 255:red, 0; green, 0; blue, 0 }  ,draw opacity=1 ]   (983,513.17) -- (1011,486.83) ;
\draw    (987,461.17) -- (1011,486.83) ;
\draw    (990,456.17) -- (1011,486.83) ;
\draw    (1011,486.83) -- (1085,463) ;
\draw    (1011,486.83) .. controls (1079,497.33) and (1100,496) .. (1161,489) ;
\draw    (1085,483) -- (1085,463) ;
\draw  [draw opacity=0] (1085,462.83) .. controls (1094.14,459.57) and (1106.77,457.55) .. (1120.73,457.55) .. controls (1135.94,457.55) and (1149.58,459.95) .. (1158.84,463.74) -- (1120.73,475.55) -- cycle ; \draw   (1085,462.83) .. controls (1094.14,459.57) and (1106.77,457.55) .. (1120.73,457.55) .. controls (1135.94,457.55) and (1149.58,459.95) .. (1158.84,463.74) ;  
\draw  [draw opacity=0] (1158.84,463.74) .. controls (1149.72,467.08) and (1137.1,469.19) .. (1123.15,469.3) .. controls (1107.94,469.42) and (1094.28,467.13) .. (1085,463.41) -- (1123.01,451.3) -- cycle ; \draw   (1158.84,463.74) .. controls (1149.72,467.08) and (1137.1,469.19) .. (1123.15,469.3) .. controls (1107.94,469.42) and (1094.28,467.13) .. (1085,463.41) ;  
\draw    (1085.64,514.67) -- (1085.64,507) -- (1085.64,494.67) ;
\draw    (1158,464) -- (1186,437) ;
\draw    (1158,464) -- (1200,474) ;
\draw    (1161,489) -- (1182,476) ;
\draw    (1161,489) -- (1196,506) ;
\draw    (430,809.33) -- (356,809) ;
\draw [color={rgb, 255:red, 0; green, 0; blue, 0 }  ,draw opacity=1 ]   (327,835.33) -- (355,809) ;
\draw    (331,783.33) -- (355,809) ;
\draw    (501,809.67) -- (427,809.33) ;
\draw    (501,809.67) -- (533,787.67) ;
\draw    (501,809.67) -- (532,833.67) ;
\draw    (435.5,771.67) -- (435.55,799) ;
\draw [shift={(435.56,801)}, rotate = 269.89] [color={rgb, 255:red, 0; green, 0; blue, 0 }  ][line width=0.75]    (10.93,-3.29) .. controls (6.95,-1.4) and (3.31,-0.3) .. (0,0) .. controls (3.31,0.3) and (6.95,1.4) .. (10.93,3.29)   ;
\draw [color={rgb, 255:red, 208; green, 2; blue, 27 }  ,draw opacity=1 ]   (436.67,809.33) -- (437,838.67) ;
\draw   (427,726.33) .. controls (427,721.92) and (430.58,718.33) .. (435,718.33) .. controls (439.42,718.33) and (443,721.92) .. (443,726.33) .. controls (443,730.75) and (439.42,734.33) .. (435,734.33) .. controls (430.58,734.33) and (427,730.75) .. (427,726.33) -- cycle ;
\draw    (500,726.33) -- (443,726.33) ;
\draw    (500,726.33) -- (531,696.33) ;
\draw    (500,726.33) -- (532,704.33) ;
\draw    (500,726.33) -- (531,750.33) ;
\draw    (500,726.33) -- (532,760.33) ;
\draw    (370.02,727.47) -- (427.02,726.86) ;
\draw    (370.02,727.47) -- (339.34,757.79) ;
\draw    (370.02,727.47) -- (338.26,749.81) ;
\draw    (370.02,727.47) -- (338.77,703.8) ;
\draw    (370.02,727.47) -- (337.66,693.81) ;
\draw [color={rgb, 255:red, 0; green, 0; blue, 0 }  ,draw opacity=1 ]   (435.33,757.67) -- (435,734.33) ;
\draw    (875,810.33) -- (801,810) ;
\draw [color={rgb, 255:red, 0; green, 0; blue, 0 }  ,draw opacity=1 ]   (772,836.33) -- (800,810) ;
\draw    (776,784.33) -- (800,810) ;
\draw    (946,810.67) -- (872,810.33) ;
\draw    (946,810.67) -- (978,788.67) ;
\draw    (946,810.67) -- (977,834.67) ;
\draw [color={rgb, 255:red, 208; green, 2; blue, 27 }  ,draw opacity=1 ]   (881.67,810.33) -- (882,839.67) ;
\draw    (815.02,728.47) -- (879.33,729) ;
\draw    (815.02,728.47) -- (784.34,758.79) ;
\draw    (815.02,728.47) -- (783.26,750.81) ;
\draw    (815.02,728.47) -- (783.77,704.8) ;
\draw    (815.02,728.47) -- (782.66,694.81) ;
\draw [color={rgb, 255:red, 0; green, 0; blue, 0 }  ,draw opacity=1 ]   (879.33,757) -- (879.33,729) ;
\draw    (880,729.41) .. controls (912.33,754.67) and (917.33,752.67) .. (956,755) ;
\draw  [draw opacity=0] (880,728.83) .. controls (889.14,725.57) and (901.77,723.55) .. (915.73,723.55) .. controls (930.94,723.55) and (944.58,725.95) .. (953.84,729.74) -- (915.73,741.55) -- cycle ; \draw   (880,728.83) .. controls (889.14,725.57) and (901.77,723.55) .. (915.73,723.55) .. controls (930.94,723.55) and (944.58,725.95) .. (953.84,729.74) ;  
\draw  [draw opacity=0] (953.84,729.74) .. controls (944.72,733.08) and (932.1,735.19) .. (918.15,735.3) .. controls (902.94,735.42) and (889.28,733.13) .. (880,729.41) -- (918.01,717.3) -- cycle ; \draw   (953.84,729.74) .. controls (944.72,733.08) and (932.1,735.19) .. (918.15,735.3) .. controls (902.94,735.42) and (889.28,733.13) .. (880,729.41) ;  
\draw    (953,730) -- (981,703) ;
\draw    (953,730) -- (995,740) ;
\draw    (956,755) -- (977,742) ;
\draw    (956,755) -- (991,772) ;
\draw [color={rgb, 255:red, 155; green, 155; blue, 155 }  ,draw opacity=1 ][line width=1.5]    (49.33,379) -- (1251.33,379) ;

\draw (60,164.4) node [anchor=north west][inner sep=0.75pt]  [font = \Large]  {$type\   \rho_a$};
\draw (379,165.4) node [anchor=north west][inner sep=0.75pt]    [font = \Large]  {$type\   \rho_b$};
\draw (697,164.4) node [anchor=north west][inner sep=0.75pt]    [font = \Large]  {$type\   \rho_c$};
\draw (1015,164.73) node [anchor=north west][inner sep=0.75pt]   [font = \Large]  {$type\   \rho_d$};
\draw (62,415.4) node [anchor=north west][inner sep=0.75pt]    [font = \Large]  {$type\   \sigma_1$};
\draw (530.23,414.39) node [anchor=north west][inner sep=0.75pt] [font = \Large]  {$type\   \sigma_2$};
\draw (959.79,417.07) node [anchor=north west][inner sep=0.75pt] [font = \Large]  {$type\   \sigma_3$};
\draw (304.8,663.73) node [anchor=north west][inner sep=0.75pt]  [font = \Large]  {$type\   \sigma_4$};
\draw (744.31,668.44) node [anchor=north west][inner sep=0.75pt] [font = \Large]  {$type\   \sigma_5$};
\draw (121,229.73) node [anchor=north west][inner sep=0.75pt]  [font=\tiny]  {$1$};
\draw (89.5,203.73) node [anchor=north west][inner sep=0.75pt]  [font=\tiny,color={rgb, 255:red, 2; green, 2; blue, 208 }  ,opacity=1 ]  {$1$};
\draw (80.5,234.73) node [anchor=north west][inner sep=0.75pt]  [font=\tiny,color={rgb, 255:red, 2; green, 2; blue, 208 }  ,opacity=1 ]  {$1$};
\draw (236,209.23) node [anchor=north west][inner sep=0.75pt]  [font=\tiny,color={rgb, 255:red, 2; green, 2; blue, 208 }  ,opacity=1 ]  {$1$};
\draw (236,266.23) node [anchor=north west][inner sep=0.75pt]  [font=\tiny,color={rgb, 255:red, 2; green, 2; blue, 208 }  ,opacity=1 ]  {$1$};
\draw (89.5,257.23) node [anchor=north west][inner sep=0.75pt]  [font=\tiny,color={rgb, 255:red, 2; green, 2; blue, 208 }  ,opacity=1 ]  {$3$};
\draw (80.5,224.73) node [anchor=north west][inner sep=0.75pt]  [font=\tiny,color={rgb, 255:red, 2; green, 2; blue, 208 }  ,opacity=1 ]  {$2$};
\draw (95,195.73) node [anchor=north west][inner sep=0.75pt]  [font=\tiny,color={rgb, 255:red, 2; green, 2; blue, 208 }  ,opacity=1 ]  {$2$};
\draw (235.5,195.73) node [anchor=north west][inner sep=0.75pt]  [font=\tiny,color={rgb, 255:red, 2; green, 2; blue, 208 }  ,opacity=1 ]  {$2$};
\draw (236.5,253.23) node [anchor=north west][inner sep=0.75pt]  [font=\tiny,color={rgb, 255:red, 2; green, 2; blue, 208 }  ,opacity=1 ]  {$2$};
\draw (163.5,222.23) node [anchor=north west][inner sep=0.75pt]  [font=\tiny,color={rgb, 255:red, 2; green, 2; blue, 208 }  ,opacity=1 ]  {$3$};
\draw (408.5,273.23) node [anchor=north west][inner sep=0.75pt]  [font=\tiny,color={rgb, 255:red, 2; green, 2; blue, 208 }  ,opacity=1 ]  {$3$};
\draw (480.17,221.9) node [anchor=north west][inner sep=0.75pt]  [font=\tiny,color={rgb, 255:red, 2; green, 2; blue, 208 }  ,opacity=1 ]  {$1$};
\draw (480.5,247.23) node [anchor=north west][inner sep=0.75pt]  [font=\tiny,color={rgb, 255:red, 2; green, 2; blue, 208 }  ,opacity=1 ]  {$1$};
\draw (481,235.65) node [anchor=north west][inner sep=0.75pt]  [font=\tiny,color={rgb, 255:red, 2; green, 2; blue, 208 }  ,opacity=1 ]  {$1$};
\draw (404.17,251.9) node [anchor=north west][inner sep=0.75pt]  [font=\tiny,color={rgb, 255:red, 2; green, 2; blue, 208 }  ,opacity=1 ]  {$1$};
\draw (412.17,219.9) node [anchor=north west][inner sep=0.75pt]  [font=\tiny,color={rgb, 255:red, 2; green, 2; blue, 208 }  ,opacity=1 ]  {$1$};
\draw (551.17,181.9) node [anchor=north west][inner sep=0.75pt]  [font=\tiny,color={rgb, 255:red, 2; green, 2; blue, 208 }  ,opacity=1 ]  {$1$};
\draw (533.17,217.9) node [anchor=north west][inner sep=0.75pt]  [font=\tiny,color={rgb, 255:red, 2; green, 2; blue, 208 }  ,opacity=1 ]  {$1$};
\draw (557.17,215.9) node [anchor=north west][inner sep=0.75pt]  [font=\tiny,color={rgb, 255:red, 2; green, 2; blue, 208 }  ,opacity=1 ]  {$2$};
\draw (553.17,264.9) node [anchor=north west][inner sep=0.75pt]  [font=\tiny,color={rgb, 255:red, 2; green, 2; blue, 208 }  ,opacity=1 ]  {$2$};
\draw (414.17,207.9) node [anchor=north west][inner sep=0.75pt]  [font=\tiny,color={rgb, 255:red, 2; green, 2; blue, 208 }  ,opacity=1 ]  {$2$};
\draw (404.17,239.9) node [anchor=north west][inner sep=0.75pt]  [font=\tiny,color={rgb, 255:red, 2; green, 2; blue, 208 }  ,opacity=1 ]  {$2$};
\draw (721.83,262.23) node [anchor=north west][inner sep=0.75pt]  [font=\tiny,color={rgb, 255:red, 2; green, 2; blue, 208 }  ,opacity=1 ]  {$3$};
\draw (793.83,236.23) node [anchor=north west][inner sep=0.75pt]  [font=\tiny,color={rgb, 255:red, 2; green, 2; blue, 208 }  ,opacity=1 ]  {$1$};
\draw (794.33,224.65) node [anchor=north west][inner sep=0.75pt]  [font=\tiny,color={rgb, 255:red, 2; green, 2; blue, 208 }  ,opacity=1 ]  {$2$};
\draw (867.33,209.07) node [anchor=north west][inner sep=0.75pt]  [font=\tiny,color={rgb, 255:red, 2; green, 2; blue, 208 }  ,opacity=1 ]  {$1$};
\draw (725.5,208.9) node [anchor=north west][inner sep=0.75pt]  [font=\tiny,color={rgb, 255:red, 2; green, 2; blue, 208 }  ,opacity=1 ]  {$1$};
\draw (727.5,196.9) node [anchor=north west][inner sep=0.75pt]  [font=\tiny,color={rgb, 255:red, 2; green, 2; blue, 208 }  ,opacity=1 ]  {$2$};
\draw (867.33,197.65) node [anchor=north west][inner sep=0.75pt]  [font=\tiny,color={rgb, 255:red, 2; green, 2; blue, 208 }  ,opacity=1 ]  {$2$};
\draw (870.33,237.07) node [anchor=north west][inner sep=0.75pt]  [font=\tiny,color={rgb, 255:red, 2; green, 2; blue, 208 }  ,opacity=1 ]  {$1$};
\draw (870.33,225.65) node [anchor=north west][inner sep=0.75pt]  [font=\tiny,color={rgb, 255:red, 2; green, 2; blue, 208 }  ,opacity=1 ]  {$2$};
\draw (867.33,265.07) node [anchor=north west][inner sep=0.75pt]  [font=\tiny,color={rgb, 255:red, 2; green, 2; blue, 208 }  ,opacity=1 ]  {$1$};
\draw (867.33,253.65) node [anchor=north west][inner sep=0.75pt]  [font=\tiny,color={rgb, 255:red, 2; green, 2; blue, 208 }  ,opacity=1 ]  {$2$};
\draw (1045.83,262.23) node [anchor=north west][inner sep=0.75pt]  [font=\tiny,color={rgb, 255:red, 2; green, 2; blue, 208 }  ,opacity=1 ]  {$3$};
\draw (1049.5,208.9) node [anchor=north west][inner sep=0.75pt]  [font=\tiny,color={rgb, 255:red, 2; green, 2; blue, 208 }  ,opacity=1 ]  {$1$};
\draw (1051.5,196.9) node [anchor=north west][inner sep=0.75pt]  [font=\tiny,color={rgb, 255:red, 2; green, 2; blue, 208 }  ,opacity=1 ]  {$2$};
\draw (1154.33,202.73) node [anchor=north west][inner sep=0.75pt]  [font=\tiny]  {$1$};
\draw (1110.17,211.57) node [anchor=north west][inner sep=0.75pt]  [font=\tiny,color={rgb, 255:red, 2; green, 2; blue, 208 }  ,opacity=1 ]  {$2$};
\draw (1188.83,176.9) node [anchor=north west][inner sep=0.75pt]  [font=\tiny,color={rgb, 255:red, 2; green, 2; blue, 208 }  ,opacity=1 ]  {$2$};
\draw (1187.5,200.9) node [anchor=north west][inner sep=0.75pt]  [font=\tiny,color={rgb, 255:red, 2; green, 2; blue, 208 }  ,opacity=1 ]  {$2$};
\draw (1186.5,224.9) node [anchor=north west][inner sep=0.75pt]  [font=\tiny,color={rgb, 255:red, 2; green, 2; blue, 208 }  ,opacity=1 ]  {$2$};
\draw (1112.17,235.9) node [anchor=north west][inner sep=0.75pt]  [font=\tiny,color={rgb, 255:red, 2; green, 2; blue, 208 }  ,opacity=1 ]  {$1$};
\draw (1184.33,236.4) node [anchor=north west][inner sep=0.75pt]  [font=\tiny,color={rgb, 255:red, 2; green, 2; blue, 208 }  ,opacity=1 ]  {$1$};
\draw (1185.11,253.73) node [anchor=north west][inner sep=0.75pt]  [font=\tiny,color={rgb, 255:red, 2; green, 2; blue, 208 }  ,opacity=1 ]  {$1$};
\draw (1184.78,277.73) node [anchor=north west][inner sep=0.75pt]  [font=\tiny,color={rgb, 255:red, 2; green, 2; blue, 208 }  ,opacity=1 ]  {$1$};
\draw (184.5,484.23) node [anchor=north west][inner sep=0.75pt]  [font=\tiny,color={rgb, 255:red, 2; green, 2; blue, 208 }  ,opacity=1 ]  {$1$};
\draw (185,472.65) node [anchor=north west][inner sep=0.75pt]  [font=\tiny,color={rgb, 255:red, 2; green, 2; blue, 208 }  ,opacity=1 ]  {$2$};
\draw (116.17,456.9) node [anchor=north west][inner sep=0.75pt]  [font=\tiny,color={rgb, 255:red, 2; green, 2; blue, 208 }  ,opacity=1 ]  {$1$};
\draw (118.17,444.9) node [anchor=north west][inner sep=0.75pt]  [font=\tiny,color={rgb, 255:red, 2; green, 2; blue, 208 }  ,opacity=1 ]  {$2$};
\draw (111.5,511.23) node [anchor=north west][inner sep=0.75pt]  [font=\tiny,color={rgb, 255:red, 2; green, 2; blue, 208 }  ,opacity=1 ]  {$3$};
\draw (328,462.57) node [anchor=north west][inner sep=0.75pt]  [font=\tiny,color={rgb, 255:red, 2; green, 2; blue, 208 }  ,opacity=1 ]  {$1$};
\draw (328,519.57) node [anchor=north west][inner sep=0.75pt]  [font=\tiny,color={rgb, 255:red, 2; green, 2; blue, 208 }  ,opacity=1 ]  {$1$};
\draw (327.5,449.07) node [anchor=north west][inner sep=0.75pt]  [font=\tiny,color={rgb, 255:red, 2; green, 2; blue, 208 }  ,opacity=1 ]  {$2$};
\draw (328.5,506.57) node [anchor=north west][inner sep=0.75pt]  [font=\tiny,color={rgb, 255:red, 2; green, 2; blue, 208 }  ,opacity=1 ]  {$2$};
\draw (255.5,475.57) node [anchor=north west][inner sep=0.75pt]  [font=\tiny,color={rgb, 255:red, 2; green, 2; blue, 208 }  ,opacity=1 ]  {$3$};
\draw (209,523.32) node [anchor=north west][inner sep=0.75pt]  [font=\tiny,color={rgb, 255:red, 2; green, 2; blue, 208 }  ,opacity=1 ]  {$2$};
\draw (235.22,521.07) node [anchor=north west][inner sep=0.75pt]  [font=\tiny,color={rgb, 255:red, 2; green, 2; blue, 208 }  ,opacity=1 ]  {$1$};
\draw (607.5,486.23) node [anchor=north west][inner sep=0.75pt]  [font=\tiny,color={rgb, 255:red, 2; green, 2; blue, 208 }  ,opacity=1 ]  {$1$};
\draw (608.67,464.65) node [anchor=north west][inner sep=0.75pt]  [font=\tiny,color={rgb, 255:red, 2; green, 2; blue, 208 }  ,opacity=1 ]  {$2$};
\draw (539.17,458.9) node [anchor=north west][inner sep=0.75pt]  [font=\tiny,color={rgb, 255:red, 2; green, 2; blue, 208 }  ,opacity=1 ]  {$1$};
\draw (541.17,446.9) node [anchor=north west][inner sep=0.75pt]  [font=\tiny,color={rgb, 255:red, 2; green, 2; blue, 208 }  ,opacity=1 ]  {$2$};
\draw (533.5,512.23) node [anchor=north west][inner sep=0.75pt]  [font=\tiny,color={rgb, 255:red, 2; green, 2; blue, 208 }  ,opacity=1 ]  {$3$};
\draw (756,497.65) node [anchor=north west][inner sep=0.75pt]  [font=\tiny,color={rgb, 255:red, 2; green, 2; blue, 208 }  ,opacity=1 ]  {$2$};
\draw (740,472.4) node [anchor=north west][inner sep=0.75pt]  [font=\tiny,color={rgb, 255:red, 2; green, 2; blue, 208 }  ,opacity=1 ]  {$2$};
\draw (747.5,425.23) node [anchor=north west][inner sep=0.75pt]  [font=\tiny,color={rgb, 255:red, 2; green, 2; blue, 208 }  ,opacity=1 ]  {$1$};
\draw (761.5,466.23) node [anchor=north west][inner sep=0.75pt]  [font=\tiny,color={rgb, 255:red, 2; green, 2; blue, 208 }  ,opacity=1 ]  {$1$};
\draw (683.5,446.23) node [anchor=north west][inner sep=0.75pt]  [font=\tiny,color={rgb, 255:red, 2; green, 2; blue, 208 }  ,opacity=1 ]  {$1$};
\draw (683.5,463.9) node [anchor=north west][inner sep=0.75pt]  [font=\tiny,color={rgb, 255:red, 2; green, 2; blue, 208 }  ,opacity=1 ]  {$1$};
\draw (683.5,481.23) node [anchor=north west][inner sep=0.75pt]  [font=\tiny,color={rgb, 255:red, 2; green, 2; blue, 208 }  ,opacity=1 ]  {$1$};
\draw (639.67,481.65) node [anchor=north west][inner sep=0.75pt]  [font=\tiny,color={rgb, 255:red, 2; green, 2; blue, 208 }  ,opacity=1 ]  {$2$};
\draw (638.5,517.07) node [anchor=north west][inner sep=0.75pt]  [font=\tiny,color={rgb, 255:red, 2; green, 2; blue, 208 }  ,opacity=1 ]  {$1$};
\draw (975.17,457.9) node [anchor=north west][inner sep=0.75pt]  [font=\tiny,color={rgb, 255:red, 2; green, 2; blue, 208 }  ,opacity=1 ]  {$1$};
\draw (977.17,445.9) node [anchor=north west][inner sep=0.75pt]  [font=\tiny,color={rgb, 255:red, 2; green, 2; blue, 208 }  ,opacity=1 ]  {$2$};
\draw (969.5,511.23) node [anchor=north west][inner sep=0.75pt]  [font=\tiny,color={rgb, 255:red, 2; green, 2; blue, 208 }  ,opacity=1 ]  {$3$};
\draw (1075.67,480.65) node [anchor=north west][inner sep=0.75pt]  [font=\tiny,color={rgb, 255:red, 2; green, 2; blue, 208 }  ,opacity=1 ]  {$2$};
\draw (1043.67,465.65) node [anchor=north west][inner sep=0.75pt]  [font=\tiny,color={rgb, 255:red, 2; green, 2; blue, 208 }  ,opacity=1 ]  {$2$};
\draw (1045.17,481.9) node [anchor=north west][inner sep=0.75pt]  [font=\tiny,color={rgb, 255:red, 2; green, 2; blue, 208 }  ,opacity=1 ]  {$1$};
\draw (1077.17,509.9) node [anchor=north west][inner sep=0.75pt]  [font=\tiny,color={rgb, 255:red, 2; green, 2; blue, 208 }  ,opacity=1 ]  {$1$};
\draw (1114.17,444.9) node [anchor=north west][inner sep=0.75pt]  [font=\tiny,color={rgb, 255:red, 2; green, 2; blue, 208 }  ,opacity=1 ]  {$1$};
\draw (1114.17,459.9) node [anchor=north west][inner sep=0.75pt]  [font=\tiny,color={rgb, 255:red, 2; green, 2; blue, 208 }  ,opacity=1 ]  {$1$};
\draw (1116.17,483.9) node [anchor=north west][inner sep=0.75pt]  [font=\tiny,color={rgb, 255:red, 2; green, 2; blue, 208 }  ,opacity=1 ]  {$1$};
\draw (1187.17,474.9) node [anchor=north west][inner sep=0.75pt]  [font=\tiny,color={rgb, 255:red, 2; green, 2; blue, 208 }  ,opacity=1 ]  {$1$};
\draw (1201.17,504.9) node [anchor=north west][inner sep=0.75pt]  [font=\tiny,color={rgb, 255:red, 2; green, 2; blue, 208 }  ,opacity=1 ]  {$1$};
\draw (1190.67,430.65) node [anchor=north west][inner sep=0.75pt]  [font=\tiny,color={rgb, 255:red, 2; green, 2; blue, 208 }  ,opacity=1 ]  {$2$};
\draw (1204.67,469.65) node [anchor=north west][inner sep=0.75pt]  [font=\tiny,color={rgb, 255:red, 2; green, 2; blue, 208 }  ,opacity=1 ]  {$2$};
\draw (432,723.73) node [anchor=north west][inner sep=0.75pt]  [font=\tiny]  {$1$};
\draw (532.17,702.23) node [anchor=north west][inner sep=0.75pt]  [font=\tiny,color={rgb, 255:red, 2; green, 2; blue, 208 }  ,opacity=1 ]  {$1$};
\draw (535,759.23) node [anchor=north west][inner sep=0.75pt]  [font=\tiny,color={rgb, 255:red, 2; green, 2; blue, 208 }  ,opacity=1 ]  {$1$};
\draw (529.87,688.73) node [anchor=north west][inner sep=0.75pt]  [font=\tiny,color={rgb, 255:red, 2; green, 2; blue, 208 }  ,opacity=1 ]  {$2$};
\draw (538.53,746.23) node [anchor=north west][inner sep=0.75pt]  [font=\tiny,color={rgb, 255:red, 2; green, 2; blue, 208 }  ,opacity=1 ]  {$2$};
\draw (462.5,715.23) node [anchor=north west][inner sep=0.75pt]  [font=\tiny,color={rgb, 255:red, 2; green, 2; blue, 208 }  ,opacity=1 ]  {$3$};
\draw (393.5,716.23) node [anchor=north west][inner sep=0.75pt]  [font=\tiny,color={rgb, 255:red, 2; green, 2; blue, 208 }  ,opacity=1 ]  {$3$};
\draw (326.17,703.23) node [anchor=north west][inner sep=0.75pt]  [font=\tiny,color={rgb, 255:red, 2; green, 2; blue, 208 }  ,opacity=1 ]  {$1$};
\draw (333,757.23) node [anchor=north west][inner sep=0.75pt]  [font=\tiny,color={rgb, 255:red, 2; green, 2; blue, 208 }  ,opacity=1 ]  {$1$};
\draw (327.87,688.73) node [anchor=north west][inner sep=0.75pt]  [font=\tiny,color={rgb, 255:red, 2; green, 2; blue, 208 }  ,opacity=1 ]  {$2$};
\draw (327.53,745.23) node [anchor=north west][inner sep=0.75pt]  [font=\tiny,color={rgb, 255:red, 2; green, 2; blue, 208 }  ,opacity=1 ]  {$2$};
\draw (422.5,751.23) node [anchor=north west][inner sep=0.75pt]  [font=\tiny,color={rgb, 255:red, 2; green, 2; blue, 208 }  ,opacity=1 ]  {$3$};
\draw (838.5,717.23) node [anchor=north west][inner sep=0.75pt]  [font=\tiny,color={rgb, 255:red, 2; green, 2; blue, 208 }  ,opacity=1 ]  {$3$};
\draw (771.17,704.23) node [anchor=north west][inner sep=0.75pt]  [font=\tiny,color={rgb, 255:red, 2; green, 2; blue, 208 }  ,opacity=1 ]  {$1$};
\draw (778,758.23) node [anchor=north west][inner sep=0.75pt]  [font=\tiny,color={rgb, 255:red, 2; green, 2; blue, 208 }  ,opacity=1 ]  {$1$};
\draw (772.87,689.73) node [anchor=north west][inner sep=0.75pt]  [font=\tiny,color={rgb, 255:red, 2; green, 2; blue, 208 }  ,opacity=1 ]  {$2$};
\draw (772.53,746.23) node [anchor=north west][inner sep=0.75pt]  [font=\tiny,color={rgb, 255:red, 2; green, 2; blue, 208 }  ,opacity=1 ]  {$2$};
\draw (867.5,752.23) node [anchor=north west][inner sep=0.75pt]  [font=\tiny,color={rgb, 255:red, 2; green, 2; blue, 208 }  ,opacity=1 ]  {$3$};
\draw (909.17,710.9) node [anchor=north west][inner sep=0.75pt]  [font=\tiny,color={rgb, 255:red, 2; green, 2; blue, 208 }  ,opacity=1 ]  {$1$};
\draw (909.17,725.9) node [anchor=north west][inner sep=0.75pt]  [font=\tiny,color={rgb, 255:red, 2; green, 2; blue, 208 }  ,opacity=1 ]  {$1$};
\draw (911.17,749.9) node [anchor=north west][inner sep=0.75pt]  [font=\tiny,color={rgb, 255:red, 2; green, 2; blue, 208 }  ,opacity=1 ]  {$1$};
\draw (982.17,740.9) node [anchor=north west][inner sep=0.75pt]  [font=\tiny,color={rgb, 255:red, 2; green, 2; blue, 208 }  ,opacity=1 ]  {$1$};
\draw (996.17,770.9) node [anchor=north west][inner sep=0.75pt]  [font=\tiny,color={rgb, 255:red, 2; green, 2; blue, 208 }  ,opacity=1 ]  {$1$};
\draw (985.67,696.65) node [anchor=north west][inner sep=0.75pt]  [font=\tiny,color={rgb, 255:red, 2; green, 2; blue, 208 }  ,opacity=1 ]  {$2$};
\draw (999.67,735.65) node [anchor=north west][inner sep=0.75pt]  [font=\tiny,color={rgb, 255:red, 2; green, 2; blue, 208 }  ,opacity=1 ]  {$2$};

\end{tikzpicture}

%% file: brtrop.tikz.tex
\tikzset{every picture/.style={line width=0.75pt}} 

\begin{tikzpicture}[x=0.75pt,y=0.75pt,yscale=-1,xscale=1]

\draw   (675.01,233.02) -- (624.27,382.58) -- (459.87,382.66) -- (409,233.16) -- (541.97,140.67) -- cycle ;
\draw    (541.29,223.11) -- (511.08,316.34) ;
\draw    (541.29,223.11) -- (574.15,315.52) ;
\draw    (574.15,315.52) -- (490.77,259.23) ;
\draw    (592.82,257.9) -- (490.77,259.23) ;
\draw    (592.82,257.9) -- (511.08,316.34) ;
\draw    (541.95,140.67) -- (541.29,223.11) ;
\draw    (675,233.02) -- (592.82,257.9) ;
\draw    (574.15,315.52) -- (624.28,382.58) ;
\draw    (511.08,316.34) -- (459.89,382.67) ;
\draw    (409,233.16) -- (490.77,259.23) ;
\draw  [color={rgb, 255:red, 208; green, 2; blue, 27 }  ,draw opacity=1 ][fill={rgb, 255:red, 208; green, 2; blue, 27 }  ,fill opacity=1 ] (536.13,228.27) .. controls (536.13,225.42) and (538.44,223.11) .. (541.29,223.11) .. controls (544.15,223.11) and (546.46,225.42) .. (546.46,228.27) .. controls (546.46,231.13) and (544.15,233.44) .. (541.29,233.44) .. controls (538.44,233.44) and (536.13,231.13) .. (536.13,228.27) -- cycle ;
\draw  [color={rgb, 255:red, 208; green, 2; blue, 27 }  ,draw opacity=1 ][fill={rgb, 255:red, 208; green, 2; blue, 27 }  ,fill opacity=1 ] (587.66,257.9) .. controls (587.66,255.05) and (589.97,252.73) .. (592.82,252.73) .. controls (595.68,252.73) and (597.99,255.05) .. (597.99,257.9) .. controls (597.99,260.75) and (595.68,263.07) .. (592.82,263.07) .. controls (589.97,263.07) and (587.66,260.75) .. (587.66,257.9) -- cycle ;
\draw  [color={rgb, 255:red, 208; green, 2; blue, 27 }  ,draw opacity=1 ][fill={rgb, 255:red, 208; green, 2; blue, 27 }  ,fill opacity=1 ] (669.85,233.02) .. controls (669.85,230.17) and (672.16,227.86) .. (675.01,227.86) .. controls (677.87,227.86) and (680.18,230.17) .. (680.18,233.02) .. controls (680.18,235.88) and (677.87,238.19) .. (675.01,238.19) .. controls (672.16,238.19) and (669.85,235.88) .. (669.85,233.02) -- cycle ;
\draw  [color={rgb, 255:red, 208; green, 2; blue, 27 }  ,draw opacity=1 ][fill={rgb, 255:red, 208; green, 2; blue, 27 }  ,fill opacity=1 ] (568.98,315.52) .. controls (568.98,312.67) and (571.3,310.36) .. (574.15,310.36) .. controls (577,310.36) and (579.32,312.67) .. (579.32,315.52) .. controls (579.32,318.38) and (577,320.69) .. (574.15,320.69) .. controls (571.3,320.69) and (568.98,318.38) .. (568.98,315.52) -- cycle ;
\draw  [color={rgb, 255:red, 208; green, 2; blue, 27 }  ,draw opacity=1 ][fill={rgb, 255:red, 208; green, 2; blue, 27 }  ,fill opacity=1 ] (454.72,382.67) .. controls (454.72,379.81) and (457.03,377.5) .. (459.89,377.5) .. controls (462.74,377.5) and (465.05,379.81) .. (465.05,382.67) .. controls (465.05,385.52) and (462.74,387.83) .. (459.89,387.83) .. controls (457.03,387.83) and (454.72,385.52) .. (454.72,382.67) -- cycle ;
\draw  [color={rgb, 255:red, 208; green, 2; blue, 27 }  ,draw opacity=1 ][fill={rgb, 255:red, 208; green, 2; blue, 27 }  ,fill opacity=1 ] (403.83,233.16) .. controls (403.83,230.3) and (406.15,227.99) .. (409,227.99) .. controls (411.86,227.99) and (414.17,230.3) .. (414.17,233.16) .. controls (414.17,236.01) and (411.86,238.32) .. (409,238.32) .. controls (406.15,238.32) and (403.83,236.01) .. (403.83,233.16) -- cycle ;
\draw  [color={rgb, 255:red, 0; green, 0; blue, 0 }  ,draw opacity=1 ][fill={rgb, 255:red, 0; green, 0; blue, 0 }  ,fill opacity=1 ] (536.8,140.67) .. controls (536.8,137.82) and (539.11,135.51) .. (541.97,135.51) .. controls (544.82,135.51) and (547.13,137.82) .. (547.13,140.67) .. controls (547.13,143.53) and (544.82,145.84) .. (541.97,145.84) .. controls (539.11,145.84) and (536.8,143.53) .. (536.8,140.67) -- cycle ;
\draw  [color={rgb, 255:red, 0; green, 0; blue, 0 }  ,draw opacity=1 ][fill={rgb, 255:red, 0; green, 0; blue, 0 }  ,fill opacity=1 ] (485.61,259.23) .. controls (485.61,256.37) and (487.92,254.06) .. (490.77,254.06) .. controls (493.63,254.06) and (495.94,256.37) .. (495.94,259.23) .. controls (495.94,262.08) and (493.63,264.39) .. (490.77,264.39) .. controls (487.92,264.39) and (485.61,262.08) .. (485.61,259.23) -- cycle ;
\draw  [color={rgb, 255:red, 0; green, 0; blue, 0 }  ,draw opacity=1 ][fill={rgb, 255:red, 0; green, 0; blue, 0 }  ,fill opacity=1 ] (505.91,316.34) .. controls (505.91,313.49) and (508.23,311.18) .. (511.08,311.18) .. controls (513.93,311.18) and (516.25,313.49) .. (516.25,316.34) .. controls (516.25,319.2) and (513.93,321.51) .. (511.08,321.51) .. controls (508.23,321.51) and (505.91,319.2) .. (505.91,316.34) -- cycle ;
\draw  [color={rgb, 255:red, 0; green, 0; blue, 0 }  ,draw opacity=1 ][fill={rgb, 255:red, 0; green, 0; blue, 0 }  ,fill opacity=1 ] (619.11,382.58) .. controls (619.11,379.73) and (621.42,377.42) .. (624.27,377.42) .. controls (627.13,377.42) and (629.44,379.73) .. (629.44,382.58) .. controls (629.44,385.44) and (627.13,387.75) .. (624.27,387.75) .. controls (621.42,387.75) and (619.11,385.44) .. (619.11,382.58) -- cycle ;
\draw [color={rgb, 255:red, 208; green, 2; blue, 27 }  ,draw opacity=1 ][line width=1.5]    (52.55,237.49) -- (200.32,280.89) ;
\draw [color={rgb, 255:red, 126; green, 211; blue, 33 }  ,draw opacity=1 ][line width=1.5]    (75.41,196.16) -- (123.41,455) ;
\draw [color={rgb, 255:red, 126; green, 211; blue, 33 }  ,draw opacity=1 ][line width=1.5]    (52.55,237.49) -- (141.96,413.67) ;
\draw [color={rgb, 255:red, 74; green, 144; blue, 226 }  ,draw opacity=1 ][line width=1.5]    (52.54,237.49) -- (123.41,455) ;
\draw [color={rgb, 255:red, 208; green, 2; blue, 27 }  ,draw opacity=1 ][line width=1.5]    (52.55,237.49) -- (185.51,145.01) ;
\draw [color={rgb, 255:red, 245; green, 166; blue, 35 }  ,draw opacity=1 ][line width=1.5]    (75.41,196.16) -- (185.51,145.01) ;
\draw [color={rgb, 255:red, 189; green, 16; blue, 224 }  ,draw opacity=1 ][line width=1.5]    (75.41,196.16) -- (201.71,101.67) ;
\draw [color={rgb, 255:red, 189; green, 16; blue, 224 }  ,draw opacity=1 ][line width=1.5]    (75.41,196.16) -- (211.04,239.17) ;
\draw [color={rgb, 255:red, 245; green, 166; blue, 35 }  ,draw opacity=1 ][line width=1.5]    (75.41,196.16) -- (200.32,280.89) ;
\draw  [color={rgb, 255:red, 2; green, 2; blue, 208 }  ,draw opacity=1 ][fill={rgb, 255:red, 2; green, 2; blue, 208 }  ,fill opacity=1 ] (47.38,237.49) .. controls (47.38,234.64) and (49.69,232.32) .. (52.55,232.32) .. controls (55.4,232.32) and (57.71,234.64) .. (57.71,237.49) .. controls (57.71,240.34) and (55.4,242.66) .. (52.55,242.66) .. controls (49.69,242.66) and (47.38,240.34) .. (47.38,237.49) -- cycle ;
\draw  [color={rgb, 255:red, 155; green, 155; blue, 155 }  ,draw opacity=1 ][fill={rgb, 255:red, 155; green, 155; blue, 155 }  ,fill opacity=1 ] (70.24,196.16) .. controls (70.24,193.3) and (72.56,190.99) .. (75.41,190.99) .. controls (78.26,190.99) and (80.58,193.3) .. (80.58,196.16) .. controls (80.58,199.01) and (78.26,201.32) .. (75.41,201.32) .. controls (72.56,201.32) and (70.24,199.01) .. (70.24,196.16) -- cycle ;
\draw    (283.57,237.29) -- (366.57,237.29) ;
\draw [shift={(368.57,237.29)}, rotate = 180] [color={rgb, 255:red, 0; green, 0; blue, 0 }  ][line width=0.75]    (10.93,-3.29) .. controls (6.95,-1.4) and (3.31,-0.3) .. (0,0) .. controls (3.31,0.3) and (6.95,1.4) .. (10.93,3.29)   ;
\draw [color={rgb, 255:red, 189; green, 16; blue, 224 }  ,draw opacity=1 ][line width=1.5]    (201.57,120.29) -- (201.71,101.67) ;
\draw [color={rgb, 255:red, 189; green, 16; blue, 224 }  ,draw opacity=1 ][line width=1.5]    (249.57,123.29) -- (201.71,101.67) ;
\draw  [color={rgb, 255:red, 248; green, 231; blue, 28 }  ,draw opacity=1 ][fill={rgb, 255:red, 248; green, 231; blue, 28 }  ,fill opacity=1 ] (196.54,101.67) .. controls (196.54,98.82) and (198.86,96.51) .. (201.71,96.51) .. controls (204.56,96.51) and (206.88,98.82) .. (206.88,101.67) .. controls (206.88,104.53) and (204.56,106.84) .. (201.71,106.84) .. controls (198.86,106.84) and (196.54,104.53) .. (196.54,101.67) -- cycle ;
\draw [color={rgb, 255:red, 208; green, 2; blue, 27 }  ,draw opacity=1 ][line width=1.5]    (182.57,170.29) -- (185.51,145.01) ;
\draw [color={rgb, 255:red, 208; green, 2; blue, 27 }  ,draw opacity=1 ][line width=1.5]    (238.57,177.29) -- (185.51,145.01) ;
\draw [color={rgb, 255:red, 245; green, 166; blue, 35 }  ,draw opacity=1 ][line width=1.5]    (185.51,145.01) -- (245.57,167.29) ;
\draw [color={rgb, 255:red, 245; green, 166; blue, 35 }  ,draw opacity=1 ][line width=1.5]    (185.51,145.01) -- (193.57,170.29) ;
\draw  [color={rgb, 255:red, 144; green, 19; blue, 254 }  ,draw opacity=1 ][fill={rgb, 255:red, 144; green, 19; blue, 254 }  ,fill opacity=1 ] (181.34,145.01) .. controls (181.34,142.15) and (183.66,139.84) .. (186.51,139.84) .. controls (189.36,139.84) and (191.68,142.15) .. (191.68,145.01) .. controls (191.68,147.86) and (189.36,150.17) .. (186.51,150.17) .. controls (183.66,150.17) and (181.34,147.86) .. (181.34,145.01) -- cycle ;
\draw [color={rgb, 255:red, 189; green, 16; blue, 224 }  ,draw opacity=1 ][line width=1.5]    (246.57,239.29) -- (211.04,239.17) ;
\draw [color={rgb, 255:red, 189; green, 16; blue, 224 }  ,draw opacity=1 ][line width=1.5]    (243.57,253.29) -- (211.04,239.17) ;
\draw  [color={rgb, 255:red, 248; green, 231; blue, 28 }  ,draw opacity=1 ][fill={rgb, 255:red, 248; green, 231; blue, 28 }  ,fill opacity=1 ] (205.88,239.17) .. controls (205.88,236.32) and (208.19,234.01) .. (211.04,234.01) .. controls (213.9,234.01) and (216.21,236.32) .. (216.21,239.17) .. controls (216.21,242.03) and (213.9,244.34) .. (211.04,244.34) .. controls (208.19,244.34) and (205.88,242.03) .. (205.88,239.17) -- cycle ;
\draw [color={rgb, 255:red, 208; green, 2; blue, 27 }  ,draw opacity=1 ][line width=1.5]    (200.32,280.89) -- (237.57,281.29) ;
\draw [color={rgb, 255:red, 208; green, 2; blue, 27 }  ,draw opacity=1 ][line width=1.5]    (200.32,280.89) -- (230.57,304.29) ;
\draw [color={rgb, 255:red, 245; green, 166; blue, 35 }  ,draw opacity=1 ][line width=1.5]    (200.32,280.89) -- (220.57,312.29) ;
\draw [color={rgb, 255:red, 245; green, 166; blue, 35 }  ,draw opacity=1 ][line width=1.5]    (200.32,280.89) -- (241.57,291.29) ;
\draw  [color={rgb, 255:red, 144; green, 19; blue, 254 }  ,draw opacity=1 ][fill={rgb, 255:red, 144; green, 19; blue, 254 }  ,fill opacity=1 ] (195.15,280.89) .. controls (195.15,278.04) and (197.46,275.73) .. (200.32,275.73) .. controls (203.17,275.73) and (205.48,278.04) .. (205.48,280.89) .. controls (205.48,283.75) and (203.17,286.06) .. (200.32,286.06) .. controls (197.46,286.06) and (195.15,283.75) .. (195.15,280.89) -- cycle ;
\draw [color={rgb, 255:red, 189; green, 16; blue, 224 }  ,draw opacity=1 ][line width=1.5]    (141.96,413.67) -- (244.38,356.34) ;
\draw [color={rgb, 255:red, 189; green, 16; blue, 224 }  ,draw opacity=1 ][line width=1.5]    (141.96,413.67) -- (180.57,414.43) ;
\draw [color={rgb, 255:red, 245; green, 166; blue, 35 }  ,draw opacity=1 ][line width=1.5]    (141.96,413.67) -- (227.65,392.89) ;
\draw [color={rgb, 255:red, 245; green, 166; blue, 35 }  ,draw opacity=1 ][line width=1.5]    (141.96,413.67) -- (171.57,420.43) ;
\draw  [color={rgb, 255:red, 155; green, 155; blue, 155 }  ,draw opacity=1 ][fill={rgb, 255:red, 155; green, 155; blue, 155 }  ,fill opacity=1 ] (136.79,413.67) .. controls (136.79,410.81) and (139.11,408.5) .. (141.96,408.5) .. controls (144.81,408.5) and (147.13,410.81) .. (147.13,413.67) .. controls (147.13,416.52) and (144.81,418.83) .. (141.96,418.83) .. controls (139.11,418.83) and (136.79,416.52) .. (136.79,413.67) -- cycle ;
\draw [color={rgb, 255:red, 208; green, 2; blue, 27 }  ,draw opacity=1 ][line width=1.5]    (123.41,455) -- (192.57,455.43) ;
\draw [color={rgb, 255:red, 208; green, 2; blue, 27 }  ,draw opacity=1 ][line width=1.5]    (123.41,455) -- (227.65,392.89) ;
\draw  [color={rgb, 255:red, 2; green, 2; blue, 208 }  ,draw opacity=1 ][fill={rgb, 255:red, 2; green, 2; blue, 208 }  ,fill opacity=1 ] (118.25,455) .. controls (118.25,452.14) and (120.56,449.83) .. (123.41,449.83) .. controls (126.27,449.83) and (128.58,452.14) .. (128.58,455) .. controls (128.58,457.85) and (126.27,460.16) .. (123.41,460.16) .. controls (120.56,460.16) and (118.25,457.85) .. (118.25,455) -- cycle ;
\draw [color={rgb, 255:red, 189; green, 16; blue, 224 }  ,draw opacity=1 ][line width=1.5]    (258.57,329.43) -- (244.38,356.34) ;
\draw [color={rgb, 255:red, 189; green, 16; blue, 224 }  ,draw opacity=1 ][line width=1.5]    (276.57,342.43) -- (244.38,356.34) ;
\draw  [color={rgb, 255:red, 248; green, 231; blue, 28 }  ,draw opacity=1 ][fill={rgb, 255:red, 248; green, 231; blue, 28 }  ,fill opacity=1 ] (239.21,356.34) .. controls (239.21,353.49) and (241.52,351.17) .. (244.38,351.17) .. controls (247.23,351.17) and (249.54,353.49) .. (249.54,356.34) .. controls (249.54,359.19) and (247.23,361.51) .. (244.38,361.51) .. controls (241.52,361.51) and (239.21,359.19) .. (239.21,356.34) -- cycle ;
\draw [color={rgb, 255:red, 208; green, 2; blue, 27 }  ,draw opacity=1 ][line width=1.5]    (227.65,392.89) -- (244.57,371.43) ;
\draw [color={rgb, 255:red, 208; green, 2; blue, 27 }  ,draw opacity=1 ][line width=1.5]    (227.65,392.89) -- (260.57,381.43) ;
\draw [color={rgb, 255:red, 245; green, 166; blue, 35 }  ,draw opacity=1 ][line width=1.5]    (227.65,392.89) -- (251.57,372.43) ;
\draw [color={rgb, 255:red, 245; green, 166; blue, 35 }  ,draw opacity=1 ][line width=1.5]    (227.65,392.89) -- (262.57,388.43) ;
\draw  [color={rgb, 255:red, 144; green, 19; blue, 254 }  ,draw opacity=1 ][fill={rgb, 255:red, 144; green, 19; blue, 254 }  ,fill opacity=1 ] (222.48,392.89) .. controls (222.48,390.04) and (224.8,387.73) .. (227.65,387.73) .. controls (230.5,387.73) and (232.82,390.04) .. (232.82,392.89) .. controls (232.82,395.75) and (230.5,398.06) .. (227.65,398.06) .. controls (224.8,398.06) and (222.48,395.75) .. (222.48,392.89) -- cycle ;
\draw  [color={rgb, 255:red, 155; green, 155; blue, 155 }  ,draw opacity=0.43 ][fill={rgb, 255:red, 155; green, 155; blue, 155 }  ,fill opacity=0.21 ][dash pattern={on 0.84pt off 2.51pt}] (511.57,365) -- (510.57,403) -- (448.57,403) -- (393.57,225) -- (540.57,121) -- (579.57,143) -- (566.57,170) -- (527.57,170) -- (454.57,223) -- (511.57,242) -- (501.57,279) -- (436.57,257) -- (466.57,343) -- (506.57,289) -- (536.57,309) -- (494.57,365) -- cycle ;

\draw (551.33,111.4) node [anchor=north west][inner sep=0.75pt]    {$\hat{\rho}_{1,i}$};
\draw (498.67,236.07) node [anchor=north west][inner sep=0.75pt]    {$\hat{\rho}_{1,j}$};
\draw (384.67,205.4) node [anchor=north west][inner sep=0.75pt]    {$\hat{\rho}_{k,l}$};
\draw (508,332.73) node [anchor=north west][inner sep=0.75pt]    {$\hat{\rho}_{1,k}$};
\draw (684,198.07) node [anchor=north west][inner sep=0.75pt]    {$\hat{\rho}_{j,k}$};
\draw (642,388.07) node [anchor=north west][inner sep=0.75pt]    {$\hat{\rho}_{1,l}$};
\draw (430.67,390.73) node [anchor=north west][inner sep=0.75pt]    {$\hat{\rho}_{i,j}$};
\draw (546.67,199.4) node [anchor=north west][inner sep=0.75pt]    {$\hat{\rho}_{j,l}$};
\draw (572,234.07) node [anchor=north west][inner sep=0.75pt]    {$\hat{\rho}_{i,l}$};
\draw (564,334.07) node [anchor=north west][inner sep=0.75pt]    {$\hat{\rho}_{i,k}$};
\draw (43.74,201.73) node [anchor=north west][inner sep=0.75pt]    {$\rho_{b|k,l}$};
\draw (209.21,261.4) node [anchor=north west][inner sep=0.75pt]    {$\rho_{c|1,j}$};
\draw (87.58,446.4) node [anchor=north west][inner sep=0.75pt]    {$\rho_{a|i,j}$};
\draw (138.13,381.07) node [anchor=north west][inner sep=0.75pt]    {$\rho_{b|i,j}$};
\draw (25.74,252.73) node [anchor=north west][inner sep=0.75pt]    {$\rho_{a|k,l}$};
\draw (220.21,86.4) node [anchor=north west][inner sep=0.75pt]    {$\rho_{d|1,i}$};
\draw (204.21,132.4) node [anchor=north west][inner sep=0.75pt]    {$\rho_{c|1,i}$};
\draw (219.21,210.57) node [anchor=north west][inner sep=0.75pt]    {$\rho_{d|1,j}$};
\draw (229.65,401.46) node [anchor=north west][inner sep=0.75pt]    {$\rho_{c|1,k}$};
\draw (208.21,337.4) node [anchor=north west][inner sep=0.75pt]    {$\rho_{d|1,k}$};
\draw (-80,76.4) node [anchor=north west][inner sep=0.75pt]  [font=\LARGE]  {$\overline{Adm}^\trop_{1\to 0}((3),(2,1)^4)$};
\draw (408,75.4) node [anchor=north west][inner sep=0.75pt]  [font=\LARGE]  {$\overline{\calM}_{0,5}^\trop$};
\draw (312,213.26) node [anchor=north west][inner sep=0.75pt] [font=\Large]    {$\mathtt{br}^\trop$};
\draw (402,184.4) node [anchor=north west][inner sep=0.75pt]  [font=\normalsize,color={rgb, 255:red, 2; green, 2; blue, 208 }  ,opacity=1 ]  {$\frac{1}{3}$};
\draw (470,227.4) node [anchor=north west][inner sep=0.75pt]  [font=\normalsize,color={rgb, 255:red, 2; green, 2; blue, 208 }  ,opacity=1 ]  {$\frac{1}{3}$};
\draw (510,120.4) node [anchor=north west][inner sep=0.75pt]  [font=\normalsize,color={rgb, 255:red, 2; green, 2; blue, 208 }  ,opacity=1 ]  {$\frac{1}{3}$};
\draw (165,127.4) node [anchor=north west][inner sep=0.75pt]  [font=\normalsize,color={rgb, 255:red, 2; green, 2; blue, 208 }  ,opacity=1 ]  {$\frac{2}{3}$};
\draw (175,83.4) node [anchor=north west][inner sep=0.75pt]  [font=\normalsize,color={rgb, 255:red, 2; green, 2; blue, 208 }  ,opacity=1 ]  {$\frac{2}{3}$};
\draw (190,207.4) node [anchor=north west][inner sep=0.75pt]  [font=\normalsize,color={rgb, 255:red, 2; green, 2; blue, 208 }  ,opacity=1 ]  {$\frac{2}{3}$};
\draw (180,288.4) node [anchor=north west][inner sep=0.75pt]  [font=\normalsize,color={rgb, 255:red, 2; green, 2; blue, 208 }  ,opacity=1 ]  {$\frac{2}{3}$};
\draw (50,170.4) node [anchor=north west][inner sep=0.75pt]  [font=\normalsize,color={rgb, 255:red, 2; green, 2; blue, 208 }  ,opacity=1 ]  {$\frac{1}{3}$};
\draw (27,223.4) node [anchor=north west][inner sep=0.75pt]  [font=\normalsize,color={rgb, 255:red, 2; green, 2; blue, 208 }  ,opacity=1 ]  {$1$};

\end{tikzpicture}

%% file: main.bbl
\begin{thebibliography}{KKMSD73}

\bibitem[ACP15]{ACP}
Dan Abramovich, Lucia Caporaso, and Sam Payne.
\newblock The tropicalization of the moduli space of curves.
\newblock {\em Ann. Sci. Ec. Norm. Super.}, 48(4):765--809, 2015.

\bibitem[BHV01]{billeraetal}
Louis~J. Billera, Susan~P. Holmes, and Karen Vogtmann.
\newblock Geometry of the space of phylogenetic trees.
\newblock {\em Adv. in Appl. Math.}, 27(4):733--767, 2001.

\bibitem[Cap14]{Caporaso}
Lucia Caporaso.
\newblock Gonality of algebraic curves and graphs.
\newblock In {\em Algebraic and complex geometry}, volume~71 of {\em Springer
  Proc. Math. Stat.}, pages 77--108. Springer, Cham, 2014.

\bibitem[Car20]{TropicalComplexes}
Dustin Cartwright.
\newblock Tropical complexes.
\newblock {\em Manuscr. Math.}, 163(1-2):227--261, 2020.

\bibitem[CCUW17]{CCUW}
Renzo Cavalieri, Melody Chan, Martin Ulirsch, and Jonathan Wise.
\newblock A moduli stack of tropical curves.
\newblock 2017.

\bibitem[CGM22]{psi-classes}
Renzo Cavalieri, Andreas Gross, and Hannah Markwig.
\newblock Tropical {{\(\psi\)}} classes.
\newblock {\em Geom. Topol.}, 26(8):3421--3524, 2022.

\bibitem[CGP22]{ChanGalatiusPayne}
Melody Chan, S{\o}ren Galatius, and Sam Payne.
\newblock Topology of moduli spaces of tropical curves with marked points.
\newblock In {\em Facets of algebraic geometry. A collection in honor of
  William Fulton's 80th birthday. Volume 1}, pages 77--131. Cambridge:
  Cambridge University Press, 2022.

\bibitem[CMR16]{CMRadmissible}
Renzo Cavalieri, Hannah Markwig, and Dhruv Ranganathan.
\newblock Tropicalizing the space of admissible covers.
\newblock {\em Math. Ann.}, 364(3-4):1275--1313, 2016.

\bibitem[Fra13]{FrancoisCocycle}
Georges Francois.
\newblock Cocycles on tropical varieties via piecewise polynomials.
\newblock {\em Proc. Am. Math. Soc.}, 141(2):481--497, 2013.

\bibitem[FS05]{FeichtnerSturmfels}
Eva~Maria Feichtner and Bernd Sturmfels.
\newblock Matroid polytopes, nested sets and {Bergman} fans.
\newblock {\em Port. Math. (N.S.)}, 62(4):437--468, 2005.

\bibitem[GJR21]{GJR}
Walter Gubler, Philipp Jell, and Joe Rabinoff.
\newblock Forms on berkovich spaces based on harmonic tropicalizations.
\newblock Preprint: {\tt arXiv:2111.05741}, 2021.

\bibitem[HS22]{hs:logint}
David Holmes and Rosa Schwarz.
\newblock Logarithmic intersections of double ramification cycles.
\newblock {\em Algebr. Geom.}, 9(5):574--605, 2022.

\bibitem[Ion02]{ionel}
Eleny-Nicoleta Ionel.
\newblock Topological recursive relations in {$H^{2g}(\scr M_{g,n})$}.
\newblock {\em Invent. Math.}, 148(3):627--658, 2002.

\bibitem[JRS18]{Lefschetz}
Philipp Jell, Johannes Rau, and Kristin Shaw.
\newblock Lefschetz {$(1,1)$}-theorem in tropical geometry.
\newblock {\em \'{E}pijournal Geom. Alg\'{e}brique}, 2:Art. 11, 27, 2018.

\bibitem[Kaj08]{KajiwaraTropicalToric}
Takeshi Kajiwara.
\newblock Tropical toric geometry.
\newblock In {\em Toric topology}, volume 460 of {\em Contemp. Math.}, pages
  197--207. Amer. Math. Soc., Providence, RI, 2008.

\bibitem[Kat12]{Katz}
Eric Katz.
\newblock Tropical intersection theory from toric varieties.
\newblock {\em Collect. Math.}, 63(1):29--44, 2012.

\bibitem[KKMSD73]{toremb}
G.~Kempf, Finn~Faye Knudsen, D.~Mumford, and B.~Saint-Donat.
\newblock {\em Toroidal embeddings. {I}}, volume Vol. 339 of {\em Lecture Notes
  in Mathematics}.
\newblock Springer-Verlag, Berlin-New York, 1973.

\bibitem[Koc01]{k:pc}
Joachim Kock.
\newblock Notes on psi classes.
\newblock Notes. {\tt http://mat.uab.es/$\sim $kock/GW/notes/psi-notes.pdf},
  2001.

\bibitem[Mik07]{MikhalkinModuli}
Grigory Mikhalkin.
\newblock Moduli spaces of rational tropical curves.
\newblock In {\em Proceedings of {G}\"okova {G}eometry-{T}opology {C}onference
  2006}, pages 39--51. G\"okova Geometry/Topology Conference (GGT), G\"okova,
  2007.

\bibitem[MPS23]{mps:Hodge}
S.~Molcho, R.~Pandharipande, and J.~Schmitt.
\newblock The {H}odge bundle, the universal 0-section, and the log {C}how ring
  of the moduli space of curves.
\newblock {\em Compos. Math.}, 159(2):306--354, 2023.

\bibitem[MR24]{mr:case}
Sam Molcho and Dhruv Ranganathan.
\newblock A case study of intersections on blowups of the moduli of curves.
\newblock {\em Algebra Number Theory}, 18(10):1767--1816, 2024.

\bibitem[MS15]{MacStu}
Diane Maclagan and Bernd Sturmfels.
\newblock {\em Introduction to tropical geometry}, volume 161 of {\em Graduate
  Studies in Mathematics}.
\newblock American Mathematical Society, Providence, RI, 2015.

\bibitem[MZ08]{MZJacobians}
Grigory Mikhalkin and Ilia Zharkov.
\newblock Tropical curves, their {J}acobians and theta functions.
\newblock In {\em Curves and abelian varieties}, volume 465 of {\em Contemp.
  Math.}, pages 203--230. Amer. Math. Soc., Providence, RI, 2008.

\bibitem[Uli17]{Uli:Fun}
Martin Ulirsch.
\newblock Functorial tropicalization of logarithmic schemes: the case of
  constant coefficients.
\newblock {\em Proc. Lond. Math. Soc. (3)}, 114(6):1081--1113, 2017.

\bibitem[Uli19]{U:art}
Martin Ulirsch.
\newblock Non-archimedean geometry of artin fans.
\newblock {\em Advances in Mathematics}, 345:346--381, 2019.

\end{thebibliography}
